\documentclass[amsfonts]{amsart}
\usepackage{amssymb,amsfonts, tikz, tikz-cd}
\usepackage[all,arc]{xy}
\usepackage{enumitem}
\usepackage{mathrsfs}
\usepackage{microtype}
\usepackage{amsaddr}
\usepackage{amsthm,etoolbox}

\usepackage{hyperref}
\hypersetup{hidelinks, colorlinks}
\hypersetup{
    citecolor = {teal},
    linkcolor = {black},
}
\usepackage[capitalise]{cleveref}

\theoremstyle{plain}
\newtheorem{thm}{Theorem}[section]
\newtheorem{cor}[thm]{Corollary}
\newtheorem{prop}[thm]{Proposition}
\newtheorem{lem}[thm]{Lemma}

\theoremstyle{definition}
\newtheorem{defn}[thm]{Definition}

\newtheorem{assm}[thm]{Assumption}

\theoremstyle{remark}
\newtheorem{rem}[thm]{Remark}
\AtBeginEnvironment{rem}{%
    \small
  \pushQED{\qed}%
}
\AtEndEnvironment{rem}{\popQED}

\numberwithin{equation}{section}

\calclayout

\AtBeginDocument{%
   \def\MR#1{}
}

\newcommand{\sw}{{{w}_\circ}}
\newcommand{\sDelta}{{\Delta}_\circ}
\newcommand{\teq}{h}

\title[Flat endomorphisms of $\nabla^{eq}$ from Quantum Steenrod operations]{Flat endomorphisms for mod $p$ equivariant quantum connections from quantum Steenrod operations}

\author{Jae Hee Lee}
\email{jaehelee@mit.edu}

\begin{document}

\begin{abstract}
    We present a method for constructing covariantly constant endomorphisms for the mod $p$ equivariant quantum connection, using the quantum Steenrod power operations of Fukaya and Wilkins. The example of the cotangent bundle of the projective line is fully computed, and we discuss the relationship with the mod $p$ solutions of trigonometric KZ equation recently constructed by Varchenko. As a byproduct, we compute the first examples of quantum Steenrod operations that are not a priori determined by ordinary Gromov--Witten theory and classical Steenrod operations, which may be of independent interest.
\end{abstract}

\maketitle
\section{Background and Results}

\subsection{A brief overview of context}
The goal of this paper is to present a new character, the quantum Steenrod operations, in the study of the mod $p$ equivariant quantum cohomology of (conical) symplectic resolutions $X$, 
\begin{equation*}
    QH^*_{eq} (X; \mathbb{F}_p),
\end{equation*}
with coefficients over the finite field $\mathbb{F}_p$ of prime order $p$. The prefix \emph{equivariant} has multiple meanings in our context: we work simultaneously with respect to a (i) symmetry on the target manifold $X$ (a circle action scaling the holomorphic symplectic form), as well as a (ii) symmetry on the domain curve of the $J$-holomorphic maps.

The enumerative geometry of genus $0$ $J$-holomorphic curves in $X$ is captured by the quantum cohomology of $X$, equipped with a formal linear differential operator known as the \emph{(equivariant) quantum connection}. The quantum connection of symplectic resolutions is the main object of study in the equivariant mirror symmetry program, the vision of which was realized first for Springer resolutions in the pioneering work of \cite{BMO11}, later developing into the monumental \cite{MO19}. 

We are interested in new aspects of such equivariant mirror symmetry results that arise over coefficients of positive characteristic. In particular, we construct covariantly constant endomorphisms for the mod $p$ quantum connection, with an eye towards a better understanding of the recently constructed mod $p$ and $p$-adic solutions to the Gauss--Manin systems on the algebraic side of mirror symmetry (see \cite{SV19}, \cite{Var22}, \cite{SV-ar}). 

To illustrate our problem, let us focus on the example of $X = T^*\mathbb{P}^1$, the resolution of the $A_1$ surface singularity. The equivariant quantum connection is explicitly given by (for $p> 2$)
\begin{equation}\label{eqn:T*P1-quantum-connection}
    \nabla_a^{S^1} = t \partial_a - a \ \ast_{S^1} = t q\partial_q - \begin{pmatrix} 0 & \frac{\teq^2 q}{1-q} \\ 1 & \frac{-2\teq q}{1-q} \end{pmatrix}
\end{equation}
for the class $a \in H^2(T^*\mathbb{P}^1)$ Poincar\'e dual to the cotangent fiber, acting on $QH_{eq}^* (T^*\mathbb{P}^1) \cong H^*_{eq}(T^*\mathbb{P}^1)[\![q]\!]$ (an additive isomorphism). The formal parameters $\teq, t$ are the equivariant parameters for the symmetry of the target $X$ and the source curve $\mathbb{P}^1$, respectively.

For cohomology with coefficients in a field of characteristic zero, the construction of endomorphisms $\Sigma : QH^*_{eq}(X) \to QH^*_{eq}(X)$ which satisfy the covariant constancy relation
\begin{equation}
    \nabla_a^{S^1} \circ \Sigma - \Sigma \circ \nabla_a^{S^1} = 0
\end{equation}
is not difficult. The relation itself provides a recursive algorithm to determine $\Sigma$ order-by-order in the quantum parameters $q^A$, $A \in H_2(X;\mathbb{Z})$ from its ``classical term'' $\Sigma|_{q=0}$ (provided that $\Sigma|_{q=0}$ commutes with ordinary cup product, and $H_2(X;\mathbb{Z})$ is torsion-free). As an example, if we assume the classical term $\Sigma|_{q=0}$ to be the cup product with $a \in H^2(T^*\mathbb{P}^1)$, then the endomorphism would be of the form
\begin{align}\label{eqn:char-0-endomorphism}
    \Sigma &= \Sigma|_{q=0} + \Sigma^{(1)}q + \Sigma^{(2)}q^2 + \Sigma^{(3)}q^3 + \cdots, \\
    \Sigma^{(1)} q&= \begin{pmatrix} -t^{-1}\teq^2 & 0 \\ -2t^{-2}\teq^2 + 2t^{-1}\teq & t^{-1}\teq^2\end{pmatrix}q , \notag\\
        \Sigma^{(2)}q^2&= \begin{pmatrix} -\frac{3}{2}t^{-3}\teq^4 + t^{-2}\teq^3 - \frac{1}{2}t^{-1}\teq^2  & t^{-2}\teq^4 \\  -\frac{3}{2}t^{-4}\teq^4 + 3t^{-3}\teq^3 - \frac{5}{2}t^{-2}\teq^2 + t^{-1}\teq   & \frac{3}{2}t^{-3}\teq^4 - t^{-2}\teq^3 + \frac{1}{2}t^{-1}\teq^2  \end{pmatrix}q^2, \notag\\
            \Sigma^{(3)} q^3&= \begin{pmatrix} -\frac{5}{6}\frac{\teq^6}{t^5} + \frac{\teq^5}{t^4} - \frac{11}{6}\frac{\teq^4}{t^3} + \frac{\teq^3}{t^2} - \frac{1}{3}\frac{\teq^2}{t^1} & \frac{\teq^6}{t^4} + \frac{\teq^4}{t^2} \\ -\frac{5}{9}\frac{\teq^6}{t^6} + \frac{5}{3}\frac{\teq^5}{t^5} - \frac{29}{9}\frac{\teq^4}{t^4} + \frac{11}{3}\frac{\teq^3}{t^3} - \frac{20}{9}\frac{\teq^2}{t^2} + \frac{2}{3}\frac{\teq^1}{t^1} & \frac{5}{6}\frac{\teq^6}{t^5} - \frac{\teq^5}{t^4} + \frac{11}{6}\frac{\teq^4}{t^3} - \frac{\teq^3}{t^2} + \frac{1}{3}\frac{\teq^2}{t^1} \end{pmatrix} q^3, \dots \notag
\end{align}
Observe that the structure constants of the endomorphism involve denominators. Therefore, the covariantly constant endomorphism in characteristic $0$ coefficients does not in general admit reduction to mod $p$ coefficients. This property is consonant with the failure of the covariant constancy to determine the endomorphism in positive characteristic. That failure is, in turn, due to the existence of nontrivial quantum parameters that are annihilated by the derivative ($q \partial_q  (q^{p}) = p \cdot q^p = 0$), which halts the aforementioned recursive algorithm that depends on existence of antiderivatives of monomials $q^A$.

In this article, we propose that nevertheless a geometrically meaningful class of covariantly constant endomorphisms in mod $p$ coefficients can be obtained from a construction from equivariant symplectic Gromov--Witten theory, known as the \emph{quantum Steenrod operations}.

\subsection{Quantum Steenrod operations}
Recently in symplectic topology, there has been increasing interest (see e.g. \cite{CGG22}, \cite{She21}, \cite{Rez-ar}, \cite{Sei19}) in the class of equivariant operations on quantum cohomology known as the quantum Steenrod operations, introduced by Fukaya \cite{Fuk97} and developed by Wilkins \cite{Wil20}. These are linear endomorphisms (see \cref{defn:QSt-structure-constants})
\begin{equation}
    Q\Sigma_b : QH^*_{\mathbb{Z}/p}(X;\mathbb{F}_p) \to QH^*_{\mathbb{Z}/p}(X;\mathbb{F}_p)
\end{equation}
given for every $b \in H^*(X;\mathbb{F}_p)$, deforming the classical Steenrod operations on $H^*(X;\mathbb{F}_p)$ by $\mathbb{Z}/p$-equivariant counts of genus $0$ $J$-holomorphic curves. The equivariance is with respect to the action of cyclic group $\mathbb{Z}/p$ of prime order on the domain genus $0$ curve for the $J$-holomorphic maps. It is crucial that the definition of Gromov--Witten invariants in symplectic geometry (under assumptions on positivity) gives well-defined counts over $\mathbb{F}_p$ (in fact over the integers $\mathbb{Z}$), not over $\mathbb{Q}$.

An important property of the quantum Steenrod operations, proved by Seidel--Wilkins \cite{SW22}, is the covariant constancy of the operations with respect to the mod $p$ (small) quantum connection. This compatibility with the connection makes these invariants partially computable, in terms of ordinary Gromov--Witten theory and classical Steenrod operations. Yet, there has not been examples computed beyond the scope of covariant constancy, and therefore the \emph{full} determination of the quantum Steenrod operation as a covariantly constant endomorphism of the quantum connection has been restricted to very limited examples, often relying on degree arguments (see \cite[Section 6.1]{SW22}).
\subsection{Main Results}
Our starting point is the following mild generalization of the main theorem of \cite{SW22} to the $S^1$-equivariant setting.
\begin{prop}[\cref{prop:S1-cov-constancy}]\label{prop:S1-cov-constancy-intro}
    For any choice of $b \in H^*(X;\mathbb{F}_p)$, the $S^1$-equivariant quantum Steenrod operation $Q\Sigma_b^{S^1}$ is a covariantly constant endomorphism for the $S^1$-equivariant quantum connection, that is it satisfies
    \begin{equation}
        \nabla_a^{S^1} \circ Q\Sigma_b^{S^1} - Q\Sigma_b^{S^1} \circ \nabla_a^{S^1} = 0
    \end{equation}
    for any $a \in H^2_{S^1}(X;\mathbb{Z})$.
\end{prop}

The endomorphism $Q\Sigma_b^{S^1}$ arise as $S^1$-equivariant generalizations of the quantum Steenrod operations $Q\Sigma_b$ of \cite{Fuk97}, \cite{Wil20}, and \cite{SW22}. The $S^1$-equivariance is essential: the ordinary quantum Steenrod operations are defined from Gromov--Witten theory, but the Gromov--Witten invariants of a symplectic resolution $X$ are trivial because $X$ admits a deformation to an affine variety. Nevertheless, the \emph{$S^1$-equivariant} Gromov--Witten invariants of $X$ (considered with the $S^1$-action scaling the holomorphic symplectic form) can be nontrivial. Accordingly, the $S^1$-equivariant generalizations of the quantum Steenrod operations may be nontrivial. We introduce this $S^1$-equivariant generalization in this article, and provide its computations.

\cref{prop:S1-cov-constancy-intro} reduces the task of finding a flat endomorphism for the equivariant quantum connection to the Gromov--Witten theoretic problem of determining the structure constants of $Q\Sigma_b^{S^1}$. Such structure constants are given in terms of counts of solutions to a perturbed $J$-holomorphic curve equation, with the perturbation data being parametrized over the classifying space of the group $\mathbb{Z}/p \times S^1$.

The main result of our paper is the following determination of a covariantly constant endomorphism of geometric origin for $\nabla^{S^1}$ over $\mathbb{F}_p$-coefficients, obtained by enumerating such solutions to the partial differential equations.

\begin{thm}[\cref{cor:S1-QSt-computation}]\label{thm:S1-QSt-computation}
    Take $X = T^*\mathbb{P}^1$. The $S^1$-equivariant quantum Steenrod operation 
    \begin{equation}
        Q\Sigma_b^{S^1} : QH^*_{\mathbb{Z}/p \times S^1} (T^*\mathbb{P}^1;\mathbb{F}_p) \to QH^*_{\mathbb{Z}/p \times S^1} (T^*\mathbb{P}^1;\mathbb{F}_p),
    \end{equation} 
    gives a covariantly constant endomorphism for the equivariant quantum connection \eqref{eqn:T*P1-quantum-connection} of $T^*\mathbb{P}^1$. 
    
    Its classical term $Q\Sigma_b^{S^1}|_{q^{A\neq0}=0}$ agrees with the cup product with the classical Steenrod power $\mathrm{St}(b)$ of the class $b \in H^2_{S^1}(T^*\mathbb{P}^1;\mathbb{F}_p)$ that is Poincar\'e dual to a cotangent fiber. 
    
    This endomorphism can be fully computed to all degrees, in the sense that its structure constants are given in terms of explicitly computable equivariant Chern class integrals (see \cref{thm:qst-equals-Chern-integral}) over a projective space. Indeed, the result can be computed for any $p >2$ and up to any order in the $\mathbb{Z}/p$-equivariant parameters $t, \theta$ and the $S^1$-equivariant parameter $\teq$.
\end{thm}

In the following sense, the flat endomorphism $Q\Sigma_b^{S^1}$ that we construct is in fact the only nontrivial flat endomorphism for the mod $p$ quantum connection $\nabla^{S^1}$, unique up to multiplication by elements in $\mathbb{F}_p[\![q^p]\!]$.

\begin{thm}[{\cref{prop:uniqueness-flat-endo}}]\label{thm:S1-QSt-unique}
    The $\mathbb{F}_p[\![q^p]\!](\!(\teq, t)\!)$-module of flat endomorphisms for the mod $p$ $S^1$-equivariant quantum connection is free of rank $2$, generated by the identity endomorphism $\mathrm{id}$ and the ($S^1$-equivariant) quantum Steenrod operation $Q\Sigma_b^{S^1}$.
\end{thm}

Let us illustrate the computation for the Chern class integral from \cref{thm:S1-QSt-computation} that agrees with the structure constants for the operation $Q\Sigma_b^{S^1}$. Fix $b \in H^2_{S^1}(T^*\mathbb{P}^1)$ to be the class Poincar\'e dual to the cotangent fiber. The integral could be evaluated via fixed point localization as following (see \cref{cor:S1-QSt-computation})). Below, $d \in \mathbb{Z}_{>0}$ is written as $d = \alpha p + \beta$ so that $\alpha$ and $\beta$ are quotient and remainder of $d$ modulo $p$, respectively. First take
    \begin{equation}
        C_{d, \ell} := C_{d, \ell}(x, t, \teq) = (x^{p-1} - t^{p-1})^{1 - 2 \alpha} \ c_{0, \infty}(x + \ell t, t) \ \prod_{k=1}^{d-1} (x  + (\ell -k)t - \teq)^2.
    \end{equation}
For a power series $P(x) \in \mathbb{F}_p(\!(\teq, t, \theta)\!)[\![x]\!]$, let $(x^k: P(x)) \in \mathbb{F}_p(\!(\teq, t, \theta)\!)$ denote the coefficient of $x^k$ in $P(x)$. The localization formula (\cref{cor:S1-QSt-computation}) is then written as
\begin{align}
    &(Q\Sigma_b^{S^1}(b_0), b_\infty)_{S^1} = (\mathrm{St}(b) \cup b_0, b_\infty)_{S^1} \notag \\
     &+  \teq \  q^d  \sum_{d > 0}  \left[ \sum_{\ell=0}^\beta  \left( x^{2\alpha}:   C_{d, \ell} \ { \prod_{j=1}^\ell (x+jt)^{-2} \prod_{j=1}^{\beta-\ell} (x-jt)^{-2}  } \right) \right. \left. + \sum_{\ell = \beta + 1}^{p-1} \left( x^{2\alpha - 2} : C_{d, \ell} \ {  \prod_{j=0}^\beta (x+(\ell-j)t)^{-2} } \right)  \right]  \notag.
\end{align}
    The computation holds for any prescribed choice of $b_0, b_\infty \in H^*(X;\mathbb{F}_p)$, and involves a universal polynomial $c_{0, \infty}(H,t)$ that only depends on $b_0, b_\infty$ and the degree $d$.

    For example, in low degrees of $\teq$ we have (computation in \cref{sec:appendix-low-order-h})
    \begin{align}
        &(Q\Sigma_b^{S^1}(b_0), b_\infty)_{S^1} = (\mathrm{St}(b) \cup b_0, b_\infty)_{S^1}  \notag \\
         &+ \teq  \left( \sum_{d > 0, p \nmid d} \left[ -  \frac{t^{p-3}}{\beta^2}  c_{0, \infty}(\beta  t, t) \right] q^d + \sum_{d>0, p \mid d} \left[ -t^{p-1} (x^2:c_{0, \infty}(x, t)) \right] q^{d} \right)   \notag \\
         &+ \teq^2  \left( \sum_{d > 0, p \nmid d} \left[ \sum_{m=1, m \not \equiv \beta}^{d-1} \frac{ 2 t^{p-4}}{(\beta - m) \beta^2 } c_{0, \infty}(\beta t, t) - 2  \alpha t^{p-4} \left(x^1 : \left( \frac{t}{\beta^2} - \frac{2}{\beta^{3}}x \right)c_{0, \infty}(x+\beta t, t) \right) \right] q^d   \right. \notag \\
         &\quad \quad \quad  \left. \sum_{d >0, p | d} \left[ \sum_{m=1, m \not \equiv 0}^{d-1} \left( x^2 : -2 \ \frac{x^{p-1} - t^{p-1}}{x-mt} \ c_{0, \infty} (x, t) \right) -2 \alpha  \left( x^3 : (x^{p-1} - t^{p-1})\ c_{0, \infty} (x, t) \right) \right]  q^d \right) \notag \\
         &+ O(\teq^3). \notag
    \end{align}
    Note the appearance of higher powers of $\teq$; here we observe nonlinear dependence on the $S^1$-equivariant parameter $\teq$ in the computation of $Q\Sigma_b^{S^1}$. This is in contrast with \cite{BMO11}, where the formulas are only linear in $\teq$.

Finally, we describe the relationship of our operation and the mod $p$ solutions to the equivariant quantum differential equation obtained by Varchenko \cite{Var-ar}. Via a gauge transform prescribed by the theory of cohomological stable envelopes (\cite{TV14}, \cite{MO19}), the equivariant quantum differential equation \eqref{eqn:T*P1-quantum-connection} can be identified with the dynamical differential equation or the trigonometric KZ equation studied in \cite{Var-ar}. Following \cite{Var21}, let us call the mod $p$ solutions obtained by Varchenko the (mod $p$) \emph{arithmetic flat sections}.

Since the $S^1$-equivariant quantum Steenrod operation is flat with respect to the mod $p$ equivariant quantum connection, its action on the arithmetic flat sections yields another mod $p$ solution. This action can be examined explicitly:

\begin{prop}[\cref{prop:QSt-annihilates-soln}]\label{prop:Varchenko-intro}
    Fix a parameter $\mu \in \mathbb{F}_p$. Let 
    \begin{equation}
    I_\mu(q) = (-1)^{p-m} (1-q)^{2m} \left( \sum_{d=0}^{p-m} \binom{p-m-1}{p-m-d} \binom{p-m}{d} q^d , \  \sum_{d=0}^{p-m} \binom{p-m}{p-m-d} \binom{p-m-1}{d} q^d \right)
    \end{equation}
    from \eqref{eqn:arithmetic-soln} be the distinguished arithmetic flat section of the equivariant quantum connection obtained by \cite{Var-ar} for the given choice of the parameter $\mu \equiv m \in \mathbb{F}_p$. After specializing $\teq / t = \mu$, we have
    \begin{equation}
        Q\Sigma_b^{S^1}|_{\teq/t = \mu} \left(  I_\mu (q) \right) = 0.
    \end{equation}
\end{prop}

That is, we find that the distinguished solution of Varchenko is exactly annihilated by the $S^1$-equivariant quantum Steenrod operations (after a particular specialization on the ratio of the equivariant parameters for the group actions on the source and the target of the $J$-holomorphic curves, see \cref{rem:QSt-simplifies}).


\subsection{Related work}
The main motivation for considering the covariantly constant endomorphisms from \cref{thm:S1-QSt-computation} is two-fold:

On the algebraic side, Schechtman--Varchenko \cite{SV19} and Varchenko \cite{Var-ar}, \cite{Var22} have recently provided a method for systematically obtaining mod $p$ and $p$-adic solutions to the Gauss--Manin integrable systems that arise on the algebraic side of equivariant mirror symmetry. Their method, following an idea of \cite{Man61}, derives the ansatz for the mod $p$ solutions from the presentation of characteristic $0$ solutions as hypergeometric integrals. Our approach provides an alternative characterization of the mod $p$ solutions of \cite{Var-ar}: the quantum Steenrod operation exactly annihilates these solutions.

On the symplectic side, the significance of our results (\cref{thm:S1-QSt-computation} and \cref{thm:QSt-computation}) is that it provides the first computations of quantum Steenrod operations beyond the range of covariant constancy. In particular, unlike previously known computations, our computations are not a priori determined by ordinary Gromov--Witten theory (the small quantum product) and classical Steenrod operations. It also features nontrivial counts for the degrees supporting multiple covers of $J$-holomorphic spheres, which is consonant with the expectation that the multiple covers provide the most interesting information stored in quantum Steenrod operations (see \cite[Task 1]{Wil-sur}).

\subsection{Methods}
We give a brief discussion of the strategy for the computation in \cref{thm:S1-QSt-computation}.

By an early observation due to Atiyah \cite{Ati58} (the ``simultaneous resolution,'' \cref{lem:nbhd-of-T*P1}), the computation of the $S^1$-equivariant Gromov--Witten invariants of $X = T^*\mathbb{P}^1$ can be reduced to the computation of ordinary Gromov--Witten invariants of a different target geometry, namely that of a local $\mathbb{P}^1$ Calabi--Yau 3-fold (CY3) $Z = \mathrm{Tot}(\mathcal{O}(-1) \oplus \mathcal{O}(-1) \to \mathbb{P}^1)$. Similarly, our task of computing the $S^1$-equivariant quantum Steenrod operations of $T^*\mathbb{P}^1$ can be reduced to the computation of the quantum Steenrod operations $Q\Sigma_b$ and $Q\Sigma_b^{S^1}$ for the local $\mathbb{P}^1$ CY3, $\mathrm{Tot}(\mathcal{O}(-1)^{\oplus 2} \to \mathbb{P}^1)$.

\begin{rem}
The example of local $\mathbb{P}^1$ CY3, in a precise sense outlined by \cite{BG08} and \cite{BMO11}, can be regarded as the ``building block'' for computations for more interesting examples of target manifolds including resolutions of ADE surface singularities (beyond the $A_1$ case) or cotangent bundles of flag varieties (beyond the $\mathrm{SL}_2$ case). 
\end{rem}

\begin{thm}[\cref{thm:qst-equals-Chern-integral}, \cref{cor:computation}]\label{thm:QSt-computation}
    The quantum Steenrod operations $Q\Sigma_b$ for the local $\mathbb{P}^1$ CY3 $Z = \mathrm{Tot}(\mathcal{O}(-1) \oplus \mathcal{O}(-1) \to \mathbb{P}^1)$ can be fully computed to all degrees, by means of equivariant Chern class computations. The result is written as
    \begin{equation}
    Q\Sigma_b = \begin{pmatrix} (Q\Sigma_b (1), b) & (Q\Sigma_b (b), b) \\ (Q\Sigma_b (1 ), 1 )  & (Q\Sigma_b (b), 1) \end{pmatrix} = \begin{pmatrix} \sum_{d=1}^{\infty} -d^{p-2}t^{p-2} q^{d} & \sum_{d=1}^\infty -t^{p-1} q^{pd} \\ \sum_{d=1}^{\infty} -2d^{p-3}t^{p-3} q^{d} & \sum_{d=1}^{\infty} d^{p-2}t^{p-2} q^{d} \end{pmatrix}
\end{equation}
where $b \in H^2(Z)$ is the class Poincar\'e dual to the fiber of the projection $Z \to \mathbb{P}^1$.
\end{thm}

The covariant constancy relation of \cite{SW22} is a result of topological flavor, and to go beyond its scope for computations one must introduce a new geometric method. We adopt the strategy of \cite{Voi96} in her proof of the Aspinwall--Morrison multiple cover formula for Gromov--Witten theory of Calabi--Yau 3-folds. The key idea is to take a particular kind of perturbation datum for our perturbed Cauchy--Riemann equations, which we call \emph{decoupling perturbation datum}, of the form
\begin{equation}
    \nu \in C^\infty ( C \times \mathbb{P}^1 ; \mathrm{Hom}^{0,1 }(TC, \mathcal{O}(-1)^{\oplus 2})).
\end{equation}
More precisely, we consider a family of such perturbation data parametrized over the classifying space $B\mathbb{Z}/p$. The main advantage of restricting to the class of such perturbation data is that the corresponding equation admits a natural decomposition of the form
\begin{equation}
    u: C \to X, \ \overline{\partial}_J u = \pi^* \nu \quad \iff \quad \begin{cases} \mbox{(i) } v: C \to \mathbb{P}^1, \ \overline{\partial}_J v = 0 \\ \mbox{(ii) } \phi \in C^\infty(C, v^*\mathcal{O}(-1)^{\oplus 2}), \ \overline{\partial} \phi = \nu \end{cases}.
\end{equation}
Then we observe that the moduli space of solutions to equation (i) admits a natural compactification into a projective space $\mathcal{P}$, and the solutions to equation (ii) can be realized as a zero locus of a section of a vector bundle $\mathrm{Obs} \to \mathcal{P}$. In a loose analogy with the Kuranishi model perspective towards moduli spaces in symplectic geometry, we call $\mathcal{P}$ and $\mathrm{Obs}$ the \emph{thickened moduli space} and \emph{obstruction bundle} respectively. The key technical result (\cref{thm:qst-equals-Chern-integral}) is that despite the difference in the description of the moduli spaces, the counts of solutions to the original problem can be identified with the counts coming from the decomposition of the equation into (i) and (ii), essentially reducing them to (equivariant) Chern class computations. 

\subsection{Organization of the paper}

In Section 2 we introduce equivariant cohomology, equivariant moduli spaces of solutions to perturbed $J$-holomorphic curve equation. We define quantum Steenrod operations and also discuss their properties, notably including the covariant constancy relationship of \cite{SW22}.

In Section 3 we focus on the extended example of the local $\mathbb{P}^1$ Calabi--Yau 3-fold $X = \mathrm{Tot}(\mathcal{O}(-1)^{\oplus 2} \to \mathbb{P}^1)$ to describe our strategy for computing the quantum Steenrod operations. We develop the non-equivariant versions of the relevant moduli spaces. This is mainly an expanded exposition of \cite{Voi96} tailored to our conventions. 

In Section 4 equivariant versions of the moduli spaces from Section 3 are developed and the local $\mathbb{P}^1$ CY3 example is fully computed. 

In Section 5 the extension of the quantum Steenrod operations to the $S^1$-equivariant setting is discussed. We introduce moduli spaces which take both the $\mathbb{Z}/p$-symmetry on domain and the $S^1$-symmetry on target into account. The $S^1$-equivariant analogue of the covariant constancy relation is discussed.

In Section 6 we use the computations in Section 4 to determine the $S^1$-equivariant quantum Steenrod operations for the unique fiber class in $T^*\mathbb{P}^1$. By the results in Section 5, this yields the desired covariantly constant endomorphism for the equivariant quantum connection of $T^*\mathbb{P}^1$. 

In Section 7 the relationship of the $S^1$-equivariant quantum Steenrod operations with the arithmetic flat sections of the quantum differential equation obtained by Varchenko \cite{Var-ar} is examined. The appendix carries out the necessary gauge transform to rewrite our connection \eqref{eqn:T*P1-quantum-connection} to be compatible with the notation in \cite{Var-ar}.

\subsection{Acknowledgements}
We would like to thank Paul Seidel for his support and encouragement throughout the gestation of this project. We also thank Davesh Maulik for pointing out the inspiring work of Schechtman--Varchenko \cite{SV19}; Zihong Chen and Nick Wilkins for many insightful discussions regarding quantum Steenrod operations; Alexander Varchenko for explaining his series of work on the mod $p$ and $p$-adic solutions; Shaoyun Bai and Dan Pomerleano for helpful discussions at different stages of this project. Finally, we thank the anonymous referee for useful comments which greatly improved the exposition. 

This work was partially supported by MIT Landis fellowship, by the National Science Foundation through grant DMS-1904997, and by the Simons Foundation through grants 6552299 (the Simons Collaboration on Homological Mirror Symmetry) and 256290 (Simons Investigator).

\section{Quantum Steenrod operations}\label{sec:QSt}

In this section, we review the definition of quantum Steenrod operations and discuss a few of its properties. The discussion in this section is not particularly original. However, we have chosen to deviate at some points from the standard conventions for technical convenience. In particular, we impose incidence constraints at the marked points by intersecting with actual cycles, instead of using Morse cocycles as in \cite{SW22}. 

\subsection{$\mathbb{Z}/p$-equivariant cohomology}\label{ssec:eq-coh}
The purpose of this subsection is mainly to set notations, and the exposition closely follows \cite{SW22}. (But sign conventions differ, see \cref{rem:equiv-convention} and \cref{rem:equiv-convention-2}.) We will explain our choice of a model for the classifying space $B\mathbb{Z}/p$ and $\mathbb{Z}/p$-equivariant cohomology, which will be relevant for later constructions that are parametrized over $B\mathbb{Z}/p$.

Fix a prime $p$ and the cyclic group of $p$ elements $\mathbb{Z}/p$. Fix a distinguished cyclic generator $\sigma$, that is an isomorphism $\mathbb{Z}/p \cong \langle \sigma : \sigma^p = 1 \rangle$. The classifying space $B\mathbb{Z}/p$ is a homotopy type constructed via a quotient of a free $\mathbb{Z}/p$-action on a (weakly) contractible space $E\mathbb{Z}/p$. We take
\begin{equation}
S^\infty = \{ w = (w_0, w_1, \dots ) \in \mathbb{C}^\infty : \ w_k = 0 \mbox{ for } k \gg 0, \|w\|^2 = 1 \}
\end{equation}
as our model for $E\mathbb{Z}/p$. For $w \in S^\infty$, denote by $w: \mathrm{pt} \to S^\infty$ the corresponding inclusion map with image $\{w\}$.

Let $\zeta = e^{2\pi i/p}$, and then the (left) action of $\mathbb{Z}/p$ is described as
\begin{equation}
\sigma \cdot w  = \zeta w = (\zeta w_0, \zeta w_1, \dots).    
\end{equation}

Then we take $B\mathbb{Z}/p := S^\infty / (\mathbb{Z}/p)$ as the quotient.

One advantage of this model is that there is a convenient cellular description. Fix $k \ge 0$. Consider
\begin{align}
\label{eqn:equiv-param}
\Delta^{2k} &= \{ w \in S^\infty : w_k > 0, w_{k+1} = w_{k+2} = \cdots = 0 \}, \\
\Delta^{2k+1} &= \{ w \in S^\infty : |w_k| > 0, \arg(w_k) \in (0, 2\pi / p), w_{k+1} = w_{k+2} = \cdots = 0 \}.
\end{align}

Then $\Delta^{2k}, \Delta^{2k+1}$ are homeomorphic to open disks of dimensions $2k, (2k+1)$ respectively. The even-dimensional cells are oriented by considering the tangent space at the point $\{w_k = 1\} \in \Delta^{2k}$ and identifying it with $\mathbb{C}^{k} \subseteq \mathbb{C}^\infty$ with its orientation from the complex structure. Similarly, the odd-dimensional cells are oriented by considering the point $\{ w_k = e^{\pi i / p} \} \in \Delta^{2k+1}$ and identifying the tangent space with $\mathbb{C}^k \times i e^{\pi i / p} \mathbb{R}$, so that the direction that increases $\mathrm{arg}(w_k)$ is the positive direction. 

Natural compactifications $\overline{\Delta}^{2k}$, $\overline{\Delta}^{2k+1}$ are given by adding in boundaries
\begin{align}
\label{eqn:equiv-param-compact}
    \partial \overline{\Delta}^{2k} &= \{ w \in S^\infty : w_k = w_{k+1} = w_{k+2} =  \cdots = 0 \}, \\
    \partial \overline{\Delta}^{2k+1} &= \{ w \in S^\infty : |w_k| \ge 0, \mathrm{arg}(w_k) \in \{ 0, 2\pi / p\}, w_{k+1} = w_{k+2} = \cdots = 0\}.
\end{align}
Note that the boundaries $\partial \overline{\Delta}^{2k}$ and $\partial \overline{\Delta}^{2k+1}$ are stratified by equivariant cells $\Delta^i$ of lower dimensions. The top strata of $\partial \overline{\Delta}^{2k}$ are given by $\bigsqcup_{k=1}^{p} \sigma^k \Delta^{2k-1}$, and the top strata of $\partial \overline{\Delta}^{2k+1}$ are given by $\sigma \Delta^{2k} \sqcup (- \Delta^{2k})$, considering the orientation. 

The cells $\Delta^{i}, \sigma \Delta^i, \dots, \sigma^{p-1} \Delta^i$ for all $i$ together define a $\mathbb{Z}/p$-CW complex structure on $S^\infty$, which descends to a CW complex structure on the quotient $B\mathbb{Z}/p$. For this CW-structure, $\Delta^i$ can be thought of as the interior of the unique $i$-cell in $B\mathbb{Z}/p$. For $\mathbb{F}_p$-coefficients, $\overline{\Delta}^i$ become cycles in cellular homology, providing the additive basis for $H_*^{\mathbb{Z}/p}(\mathrm{pt};\mathbb{F}_p) := H_*(B\mathbb{Z}/p; \mathbb{F}_p)$.

For a space $M$ with a $\mathbb{Z}/p$-action, we define the (Borel) equivariant cohomology ring as
\begin{equation}\label{eqn:borel-eq-coh}
H^*_{\mathbb{Z}/p}(M; \mathbb{F}_p) := H^*(S^\infty \times_{\mathbb{Z}/p} M; \mathbb{F}_p),    
\end{equation}
where $S^\infty \times_{\mathbb{Z}/p} M$ denotes the quotient of $S^\infty \times M$ by the action $(\sigma w, m) = (w, \sigma m)$. Note that $H^*_{\mathbb{Z}/p}(M)$ is always an algebra over the equivariant ground ring $H^*_{\mathbb{Z}/p}(\mathrm{pt};\mathbb{F}_p) := H^*(B\mathbb{Z}/p ; \mathbb{F}_p)$, by the  functoriality for the map $M \to \mathrm{pt}$. To describe the equivariant ground ring in the cellular model, first fix the generators as follows.
\begin{itemize}
    \item For $p=2$, fix a generator $h \in H^1(B\mathbb{Z}/2;\mathbb{F}_2)$ such that $\langle h, \Delta_1 \rangle = 1$. 
    \item For $p>2$, fix generators $\theta \in H^1(B\mathbb{Z}/p;\mathbb{F}_p)$ and $t \in H^2(B\mathbb{Z}/p;\mathbb{F}_p)$ such that \begin{equation}\langle \theta, \Delta^1 \rangle =1 , \quad \langle t, \Delta^2 \rangle = 1.\end{equation}
\end{itemize}

From our perspective, it is most natural to consider $S^\infty \simeq E\mathbb{Z}/p$ and its quotient $B\mathbb{Z}/p$ as a direct limit of their finite dimensional approximations (sub-CW complexes), and hence their cohomology as an inverse limit. This yields
\begin{equation}
\label{eqn:equiv-coh-ground-ring}
    H^*_{\mathbb{Z}/2}(\mathrm{pt};\mathbb{F}_2) \cong \mathbb{F}_2[\![h]\!], \quad H^*_{\mathbb{Z}/p}(\mathrm{pt};\mathbb{F}_p) \cong \mathbb{F}_p[\![t, \theta]\!] \  ( p > 2). 
\end{equation}

Moreover note that $\langle \Delta^{2k}, t^k \rangle = \langle \Delta^{2k+1}, t^k \theta \rangle =1$. We will succinctly denote the equivariant parameters using the following notation:

\begin{equation}
\label{eqn:equiv-param-notation}
    (t,\theta)^i = \begin{cases} h^i & p=2 \\ t^{i/2} & p > 2, \ i \mbox{ even} \\ t^{(i-1)/2} \ \theta & p >2 , \ i \mbox{ odd} \end{cases}
\end{equation}
for each $i \in \mathbb{Z}_{\ge 0}$. 

\begin{rem}\label{rem:equiv-convention}
    Our choice of even degree generator for $p>2$ case, prescribed by $\langle t, \Delta^2 \rangle = 1$, is the opposite convention from that of \cite[Equation 2.11]{SW22}. Our choice amounts to that $t$ is the (mod $p$ reduction of) the first Chern class of the equivariant hyperplane bundle $\mathcal{O}(1) \to \mathbb{CP}^\infty$ pulled back under $B\mathbb{Z}/p = S^\infty/(\mathbb{Z}/p) \to S^\infty/\mathrm{U}(1) \cong \mathbb{CP}^\infty$. In contrast, \cite{SW22} uses the pullback of $\mathcal{O}(-1)$. Our choice seems to be more popular in the representation theory or equivariant Gromov--Witten theory literature, but adopting this convention means the formulae and computations from \cite{SW22} must be imported with the change of variables $t \leftrightarrow -t$. See also \cref{rem:equiv-convention-2}.
\end{rem}


\subsection{Non-equivariant moduli spaces}
We now describe the basic (non-equivariant) moduli spaces of perturbed solutions to the $J$-holomorphic equations. Later, the construction will be generalized equivariantly.

Denote $C = \mathbb{CP}^1 = \mathbb{C} \cup \{ \infty\}$, which is the curve that will be taken as the source of the (perturbed) $J$-holomorphic maps. Let $\zeta = e^{2\pi i/p}$, and take $(p+2)$ special marked points 
\begin{equation}\label{eqn:marked-points}
z_0 = 0, \ z_1 = \zeta, \ z_2 = \zeta^2, \dots, \ z_p = \zeta^{p} = 1, \ z_\infty = \infty
\end{equation}
on $C$. Equip $C$ with its standard almost complex structure $j_C$, coming from its description as gluing of two open discs $\mathbb{C}_{z} \cup \mathbb{C}_{w}$ along $z = 1/w$. Here, $0 \in \mathbb{C}_w$ is identified with $\infty \in C$.


Fix a connected symplectic manifold $(X, \omega)$ and a compatible almost complex structure $J$. Furthermore impose that the symplectic manifold satisfies the (spherical) positivity assumption 
\begin{assm}\label{assm:pos-condition}
\begin{equation}
\label{eqn:pos-condition}
\mbox{There exists } \lambda \ge 0 \mbox{ such that on } \mathrm{im}(\pi_2(X) \to H_2(X;\mathbb{Z})), \   c_1(TX, J)|_{\pi_2(X)} = \lambda \cdot  [\omega]|_{\pi_2(X)} .
\end{equation}
\end{assm}

Examples in practice that satisfy this condition arise from K\"ahler manifolds that have an ample anticanonical bundle or are Calabi--Yau. The examples considered in this article will in fact be K\"ahler and moreover always satisfy this positivity condition.

To cite standard theorems we will need to assume $X$ is closed, but a weakening of this assumption will be addressed in the next section where we discuss non-compact examples.

Consider the product $C \times X$. Over $C \times X$ there are pullbacks of the tangent bundles $TC$ and $TX$, which we will (abusively) denote by the same letters. By means of $j_C$ and $J$, $TC$ and $TX$ are complex vector bundles. Denote by $\mathrm{Hom}^{0,1}(TC, TX)$ the bundle of $J$-complex antilinear bundle homomorphisms from $TC$ to $TX$ over $C \times X$. A section of this bundle defines the perturbation datum:
\begin{equation}\label{noneq-nu}
\nu_X \in C^\infty \left(C \times X, \mathrm{Hom}^{0,1}(TC, TX)\right).
\end{equation}
We also require that $\nu_X$ vanishes near the marked points $\{z\} \subseteq C$.

The (non-equivariant) moduli space is defined as the solutions to the $J$-holomorphic equation perturbed by $\nu_X$. It admits a decomposition by the choice of a homology class $A \in H_2(X;\mathbb{Z})$, the degree for the curve.

To apply the standard transversality results for Gromov--Witten theory (see e.g. \cite[Chapter 3, Section 8.3]{MS12}), consider the target space $\widetilde{X} = C \times X$, the total space of the trivial fiber bundle over $C$ with fiber $X$. 

For the curves, we will use the following notation:
\begin{equation}\label{eqn:curves-convention}
    \widetilde{u}: C \to C \times X,  \quad u = \mathrm{proj}_X \circ \widetilde{u}: C \to X, \quad \mathrm{id}_C = \mathrm{proj}_C \circ \widetilde{u} : C \to C.
\end{equation}
We will always understand $\widetilde{u}$ as the graph of $u$, and in particular we will always assume the last expression (the graph condition) whenever the notation $\widetilde{u}: C \to C \times X$ is used, hoping that this will not cause confusion.

\begin{defn}\label{defn:noneq-map-moduli}
Fix $A \in H_2(X;\mathbb{Z})$. The \emph{(non-equivariant) map moduli space of degree $A$} denoted by $\mathcal{M}_A := \mathcal{M}_A(X;\nu_X)$ is the set
\begin{equation}
    \mathcal{M}_A(X;\nu_X) : = \left\{ (u: C \to X) : \ \overline{\partial}_J u = (\widetilde{u})^* \nu_X , \ u_*[C] = A \right\},
\end{equation}
where $[C]$ denotes the fundamental class of $C$.
\end{defn}

By standard transversality theory (which involves the ``Gromov trick'' of replacing solutions $u$ with their graphs $\widetilde{u}$), for a generic choice of $\nu_X$, $\mathcal{M}_A$ is a smooth (oriented) manifold of real dimension 
\begin{equation}
\dim_{\mathbb{R}} \mathcal{M}_A = \dim_{\mathbb{R}} X + 2 c_1(A)
\end{equation}
where $c_1(A) = \langle c_1(TX) , A \rangle$ is the Chern number of $A$. Using the marked points \eqref{eqn:marked-points}, it carries a natural evaluation map
\begin{align}
\mathrm{ev}: \mathcal{M}_A &\to X \times X^p \times X \\
u &\mapsto (u(z_0); u(z_1), \dots, u(z_p) ; u(z_\infty)).
\end{align}

The form of the standard theory (see e.g. \cite[{\S}6.7, 8.5]{MS12}) that applies to this evaluation map is the following.

\begin{prop}[Gromov Compactness]
\label{prop:gromov-generic}
Assume $X$ is closed. Under the positivity assumption (\cref{assm:pos-condition}) on $X$, the evaluation map is a pseudocycle for a generic choice of $\nu_X$. The bordism class of the pseudocycle is well-defined, independent of choices of $\nu_X$ as long as they are regular.

The limit set of the evaluation map is covered by evaluation maps from \emph{simple stable maps}, which constitute of evaluation maps from moduli of nodal curves obtained by (i) a graph in $C \times X$ satisfying the perturbed (inhomogeneous) equation (the \emph{principal component}), with (ii) trees of simple $J$-holomorphic spheres (the \emph{vertical bubbles}) attached along $\{z\} \times X$ for points $z \in C$. If the bubbles are attached at any of the distinguished marked points $\{ z_0, z_1, \dots, z_p, z_\infty\}$, the marked point is transferred to the bubble tree. 

The simple stable maps without any bubbles identify with $\mathcal{M}_A$. The simple stable maps with at least one bubble, denoted by $\partial \overline{\mathcal{M}}_A$, is a union of smooth manifolds of real dimension $\le \dim_{\mathbb{R}} \mathcal{M}_A -2$ for the generic choice of $\nu_X$.
\end{prop}

The identification of pseudocycle bordisms with singular homology is rigorously established in \cite{Zin08}. In particular, the bordism class of the evaluation pseudocycle defines a (singular) homology class in $H_*(X \times X^p \times X)$. The ordinary Gromov--Witten invariants can then be defined as the intersection product of the evaluation pseudocycle with another homology class that represents the incidence constraints on marked points. 

\begin{defn}\label{defn:noneq-cycle}
Let $Y_0, Y_1, \dots, Y_p, Y_\infty \subseteq X $ be oriented submanifolds which represent (possibly locally finite) cycles in $H_*(X)$ (or $H_*^{\mathrm{BM}}(X)$). The \emph{(non-equivariant) incidence cycle} is the data of the induced inclusion
\begin{equation}\label{eqn:noneq-cycle}
    \mathcal{Y}: Y_0 \times (Y_1 \times \cdots \times Y_p) \times Y_\infty \subseteq X \times X^p \times X.
\end{equation}
We also denote the underlying submanifold by $\mathcal{Y}$.
\end{defn}

\begin{defn}\label{defn:cycle-strongly-transverse}
Let $\mathcal{Y} \subseteq X \times X^p \times X$ be an incidence cycle (see \cref{defn:noneq-cycle}). Then $\mathcal{Y}$ is \emph{strongly transverse} to the evaluation pseudocycle if 
\begin{enumerate}[topsep=-2pt, label=(\roman*)]
\item it is transverse to the evaluation map from $\mathcal{M}_A$ and \item it is also transverse to the evaluation maps from the simple stable maps $\partial \overline{\mathcal{M}}_A$.
\end{enumerate}
\end{defn}

Note that since there are finitely many components of the simple stable maps $\partial \overline{\mathcal{M}}_A$, any incidence cycle $\mathcal{Y}$ can be perturbed (even as an incidence cycle, i.e. by perturbing each $Y_j \subseteq X$) to achieve strong transversality.


\begin{lem}
\label{lem:noneq-0-dim-count-map-moduli}
Assume that $\mathcal{Y}$ is strongly transverse to the evaluation pseudocycle. If $\mathrm{codim}_{X \times X^p \times X} \mathcal{Y} = \dim \mathcal{M}_A$, then the homological intersection number $\mathrm{ev} \cdot \mathcal{Y} \in \mathbb{Z}$ is equal to the (signed) count of the $0$-dimensional moduli space cut out by the incidence constraints given by $\mathcal{Y}$:
\begin{equation}
\# \ \mathrm{ev}(\mathcal{M}_A) \pitchfork \mathcal{Y} = \# \left\{ (u : C \to X) \in \mathcal{M}_A  : \ u(z_j) \in Y_j \mbox{ for } j \in \{ 0, 1, \dots, p, \infty \} \right\}.
\end{equation}
\end{lem}
\begin{proof}
By the strong transversality assumption (i) on $\mathcal{Y}$, indeed $\mathcal{M}_A \cap \mathrm{ev}^{-1}(\mathcal{Y})$ is a $0$-dimensional oriented submanifold of $\mathcal{M}_A$, whose image under $\mathrm{ev}$ has compact closure. We claim the count $\# \mathcal{M}_A \cap \mathrm{ev}^{-1}(\mathcal{Y})$ is finite. Suppose for a contradiction that there is no fixed compact subset $K \subseteq \mathcal{M}_A$ in which all the intersection happens. Then for any $K$ there exists $u_K \in (\mathcal{M}_A \setminus K) \cap \mathrm{ev}^{-1}(\mathcal{Y})$. Choose exhausting compact subsets $K_1 \subseteq K_2 \subseteq \cdots \subseteq \mathcal{M}_A$. The corresponding $\mathrm{ev}(u_{K_i})$ has a convergent subsequence, whose limit is contained in the intersection of $\mathcal{Y}$ with the limit set of $\mathrm{ev}$. This intersection is empty by the strong transversality assumption (ii) on $\mathcal{Y}$ for dimension reasons, so this is a contradiction.

The agreement of this count with the homological intersection number is part of the isomorphism between pseudocycle bordism and integral homology, see for example \cite[Lemma 6.5.5]{MS12}.
\end{proof}


\subsection{Equivariant moduli spaces}\label{ssec:eq-moduli-spaces}
The structure constants of quantum Steenrod operations are computed by looking at equivariant versions of the moduli spaces described in the previous section.

Adopting the usual approach in symplectic topology (see e.g. \cite{SS10}) towards equivariant operations, we implement the equivariant moduli spaces using the ``Borel model.'' Concretely, this means that we consider the parametrized moduli problem where the perturbation data are allowed to vary over the parameter space $B\mathbb{Z}/p$. 

Again $C = \mathbb{CP}^1$ is the marked curve, domain of the perturbed $J$-holomorphic curve equation. There is a natural $\mathbb{Z}/p$-action on $C$ given by rotation of angle $2\pi / p$ along the axis through $0$ and $\infty$, counterclockwise near $0$. Denote the action of the generator $\sigma \in \mathbb{Z}/p$ by $\sigma_C$. In terms of the distinguished marked points,
\begin{equation}\label{action-on-C}
\sigma_C : (C ; z_0, z_1, z_2, \dots, z_{p-1}, z_p, z_\infty) \mapsto (C ; z_0, z_2, z_3, \dots, z_{p}, z_1, z_\infty).
\end{equation}
The homotopy orbits (the Borel construction) $S^\infty \times_{\mathbb{Z}/p} (C \times X)$ is the quotient of $S^\infty \times (C \times X)$ by the relation $(\sigma w, (z, x)) = (w, \sigma(z,x))$  where $\mathbb{Z}/p$-action on $C\times X$ is given by $\sigma(z,x) = (\sigma_C z, x)$. 

\begin{defn}\label{defn:eq-perturbation-data}
An \emph{equivariant perturbation data} is a perturbation datum defined over the homotopy orbits,
\begin{equation}\label{eq-nu}
    \nu_X^{eq} \in C^\infty \left(S^\infty \times_{\mathbb{Z}/p} (C \times X) , \mathrm{Hom}^{0,1}(TC, TX) \right),
\end{equation}
where again $TC$ and $TX$ is understood as their pullbacks to $S^\infty \times_{\mathbb{Z}/p} (C \times X)$.
\end{defn}

We topologize $C^\infty(S^\infty \times (C \times X), \mathrm{Hom}^{0,1}(TC, TX))$ by exhausting $S^\infty = E\mathbb{Z}/p$ with its finite dimensional approximations, i.e. equip the coarsest topology that makes its restrictions to finite dimensional approximations continuous.

It is useful to consider equivariant perturbation data as a parametrized version of ordinary, non-equivariant perturbation datum. For every $w \in S^\infty$, suppose we choose a perturbation datum depending smoothly on $w$,
\begin{align}
\nu^{eq}_{X,w} \in C^\infty \left(C \times X, \mathrm{Hom}^{0,1}(TC, TX)\right)
\end{align}
satisfying the equivariant condition $\nu^{eq}_{X, \sigma \cdot w} = \sigma_C^* \ \nu^{eq}_{X,w}$. Such a family then assembles to define equivariant perturbation data in the sense of \eqref{eq-nu}, satisfying
\begin{align}
\nu^{eq}_{X,w} \cong (w \times \mathrm{id}_{C \times X})^* \nu_X^{eq}.
\end{align}
Indeed, this parametrized viewpoint is the one taken to establish the existence of equivariant perturbation data, by working inductively over the dimension of the skeleta of $B\mathbb{Z}/p$. (See \cite[Appendix C]{Wil20} for further discussion of the construction of equivariant data, or see \cref{lem:eq-transversality} for an explicit implementation of the same idea.)

Using equivariant perturbation data, we can now define the equivariant moduli spaces. Recall $S^\infty \simeq E\mathbb{Z}/p$ admits a cellular decomposition with $p$ many open $i$-cells $\Delta^i, \sigma \Delta^i, \dots, \sigma^{p-1} \Delta^i$ for every $i \ge 0$. We choose $\Delta^i$ as the parameter space for the equivariant moduli problem, considering this as the unique $i$-cell in the cellular description of $B\mathbb{Z}/p \simeq S^\infty / (\mathbb{Z}/p)$.

\begin{defn}
\label{defn:eq-map-moduli}
    Fix $A \in H_2(X;\mathbb{Z})$ and $i \ge 0$. The \emph{equivariant map moduli space of degree $A$ and equivariant degree $i$}, or the \emph{$i$th equivariant map moduli space of degree $A$} denoted by $\mathcal{M}_A^{eq, i} := \mathcal{M}_A^{eq,i}(X; \nu_X^{eq})$ is the set
    \begin{equation}
        \mathcal{M}_A^{eq, i} := \left\{ (w \in \Delta^i, u: C \to X) : \ \overline{\partial}_J u = (w \times \widetilde{u})^* \nu_X^{eq} , \ u_*[C] = A \right\}.
    \end{equation}
Other copies of moduli spaces where the condition $w \in \Delta^i$ is replaced by $w \in \sigma \Delta^i, \dots, w \in \sigma^{p-1} \Delta^i$ will be denoted $\sigma \mathcal{M}_A^{eq, i}, \dots, \sigma^{p-1} \mathcal{M}_A^{eq, i}$, respectively.
\end{defn}

We can choose $\nu_X^{eq}$ so that for each $i \ge 0$, this moduli space is regular of dimension $i + \dim_{\mathbb{R}} \mathcal{M}_A$, and moreover a generic choice of $\nu_X^{eq}$ is regular. We will postpone the discussion of the choice of regular equivariant perturbation data to later sections, where we will specialize to a particular type of equivariant perturbation data suitable for our applications. A generic choice of that particular subclass of equivariant perturbation data will ensure the regularity of the moduli spaces.

One can think of the $i$th equivariant moduli space as being obtained by restricting the equivariant parameter $w \in S^\infty$ to $w \in \Delta^i$ from the full equivariant moduli space
\begin{equation}
    \widetilde{\mathcal{M}}_A^{eq} := \left\{ (w \in S^\infty, u: C \to X) : \ \overline{\partial}_J u = (w \times \widetilde{u})^* \nu_X^{eq} , \ u_*[C] = A \right\}, \quad \mathcal{M}_A^{eq} := \widetilde{\mathcal{M}}^{eq}_A / (\mathbb{Z}/p),
\end{equation}
which contains copies of $\mathcal{M}^{eq, i}_A$ obtained from $w \in \Delta^i$ for each $i \ge 0$.

Just as in the non-equivariant case, the equivariant moduli space admits an evaluation map
\begin{align}
    \mathrm{ev}^{eq, i}: \mathcal{M}_A^{eq, i} &\to X \times \left(\overline{\Delta}^i \times  X^p\right) \times X \\
(w,u) &\mapsto (u(z_0); (w, u(z_1), \dots, u(z_p)) ; u(z_\infty)).
\end{align}
In the perspective of the full equivariant moduli space, it could be understood as the restriction of $\mathrm{ev}^{eq} : \widetilde{\mathcal{M}}^{eq}_A \to X \times (S^\infty \times X^p) \times X$. Note that we have chosen the target to be the compactification $X \times (\overline{\Delta}^i \times  X^p) \times X$, but the image is contained in $X \times (\Delta^i \times  X^p) \times X \subseteq X \times (\overline{\Delta}^i \times  X^p) \times X$.

Besides (parametrized) Gromov compactness, the equivariant moduli space exhibits non-compact behavior as we approach the boundary $w \in \Delta^i \to w \in \partial \overline{\Delta}^i$ in the equivariant parameter space. The limit points of $\mathrm{ev}^{eq, i}$ are therefore covered not only by the evaluation maps from parametrized analogues of simple stable maps $\partial \overline{\mathcal{M}}^{eq, i}_A$, but also by the evaluation maps from moduli spaces of lower equivariant degrees, namely $\mathcal{M}_A^{eq, j}, \sigma \mathcal{M}_A^{eq, j}, \dots, \sigma^{p-1} \mathcal{M}_A^{eq, j}$ for $j < i$ (and the corresponding analogues of simple stable maps).

The incidence constraints (analogue of \eqref{eqn:noneq-cycle}) must be modified to incorporate the symmetry of the target $X \times X^p \times X$. Denote by $\sigma_{X^p}$ the action of the generator $\sigma \in \mathbb{Z}/p$ which cyclically permutes the $X^p$ factor:
\begin{align}\label{eqn:action-on-X}
    \sigma_{X^p} : \ &X \times X^p \times X \to X \times X^p \times X , \\
    &(x_0; x_1, x_2, \dots, x_{p-1}, x_p; x_\infty) \mapsto (x_0; x_2, x_3, \dots, x_p, x_1 ; x_\infty).
\end{align}

To ensure the existence of an equivariant incidence cycle, we start from an incidence cycle such that the incidence constraints at the $p$ marked points along the equator must be the same. That is, assume that $Y_1 = \cdots = Y_p \subseteq X$ as embedded submanifolds. In the non-equivariant case, transverse intersection was achieved by perturbing the incidence cycle. For the equivariant case we require that this perturbation is parametrized over $S^\infty$ in an equivariant way:

\begin{defn}
\label{defn:eq-cycle}
    Fix a (non-equivariant) incidence cycle $\mathcal{Y} = Y_0 \times (Y_1 \times \cdots \times Y_p) \times Y_\infty \subseteq X \times X^p \times X$ as in \eqref{eqn:noneq-cycle}, such that $Y_1 = \cdots = Y_p$. An \emph{equivariant incidence cycle} corresponding to $\mathcal{Y}$ is a smooth map
    \begin{equation}
        \mathcal{Y}^{eq} : S^\infty \times \mathcal{Y} \to X \times (S^\infty \times X^p) \times X
    \end{equation}
    such that
    \begin{enumerate}[topsep=-2pt, label=(\roman*)]
    \item For each $w \in S^\infty$, the restriction $\mathcal{Y}^{eq}_w := \mathcal{Y}^{eq}|_{\{w \} \times \mathcal{Y}}$ satisfies $\mathcal{Y}^{eq}_w (\{w\} \times \mathcal{Y}) \subseteq X \times (\{w \} \times X^p) \times X$, hence $\mathcal{Y}^{eq}_w : \mathcal{Y} \to X \times X^p \times X$;
    \item For each $w \in S^\infty$, $\mathcal{Y}^{eq}_w : \mathcal{Y} \to X \times X^p \times X$ is a $\mathcal{Y}^{eq}_w$ a product of embeddings $Y_j \to X$ isotopic to the distinguished inclusion $Y_j \subseteq X$, and
    \item Finally, understood as maps from $\mathcal{Y}$, we have $\mathcal{Y}^{eq}_{\sigma \cdot w} = \sigma_{X^p} \circ \mathcal{Y}^{eq}_w$.
    \end{enumerate}
\end{defn}

The first two conditions (i), (ii) of \cref{defn:eq-cycle} characterizes $\mathcal{Y}^{eq}$ as a $S^\infty$-parametrized family of isotopic perturbations of $\mathcal{Y} \subseteq X \times X^p \times X$ \emph{as an incidence cycle}. Indeed, $Y_j (w) := \mathrm{proj}_j \circ \mathcal{Y}^{eq}_w$ is a $S^\infty$-dependent family of embedded submanifolds of $X$, all isotopic to $Y_j$. The last condition is the equivariance condition for this family. Note that (ii) together with (iii) can be achieved only if $Y_1 \cong \cdots \cong Y_p$ are all isotopic, and we start from the stronger assumption $Y_1 = \cdots = Y_p$ for the proof of existence.

Given an equivariant incidence cycle $\mathcal{Y}^{eq}$, denote by
\begin{equation}
    \mathcal{Y}^{eq, i} := \mathcal{Y}^{eq}|_{\Delta^i \times \mathcal{Y}} : \Delta^i \times \mathcal{Y} \to X \times (\overline{\Delta}^i \times X^p) \times X
\end{equation}
the restriction to the parameter space $\Delta^i \subseteq S^\infty$, and refer to this as $i$th equivariant incidence cycle.

\begin{defn}
\label{defn:eq-strong-transverse}
    A given equivariant incidence cycle $\mathcal{Y}^{eq}$ is \emph{equivariantly strongly transverse} to the collection of evaluation maps from $\mathcal{M}^{eq, i}_A$ if for every $i \ge 0$,
    \begin{enumerate}[topsep=-2pt, label=(\roman*)]
        \item $\mathcal{Y}^{eq, i}$ is transverse to the evaluation map from $\mathcal{M}_A^{eq, i}$, and
        \item $\mathcal{Y}^{eq, i}$ is also transverse to the evaluation maps from the ($\Delta^i$-parametrized versions of) simple stable maps $\partial \overline{\mathcal{M}}^{eq, i}_A$.
    \end{enumerate}    
\end{defn}

Establishing the existence of an equivariant incidence cycle that is equivariantly strongly transverse depends on the following lemma in differential topology. The basic claim underlying the lemma is that (parametric) transversality against $Y \times \cdots \times Y \subseteq X \times \cdots \times X$ can be achieved by considering it as a $p$-tuple of maps $Y \to X$ and perturbing each map.

For smooth maps $f_0, f_1: B \to X$, denote $f_0 \times f_1 : B \times B \to X \times X$ to be the map given by $(f_0\times f_1) (b_0, b_1) = (f(b_0), f(b_1))$.

\begin{lem}
\label{lem:eq-transversality}
    Suppose $\pi_A: A^{eq} \to S^\infty$ is an equivariant fibration such that for each $i \ge 0$, $A^{eq,i} = \pi_A^{-1}(\Delta^i)$ is a smooth manifold of dimension $i + \dim A^{eq, 0}$. 
    Let $X$ be a smooth manifold, and equip $X^p$ with a $\mathbb{Z}/p$-action as in \eqref{eqn:action-on-X}.  Assume there exists a smooth equivariant map $e^{eq}: A^{eq} \to S^\infty \times X^p$ such that $e^{eq}(\pi_A^{-1}(w)) \subseteq \{w\} \times X^p$. 
    
    Then for a fixed smooth manifold $B$ and a smooth map $f: B \to X$, there exists an equivariant map $f^{eq}: S^\infty \times B^p \to S^\infty \times X^p$ with $f^{eq}_w := f^{eq}|_{\{w\} \times B^p} \subseteq \{w\} \times X^p$ such that $f^{eq, i} : = f|_{\Delta^i \times B^p}$ is transverse to $e^{eq, i} := e|_{A^{eq, i}}$ for every $i \ge 0$. The map $f^{eq}$ extends $f$ in the sense that every $f^{eq}_w :  B^p \to X^p$ is in fact of the form $f_{1, w} \times \cdots \times f_{p, w}$, where each $f_{j, w}$ is a small perturbation of $f$.
\end{lem}

\begin{proof}
    Take $p >2$. We inductively construct $f^{eq}$ as follows. The base case is when $w =1 = \Delta^0 \in S^\infty$ is fixed. For $f: B \to X$, denote by $f^{\times p}: B^p \to X^p$ such that $f^{\times p}(b_1, \dots, b_p) = (f(b_1), \dots, f(b_p))$.  Choose a $C^\infty$-small neighborhood $\mathcal{S}$ of $f^{\times p}: B^p \to X^p$ in $C^\infty(B, X)^p$. Take a small open neighborhood $\mathcal{S} \subseteq C^\infty(B, X)^p$ of $f^{\times p}: B^p \to X^p$. By Sard's theorem, there exists a comeager subset $\mathcal{S}^{0, \mathrm{reg}} \subseteq \mathcal{S}$ such that for $g: B^p \to X^p \in \mathcal{S}^{0,\mathrm{reg}}$, the map $g$ is transverse to $e^{eq, 0}: A^{eq, 0} \to X^p$. Moreover, such $g$ is in $\mathcal{S}$ so $g = f_1 \times \cdots \times f_p$ for some $f_j : B \to X$. Choose $f^{eq, 0}$ to be any such $g$.

    The inductive assumption is that $f^{eq, j} : \Delta^j \times B^p \to \overline{\Delta}^j \times X^p$ has been chosen for all $j < i$, so that (i) $f^{eq, j}$ is transverse to $e^{eq, j}$, (ii) $f^{eq, j}_w (B^p) \subseteq \{w\} \times X^p$ and (iii) $f^{eq,j}_w \in \mathcal{S}$ for every $w \in \Delta^j$. Consider $\overline{\Delta}^i$ and its boundary $\partial \overline{\Delta}^i$. 

    (Case 1: $i$ is even.) If $i$ is even, the boundary $\partial \overline{\Delta}^i$ is a smooth manifold diffeomorphic to $S^{i-1}$. Fix a small collar neighborhood $U^i$ of $\partial \overline{\Delta}^i \subseteq \overline{\Delta}^i$ with a diffeomorphism $\phi: \partial \overline{\Delta}^i \times [0, \varepsilon) \cong U^i$. We extend the data defined over $\partial\overline{\Delta}^i$ to $U^i$ by constantly extrapolating along the interval $[0, \varepsilon)$. Explicitly, the extension $f_{\partial}^{eq, i} : U^i \times B^p \to \overline{\Delta}^i \times X^p$ is given by $(\phi(\sigma^k w, \delta), b) \mapsto (\sigma^k)^*f^{eq, i-1}(w, b)$ for $w \in \overline{\Delta}^{i-1}$. By inductive assumption (ii) and (iii), we can write $f^{eq, i}_{\partial} : U^i \to \mathcal{S} \subseteq C^\infty(B, X)^p$. Since $\mathcal{S}$ is contractible,  $f^{eq, i}_{\partial}$ (as a map from $U^i$) admits an extension to (not necessarily transverse) map $g^{eq, i}: \Delta^i \to C^\infty(B, X)^p$. Denoting $g^{eq, i}(w) = (g^{eq, i}_{1,w} : B \to X, \dots, g^{eq, i}_{p,w}: B \to X)$, the map $g^{eq, i}$ can be equivalently described as $g^{eq, i} = (g^{eq, i}_1 : \Delta^i \times B \to X, \dots, g^{eq, i}_p: \Delta^i \times B \to X) \in C^\infty(\Delta^i \times B, X)^p$ by adjunction. Moreover, it induces a map $g^{eq, i} = g^{eq, i}_1 \times \cdots \times g^{eq, i}_p : \Delta^i \times B^p \to X^p$ by the $\times$ operation (applied relatively to $\Delta^i$).

    By inductive assumption (i), $f_{\partial}^{eq, i}$ is transverse to (the $\mathbb{Z}/p$-orbit of) $e^{eq, i-1}$ along $\partial \overline{\Delta}^i \times B^p$. Since $f_\partial^{eq, i}$ is a constant extension from $\partial \overline{\Delta}^i$ to $U^i$, it follows by openness of transversality that $f_\partial^{eq, i}$ is also transverse to $e^{eq, i}$ (in the $\partial \overline{\Delta}^i$-parametrized sense) along constant $\delta$ slices in $U^i$. A fortiori, $f_{\partial}^{eq, i}$ is also transverse to $e^{eq, i}$ along $U^i$ (in the $\Delta^i$-parametrized sense). Now one can choose a perturbation $f^{eq, i}: \Delta^i \times B^p \to X^p$ of $g^{eq, i}: \Delta^i \times B^p \to X^p$ so that it is transversal to $e^{eq, i}$ and agrees with $f_\partial^{eq, i}$ on (a possibly smaller open subset of) $(U^i \cap \Delta^i) \times B^p$ (establishing (i)), and $f_w^{eq, i} \in \mathcal{S}$ (establishing (iii)). Equivalently $f^{eq,i}$ can be thought of as  a map $f^{eq, i}: \Delta^i \times B^p \to \overline{\Delta}^i \times X^p$ by sending $(w, \{b_j\})$ to $(w, f^{eq, i}(\{b_j\}))$ (establishing (ii)). Finally the extension to $\sigma \Delta^i, \dots, \sigma^{p-1} \Delta^i$ is determined by equivariance. This completes the inductive step. 

    (Case 2: $i$ is odd.) If $i$ is odd, the boundary $\partial \overline{\Delta}^i$  is a manifold with corners, with two faces glued along a corner stratum. It admits an explicit description as $\partial \overline{\Delta}^i \cong D^{i-1} \cup_{S^{i-2}} D^{i-1}$, where the corner stratum is given by $S^{i-2} \cong \bigcup_{j=0}^{i-2} (\Delta^j \cup \sigma \Delta^j \cup \cdots \cup \sigma^{p-1}\Delta^j)$. While $\partial \overline{\Delta}^i$ is not smooth, there is a collar neighborhood $U^i$ of $\partial \overline{\Delta}^i$ inside $\overline{\Delta}^i$, admitting a homeomorphism $\phi: \partial \overline{\Delta}^i \times [0, \varepsilon) \cong U^i$ which is a diffeomorphism along $\partial \overline{\Delta}^i \times (0, \varepsilon)$. Explicitly, let $i = 2k+1$ and consider $(w_0, \dots, w_k ; \delta) \in \partial \overline{\Delta}^i \times [0, \varepsilon)$. Since $w \in \partial \overline{\Delta}^{2k+1}$ (recall the description \eqref{eqn:equiv-param-compact}), $w_k \in \mathbb{C}$ lies in the intersection $\{ |z| \le 1 \} \cap \{ z = 
    s+it \in \mathbb{C} : f(s,t) := t ( \sin (2\pi/ p) s - \cos (2\pi / p) t ) = 0\}$. As we change the value of $\delta$, the intersection $\{|z|\le 1\} \cap \{ f=\delta\}$ traces out the locus of $w_k$ for $w \in U^i$. Now using $\phi$ we can constantly extend $f^{eq, i-1}$ to $U^i \times B^p$ as in the case where $i$ is even, and argue as before.

    The case $p =2$ follows the same argument, in fact an easier version of it. Namely, the even and odd cells can be treated uniformly in the same way as the even cells for $p >2$.
\end{proof}

\begin{lem}
\label{lem:eq-transverse-incidence-cycles}
    Given a sequence of regular equivariant moduli spaces $\{\mathcal{M}_A^{eq, i}\}_{i \ge 0}$, regular equivariant simple stable maps $\{\partial \mathcal{M}_A^{eq, i}\}_{i \ge 0}$ and a choice of incidence cycle $\mathcal{Y}$ such that $Y_1 = \cdots = Y_p$, there exists an equivariantly strongly transverse equivariant incidence cycle $\mathcal{Y}^{eq}$.
\end{lem}
\begin{proof}
    Let $A^{eq} := \bigcup_{i \ge 0} \bigcup_{k=1}^p \sigma^k \mathcal{M}_A^{eq,i}$, $(A')^{eq} := \bigcup_{i \ge 0} \bigcup_{k=1}^p \sigma^k \partial\overline{\mathcal{M}}_A^{eq,i}$, with $e^{eq}: A^{eq} \to X \times (S^\infty \times X^p) \times X$ and $(e')^{eq} : (A')^{eq} \to X \times (S^\infty \times X^p) \times X$ being the evaluation maps. Apply \cref{lem:eq-transversality} so that $f: \mathcal{Y} \subseteq X \times X^p \times X$ is the fixed incidence cycle $\mathcal{Y}$, and the lemma provides a choice of $\mathcal{Y}^{eq} : S^\infty \times \mathcal{Y} \to X \times (S^\infty \times X^p) \times X$ such that for each $i \ge 0$,  $\mathcal{Y}^{eq, i}$ is transverse to both evaluation maps $e^{eq, i}$, $(e')^{eq, i}$ from the equivariant moduli spaces and the simple stable maps.
\end{proof}

\begin{rem}\label{rem:equivariant-genericity-cycle}
    Moreover, the proof above shows that the choice of equivariantly strongly transverse $\mathcal{Y}^{eq}$ is generic in the following sense. For the fixed $\mathcal{Y}$, there is a constant inclusion of $\mathcal{Y}_0^{eq}: S^\infty \times \mathcal{Y} \to X \times (S^\infty\times X^p) \times X$ where its restriction to $\{w\} \times \mathcal{Y}$ for every $w$ is just the distinguished inclusion of $\mathcal{Y}$. (This automatically satisfies the equivariance condition.) Write $\mathcal{Y}_0^{eq, i}$ for its restriction to $\Delta^i \times \mathcal{Y}$. Then for a neighborhood $\mathcal{S}^i$ of $\mathcal{Y}^{eq, i}_0$ in $C^\infty(\Delta^i \times \mathcal{Y}, X \times (\overline{\Delta}^i \times X^p) \times X)$, there is a comeager subset $\mathcal{S}^{i, \mathrm{reg}}$ of cycles satisfying strong transversality in the $\Delta^i$-parametrized sense. By topologizing the space $C^\infty(S^\infty \times \mathcal{Y}, X \times (S^\infty \times X^p) \times X)$ with the coarsest topology where its restriction to the space of maps from finite dimensional cells $\Delta^i$ is continuous, we see that the (preimage of) $\mathcal{S}^{i, \mathrm{reg}}$ is a comeager subset of a neighborhood of $\mathcal{Y}^{eq}_0$. The intersection $\bigcap_{i \ge 0} \mathcal{S}^{i, \mathrm{reg}}$ then defines the comeager subset consisting of ``equivariantly strongly transverse perturbations of $\mathcal{Y}$''.
\end{rem}

We have the following analogue of \cref{lem:noneq-0-dim-count-map-moduli}. The presence of ``codimension 1'' limit points of the evaluation map prevent the counts from being well-defined as integers, but the count is still well-defined mod $p$:

\begin{lem}
\label{lem:eq-0-dim-count-map-moduli}
Assume that $\mathcal{Y}^{eq}$ is equivariantly strongly transverse to the evaluation maps. Fix the unique $i \ge 0$ such that $\dim_{\mathbb{R}} \mathcal{M}_A^{eq, i} = i + \dim_{\mathbb{R}} \mathcal{M}_A =  \mathrm{codim}_{X \times X^p \times X}  \mathcal{Y}$. Then the mod $p$ count of the $0$-dimensional moduli space cut out by the incidence constraints given by $\mathcal{Y}^{eq,i}$ is well-defined:
\begin{align}
\mathrm{ev}^{eq, i}(\mathcal{M}^{eq, i}_A) &\pitchfork \mathcal{Y}^{eq, i} = \\ &\# \left\{ (w \in \Delta^i, u : C \to X) \in \mathcal{M}_A^{eq, i}  : \ u(z_j) \in Y_j(w) \mbox{ for } j \in \{ 0, 1, \dots, p, \infty \} \right\} \in \mathbb{F}_p.
\end{align}
\end{lem}
\begin{proof}
The proof that the cut out moduli space is $0$-dimensional is almost identical to (the parametrized version of) \cref{lem:noneq-0-dim-count-map-moduli}. The only difference is that there are limit points of $\mathrm{ev}^{eq, i}$ covered by images from ($\mathbb{Z}/p$-orbits of) moduli spaces $\mathcal{M}^{eq, j}_A$ for $j < i$, and we must show that these cannot meet $\mathcal{Y}^{eq}$. But note that $\mathrm{codim}_{X \times X^p \times X} \mathcal{Y} = \mathrm{codim}_{X \times (\Delta^j \times X^p) \times X} \mathcal{Y}^{eq, j}$ for any $j$. Therefore, by the assumption $\mathrm{codim}_{X \times X^p \times X} \mathcal{Y} = \dim \mathcal{M}^{eq, i}_A = i + \dim \mathcal{M}_A$ combined with the equivariance of $\mathcal{Y}^{eq}$ and equivariant strong transversality, $\mathcal{Y}^{eq}$ avoids those limit points by dimension reasons.

To see that the count is independent of the choice of equivariantly strongly transverse incidence cycle $\mathcal{Y}^{eq}$, consider a $1$-parametric family of equivariant incidence cycles. This is a family $\mathcal{Y}^{eq}_{\lambda}$ (for $\lambda \in [0,1]$) interpolating two such choices $\mathcal{Y}^{eq}_0$ and $\mathcal{Y}^{eq}_1$, which is equivariantly strongly transverse in the $[0,1]$-parametrized sense. By dimension, $\mathcal{Y}^{eq}_{\lambda}$ may intersect limit points from $\mathcal{M}_A^{eq, i-1}$ in isolated points.  These intersections, however, always occur in $p$-fold families due to equivariance: any intersection $\mathrm{ev}^{eq, i-1}(\mathcal{M}_A^{eq, i-1}) \pitchfork \mathcal{Y}^{eq}_{\lambda}$ implies there are intersections with $\sigma \mathcal{M}_A^{eq, i-1}, \dots, \sigma^{p-1} \mathcal{M}_A^{eq, i-1}$ as well. The only way that the $1$-manifold $\mathrm{ev}^{eq, i}(\mathcal{M}_A^{eq, i}) \pitchfork \mathcal{Y}_{\lambda}^{eq}$ fails to yield a cobordism (in the usual sense) between the boundaries at the end $\mathrm{ev}^{eq, i}(\mathcal{M}_A^{eq, i}) \pitchfork \mathcal{Y}_{0}^{eq}$ and $\mathrm{ev}^{eq, i}(\mathcal{M}_A^{eq, i}) \pitchfork \mathcal{Y}_{1}^{eq}$ is exactly when such $p$-fold intersections occur in the compactification. Whenever this happens, the count of the $0$-dimensional moduli space only changes by a multiple of $p$, so the mod $p$ count is independent of the choice of $\mathcal{Y}^{eq}$. 
\end{proof}

\begin{rem}\label{rem:no-Fp-pseudocycle}
    Another approach to defining our counts would be to develop the theory of mod $p$ pseudocycles and their $\mathbb{F}_p$-valued intersection product, as considered in \cite[Appendix A]{Wil-sur}. Indeed, the equivariant evaluation maps $\mathrm{ev}^{eq, i}$ and $\mathcal{Y}^{eq, i}$ ought to define such mod $p$ pseudocycles (briefly speaking, as opposed to ordinary pseudocycles, the limit points are allowed to be covered by codimension $1$ strata as long as they are $p$-fold covers). But for defining the structure constants of the quantum Steenrod operations, we only need the counts from $0$-dimensional moduli spaces. This does not require the more general theory of mod $p$ pseudocycles, so we have chosen not to introduce it.
\end{rem}

\subsection{Quantum Steenrod operations}\label{ssec:QSt-defn}
Using the equivariant moduli spaces of perturbed solutions, we may now define the quantum Steenrod operations on $QH^*_{\mathbb{Z}/p} (X)$. We restrict ourselves to a much more constrained setting than \cite{SW22}, by considering incidence constraints that come from embedded submanifolds rather than general Morse cocycles. This suffices for our applications.

Let $(X, \omega)$ be a closed symplectic manifold satisfying the positivity condition from \cref{assm:pos-condition}. The cohomology of $X$ admits the Poincar\'e pairing $(\cdot, \cdot) : H^*(X;\mathbb{F}_p) \otimes H^*(X;\mathbb{F}_p) \to \mathbb{F}_p$ given by $(a, b ) = \int_X a \cup b$. The Poincar\'e pairing is nondegenerate, and for a fixed $b \in H^*(X;\mathbb{F}_p)$ we denote its linear dual under the Poincar\'e pairing by $b^\vee \in H^*(X;\mathbb{F}_p)$.

We fix a Novikov ring, which is a graded $\mathbb{F}_p$-algebra generated by symbols of the form $q^A$. We refer to symbols $q^A$ as the quantum parameters. The index $A$ ranges over those $A \in H_2(X;\mathbb{Z})$ that lie in the image of the Hurewicz homomorphism $\pi_2(X) \to H_2(X;\mathbb{Z})$ (although, in fact, examples in this article will be simply connected so we can let $A \in H_2(X;\mathbb{Z})$), for which $\int_A \omega_X \ge 0$. The ring is given by
\begin{equation}
    \label{novikov-ring}
    \Lambda = \Lambda_{\mathbb{F}_p} := \left\{ \lambda = \sum_{A} c_A q^A  : c_A \in \mathbb{F}_p \right\}, \quad |q^A| = 2c_1(A),
\end{equation}
such that for each given bound on $\int_A \omega_X$, for the degrees $A$ satisfying the bound, only finitely many of $c_A$ may be nonzero. In particular, the expression $\sum_A c_A q^A$ itself may involve an infinite sum.

The cohomology algebra with coefficients in the Novikov ring, $H^*(X;\Lambda)$, admits the (small) quantum product $\ast_q$. The product is defined using the three-pointed Gromov--Witten invariants and the Poincar\'e pairing: 
\begin{defn} The \emph{quantum product} $\ast = \ast_q$ is defined for $a, b \in H^*(X;\mathbb{F}_p)$ by means of the Poincar\'e pairing $( \cdot, \cdot ): H^*(X;\mathbb{F}_p) \otimes H^*(X;\mathbb{F}_p) \to \mathbb{F}_p$ as
\begin{equation}
    a \ast b = \sum_{A}  \langle a, b, c \rangle_A \ q^A \ c^\vee \in \Lambda
\end{equation}
where $\langle a, b, c \rangle_A$ denotes the three-pointed Gromov--Witten invariants of degree $A$ with incidence constraints given by Poincar\'e dual cycles of $a, b, c$. The sum on the right hand side is well-defined as element of $\Lambda$ by Gromov compactness. The product is extended $q$-linearly to $H^*(X;\Lambda)$, and defines an associative product on $H^*(X; \Lambda)$.
\end{defn}

\begin{defn}\label{defn:equivariant-QH}
    The \emph{$\mathbb{Z}/p$-equivariant quantum cohomology algebra} of $X$ is the ring
    \begin{equation}
        QH^*_{\mathbb{Z}/p} (X;\mathbb{F}_p) := \left( H^*(X;\Lambda)[\![t, \theta]\!], \  \ast \right)
    \end{equation}
    where $\ast = \ast_q$ denotes the quantum product, extended linearly for the equivariant parameters $t, \theta$.
\end{defn}
The product structure on the equivariant quantum cohomology algebra is extended only formally for the equivariant parameters, which is an artifact of that $\mathbb{Z}/p$ acts on $X$ trivially. Nevertheless, $QH^*_{\mathbb{Z}/p}(X;\mathbb{F}_p)$ admits new equivariant operations, analogous to the classical Steenrod operations on $\mathbb{F}_p$-coefficient cohomology algebra.

Take $QH^*_{\mathbb{Z}/p}(X; \mathbb{F}_p)$, and fix a cohomology class $b = \mathrm{PD}[Y] \in H^*(X;\mathbb{F}_p)$ that is (mod $p$ reduction of) Poincar\'e dual to a (locally finite) homology cycle represented by an inclusion of an embedded submanifold $Y \subseteq X$. Similarly, fix cohomology classes $b_0$, $b_\infty \in H^*(X)$ Poincar\'e dual to $Y_0$, $Y_\infty \subseteq X$. 

Fix the corresponding incidence cycle $\mathcal{Y} = \mathcal{Y}(b_0; b; b_\infty) = Y_0 \times (Y \times \cdots \times Y) \times Y_\infty \subseteq X \times X^p \times X$. Assume that equivariant perturbation data $\nu_X^{eq}$ and equivariant incidence cycle $\mathcal{Y}^{eq}$ are chosen so that the count $ \mathrm{ev}^{eq, i} (\mathcal{M}_A^{eq, i}) \pitchfork \mathcal{Y}^{eq, i} \in \mathbb{F}_p$ is defined. Recall from \eqref{eqn:equiv-param-notation} that $(t,\theta)^i$ denotes the $i$th equivariant parameter, so that $(t,\theta)^i = t^{i/2}$ or $(t,\theta)^i = t^{(i-1)/2} \theta$ depending on the parity of $i$.

\begin{defn}\label{defn:QSt-structure-constants}
    For $b \in H^*(X;\mathbb{F}_p)$, the \emph{quantum Steenrod cap product} is a map $Q\Sigma_b : QH^*_{\mathbb{Z}/p}(X;\mathbb{F}_p) \to QH^*_{\mathbb{Z}/p}(X;\mathbb{F}_p)$ defined for $b_0 \in H^*(X;\mathbb{F}_p)$ as the sum 
    \begin{equation}
        Q\Sigma_b (b_0) = \sum_{b_\infty} \sum_A \sum_i (-1)^\star \left( \# \ \mathrm{ev}^{eq, i} (\mathcal{M}_A^{eq, i}) \pitchfork \mathcal{Y}^{eq, i} \right)  q^A \ (t,\theta)^i \  b_\infty^\vee  \in H^*(X;\Lambda)[\![ t, \theta]\!].
    \end{equation}
    The first sum is over a chosen basis of $H^*(X;\mathbb{F}_p)$ as a vector space over $\mathbb{F}_p$, and $b_\infty^\vee \in H^*(X;\mathbb{F}_p)$ denotes the linear dual of $b_\infty \in H^*(X;\mathbb{F}_p)$ under the (nondegenerate) Poincar\'e pairing. Here the counts appearing as structure constants are defined to be zero if the moduli space is not discrete, i.e.
    
    \begin{align}
    \# \ \mathrm{ev}^{eq, i} (\mathcal{M}_A^{eq, i}) \pitchfork \mathcal{Y}^{eq, i} = \begin{cases} \#  \ \mathrm{ev}^{eq, i} (\mathcal{M}_A^{eq, i}) \pitchfork \mathcal{Y}^{eq, i} \in \mathbb{F}_p & \mbox{if } i + \dim_{\mathbb{R}} X + 2c_1(A) = |b_0|+p|b|+|b_\infty|  \\ 0 & \mbox{otherwise} \end{cases}.
    \end{align}
    
    Then it is extended linearly in the $q$-parameters and $(t, \theta)$-parameters from $H^*(X;\mathbb{F}_p)$ to $QH^*_{\mathbb{Z}/p}(X;\mathbb{F}_p) = H^*(X;\Lambda)[\![t,\theta]\!]$. The sign is given by
    \begin{equation}\label{eqn:sign-in-QSt}
        \star = \begin{cases} |b||b_0| & \mbox{ if } i \mbox{ even } \\ |b||b_0| + |b|+|b_0| & \mbox{ if } i \mbox{ odd } \end{cases}.
    \end{equation}
    By the dimension constraint, $Q\Sigma_b: QH^*_{\mathbb{Z}/p}(X;\mathbb{F}_p) \to QH^*_{\mathbb{Z}/p}(X;\mathbb{F}_p)$, is a map of graded vector spaces of degree $p|b|$. 
\end{defn}

\begin{rem}[Compare {\cite[Remark 4.2]{SW22}}]\label{rem:equiv-convention-2}
    The sign convention for parametrized map moduli spaces, where the parameter space is ${\Delta}$ a manifold with boundary, is chosen as follows. If $\phi_{{\Delta}}$ is associated to the parametrized problem, then $ \phi_{\partial \Delta} = \phi_\Delta d - (-1)^{|\Delta|} d \phi_\Delta$. This differs from the convention of \cite{SW22}, consistent with the fact that our choice of the equivariant parameter $t$ differs from that of \cite{SW22}. As a result, the sign $(-1)^\star$ does not carry the explicit factor of $(-1)^{\lfloor |\Delta|/2\rfloor}$.
\end{rem}

In the language of \cite{Sei19}, The operation $Q\Sigma_b$ is the ``cap product analogue'' of quantum Steenrod operations, and it satisfies $Q\Sigma_b(1) = Q\mathrm{St}(b)$ where $Q\mathrm{St}$ denotes the total quantum Steenrod power operations of \cite{Wil20}. We will loosely refer to $Q\Sigma_b$ as the ``quantum Steenrod operations'' in this article.

We record some of the key properties of the quantum Steenrod operations, established in \cite{Wil20} and \cite{SW22}.

\begin{prop}\label{prop:QSt-properties}
    For $b \in H^*(X;\mathbb{F}_p)$, the operations $Q\Sigma_b$ satisfy the following properties.
    \begin{itemize}
        \item $Q\Sigma_{1} = \mathrm{id}$ for $1 \in H^0(X;\mathbb{F}_p)$ the unit of cohomology algebra.
        \item $Q\Sigma_b(c)|_{t = \theta = 0} = \overbrace{b \ast \cdots \ast b}^{p} \ast c$, hence $Q\Sigma_b$ deforms $p$-fold quantum product in equivariant parameters.
        \item $Q\Sigma_b(c)|_{q^{A \neq 0}=0} = \mathrm{St}(b) \cup c$, hence quantum Steenrod operations deform classical Steenrod operations in quantum parameters.
        \item $Q\Sigma_b(1) = Q\mathrm{St}(b)$, hence $Q\Sigma_b$ with $1 \in H^0(X;\mathbb{F}_p)$ as an input specializes to quantum Steenrod powers of \cite{Wil20}.
    \end{itemize}
\end{prop}

\subsection{Covariant constancy of quantum Steenrod operations}\label{ssec:cov-constancy}
Our goal is to compute the structure constants $(Q\Sigma_b (b_0) , b_\infty) \in \Lambda [\![t, \theta]\!]$. The most effective computational tool available is the covariant constancy of $Q\Sigma_b$ with respect to the quantum connection, established in \cite{SW22}, which we now introduce.

Fix an element $a \in H^2(X;\mathbb{Z})$. Define the corresponding differentiation operator for the Novikov ring, $\partial_a : \Lambda \to \Lambda$, which acts on symbols $q^A$ by $\partial_a q^A = (a \cdot A)q^A$. Extend this linearly in parameters $t, \theta$ to define an endomorphism $\partial_a : \Lambda[\![t, \theta]\!] \to \Lambda[\![t, \theta]\!]$.

\begin{defn}\label{defn:quantum-connection}
    The \emph{(small) quantum connection} of $X$ is the collection of endomorphisms $\nabla_a : QH^*_{\mathbb{Z}/p}(X;\mathbb{F}_p) \to QH^*_{\mathbb{Z}/p}(X;\mathbb{F}_p)$ for each $a \in H^2(X;\mathbb{Z})$ defined by
    \begin{equation}\label{eqn:quantum-connection}
        \nabla_a \beta = t \partial_a \beta - a \ast \beta
    \end{equation}
    where $a \ast $ denotes the quantum product with $a \in H^2(X;\mathbb{Z})$. 
\end{defn}

The main theorem of \cite{SW22} is the following:

\begin{thm}[{\cite[Theorem 1.4]{SW22}}]\label{thm:covariant-constancy}
    For any choice of $b \in H^*(X;\mathbb{F}_p)$, the quantum Steenrod operation $Q\Sigma_b$ is a covariantly constant endomorphism for the quantum connection, that is it satisfies
    \begin{equation}\label{eqn:covariant-constancy}
        \nabla_a \circ Q\Sigma_b - Q\Sigma_b \circ \nabla_a = 0
    \end{equation}
    for any $a \in H^2(X;\mathbb{Z})$.
\end{thm}
\begin{proof}[Proof sketch]
    Since we will generalize this result in our later discussion of the $S^1$-equivariant quantum Steenrod operations, we reproduce below the main idea of the proof for the convenience of the reader.
    
    The idea is to consider the moduli space $\mathcal{M}_A(a)$ of maps from our fixed curve $C$ equipped with one extra marked point, which is allowed to freely move around the domain $C$ (and bubbles off a nodal curve when it collides with one of the special points $z_0, z_1, \dots, z_p, z_\infty \in C$). The new free marked point carries an incidence constraint that it should pass through the Poincar\'e dual of $a$. One can define the equivariant analogues $\mathcal{M}^{eq, i}_A(a)$ accordingly, using equivariant perturbation data, and take the transversal intersection with a generically chosen incidence cycle $\mathcal{Y}^{eq}$.

    Using the counts from $\mathrm{ev}^{eq, i} (\mathcal{M}^{eq, i}_A(a)) \pitchfork \mathcal{Y}^{eq}$ as the structure constants, one may define a new operation denoted $Q\Pi_{a, b}$ in \cite{SW22}. Observe that each degree $A$ solution $u \in \mathcal{M}_A^{eq, i}$ contributing to a structure constant in $Q\Sigma_b$ exactly contributes $(a \cdot A)$ to the corresponding structure constant for $Q\Pi_{a,b}$: one for every solution where the free marked point hits one of the $(a \cdot A)$ many preimages of $u^{-1}(\mathrm{PD}[a])$. Therefore, the operation $Q\Pi_{a, b}$ may be identified with $\partial_a Q\Sigma_b$.

    Now observe that there are two special strata of the moduli space $\mathcal{M}_A^{eq, i}(a)$ where the free marked point collides with $z_0 = 0$ or $z_\infty = \infty$ in $C$. Consider the corresponding operations using the counts of solutions in such strata as the structure constants. These operations can then be identified (by a TQFT-type argument) with $Q\Sigma_b ( a \ast - ) $ and $a \ast Q\Sigma_b(-)$, respectively. 
    
    The key insight in \cite{Wil20} and \cite{SW22} is that the two points $z_0, z_\infty \in C$ are fixed points of the $\mathbb{Z}/p$-action $\sigma_C$ on $C$, and one can ``lift'' a fixed point localization result for $C$ (the parameter space for the free marked point) to the level of moduli spaces. Namely, one fixes (\cite[Section 2c]{SW22}) a $\mathbb{Z}/p$-equivariant cellular chain complex of $C$ in which the relationship $t[C] = [z_\infty] - [z_0]$ holds (where $[C]$ is the fundamental chain), and exploits this relationship to obtain the relation
    \begin{equation}
        t \partial_a Q\Sigma_b(-) = a \ast Q\Sigma_b(-) - Q\Sigma_b(a*-).
    \end{equation}
    This equality is indeed equivalent to the desired result \eqref{eqn:covariant-constancy}.
\end{proof}

If the underlying coefficient ring for the Novikov ring $\Lambda$ is characteristic zero and moreover $H_2(X;\mathbb{Z})$ is torsion-free, covariant constancy would imply that $Q\Sigma_b$ is determined from $\nabla_a$ and its classical part $Q\Sigma_b|_{q^{A \neq 0} = 0} = \mathrm{St}(b) \cup$. (The recursive strategy allowing this determination is explained in \cite[Section 5a]{Wil-sur}.) In characteristic $p$, this is no longer true due to the existence of nontrivial elements in the Novikov ring $\Lambda$ annihilated by $\partial_a$: note that $\partial_a q^{pA} = (a \cdot pA) q^{pA} = 0$. In practice, covariant constancy of the quantum Steenrod operations only determines the contribution of rational curves of low degrees in the expansion of $Q\Sigma_b$ in the quantum parameters $q^A$. We remark that every previously known computation of the operations $Q\Sigma_b$ relies on covariant constancy (or its corollary, the quantum Cartan relation of \cite{Wil20}).

By covariant constancy, a full computation of $Q\Sigma_b$ yields a construction of a covariantly constant endomorphism for the quantum connection. Later, we will introduce the analogue of this fact for the $S^1$-equivariant quantum connection of $X$, and use it to determine new covariantly constant endomorphisms for the $S^1$-equivariant quantum connection.

\section{The local $\mathbb{P}^1$ Calabi--Yau 3-fold (Non-equivariant)}\label{sec:noneq-localP1}

The aim of this section and the following one is to provide a discussion of an extended example for the method we develop for computing quantum Steenrod operations for certain negative bundles over smooth projective varieties. The relevant geometry of the example is that of a local genus $0$ curve Calabi--Yau $3$-fold, which is the $X = \mathrm{Tot}((\mathcal{O}(-1) \oplus \mathcal{O}(-1)) \to \mathbb{P}^1)$ the total space of the vector bundle $\mathcal{O}(-1)^{\oplus 2} \to \mathbb{P}^1$. 

The enumerative geometry of this example is extensively studied, going back to the early papers in Gromov--Witten theory. In particular it is the subject of the celebrated multiple cover formula of Aspinwall--Morrison \cite{AM93} (see also \cite{Voi96}). For our purposes, it will importantly serve the role of a building block of later computations we carry out for $T^*\mathbb{P}^1$.

We will describe how to compute the quantum Steenrod operations for the class $b \in H^2(X)$ that is Poincar\'e dual to the fiber. The main strategy is to adopt a special type of perturbation data, which we refer to as \emph{decoupling perturbation data}. Using decoupling perturbation data, we may identify the count of the moduli space with integrals of (equivariant) Chern classes over a projective space. This strategy can be understood as an equivariant analogue of the approach in \cite{Voi96}. 

In this section, we will introduce the geometric setup, and develop \emph{non-equivariant} versions for the relevant constructions. It is largely an exposition of existing material in \cite{AM93} and \cite{Voi96}, but tailored for our purposes.

\subsection{Geometric setup}

For the target symplectic manifold, for this discussion, we restrict ourselves to the geometry of a local rational curve in a Calabi--Yau 3-fold: let $X = \mathrm{Tot}( (\mathcal{O}(-1) \oplus \mathcal{O}(-1)) \to \mathbb{P}^1)$. Fix its integrable complex structure $J$ on $X$. Using the inclusion of the zero section $\mathbb{P}^1 \to X$ we obtain a canonical splitting $TX = T\mathbb{P}^1 \oplus \pi^* (\mathcal{O}(-1) \oplus \mathcal{O}(-1))$ over the zero section $\mathbb{P}^1$. We denote the normal summand as $N=\pi^* (\mathcal{O}(-1) \oplus \mathcal{O}(-1))$, which fits into the exact sequence
\[
0 \to N \to TX \to TX/N \to 0
\]
of complex vector bundles over $X$.

Note that $c_1(TX) = 0$, so all Chern numbers of holomorphic curves inside $X$ are zero. In particular, the index of the $J$-holomorphic curve equation is independent of the degree of the curve. (This is why this example is the subject of the multiple cover formula in \cite{AM93}.)

Since $X$ is non-compact, the discussion of the moduli spaces involved and quantum Steenrod operations (which were defined for \emph{closed} manifolds) must be modified accordingly. The required modifications will be discussed below as needed.

\subsection{Decoupling perturbation data}

Following the strategy of \cite{Voi96}, we choose a special type of perturbation data for the perturbed, inhomogeneous equations. The particular choice will decouple the perturbed equation to a system of two equations, which respectively model the thickened moduli space and the obstruction bundle over it. This idea is somewhat orthogonal to issues with equivariance, so we have chosen to describe the non-equivariant version first. Later we upgrade this to the equivariant setting. 

For this consider the following version of the perturbation data:

\begin{defn}\label{defn:noneq-nu-dc}
The \emph{decoupling perturbation data} is a section
\begin{equation}\label{noneq-nu-dc}
\nu  \in C^\infty \left(C \times \mathbb{P}^1, \mathrm{Hom}^{0,1 }(TC, N)\right).    
\end{equation}
\end{defn}

The main difference from earlier \eqref{noneq-nu} is that the perturbation data is supported on the base $\mathbb{P}^1 \subseteq X$, but takes values in the subbundle $\mathrm{Hom}^{0,1} (TC, N) \hookrightarrow \mathrm{Hom}^{0,1} (TC, TX)$.

To see this is a special type of a perturbation data introduced earlier, denote by $\pi^* \nu$ the pullback of $\nu$ by the map $(C \times X) \to (C \times \mathbb{P}^1) $ induced by the projection $\pi: X \to \mathbb{P}^1$. By the inclusion $N \hookrightarrow TX$, we can think
\begin{equation}
\pi^*\nu \in C^\infty \left(C \times X, \mathrm{Hom}^{0,1} (TC, TX)\right)
\end{equation}
as a choice of a perturbation data $\nu_X = \pi^* \nu$ in the previous sense. 

The choice of decoupling perturbation data effectively decouples the perturbed equation. For the curves, extend the previous notation \eqref{eqn:curves-convention} as follows:
\begin{align}
    \widetilde{u} : C \to C \times X , &\quad u = \mathrm{proj}_X \circ \widetilde{u} : C \to X \\
    \widetilde{v} = (\mathrm{id}_C \times \pi) \circ \widetilde{u} : C \to C \times \mathbb{P}^1, &\quad v = \mathrm{proj}_{\mathbb{P}^1} \circ \widetilde{v} : C \to \mathbb{P}^1.
\end{align}

Fix $A \in H_2(X;\mathbb{Z}) \cong H_2(\mathbb{P}^1; \mathbb{Z})$, and consider the moduli problem using the decoupling perturbation data $\nu_X = \pi^* \nu$:
\begin{align}\label{eqn:map-moduli-decoupling}
\mathcal{M}_A = \mathcal{M}_A(X;\nu) :=   \left\{ u: C \to X : \ \overline{\partial}_J u = (\widetilde{u})^*   \pi^* \nu, \ u_*[C] = A  \right\}
\end{align}

Our particular choice of perturbation data is special in that the projected image $v = \pi \circ u :C \to \mathbb{P}^1$ \emph{becomes an actual holomorphic curve} for the complex structure on $\mathbb{P}^1$. To see this, note that for any tangent vector $\xi \in T_z C$, the perturbation $\nu_{z, v(z)}(\xi)$ takes values in $N \subseteq TX$. But $N$ is exactly the kernel of $d\pi: TX \to T\mathbb{P}^1$, so the perturbation vanishes under the projection. 

In other words, the perturbed $J$-holomorphic curve equation
\begin{equation}
\overline{\partial}_J u = (\widetilde{u})^* \pi^* \nu \in C^\infty \left( C, \mathrm{Hom}^{0,1}(TC, u^* TX) \right) \quad \iff \quad (\overline{\partial}_J u)_z = (\pi^*\nu)_{z, u(z)} = \nu_{z, v(z)}
\end{equation}
decouples as a pair of equations, using a section $\phi \in C^\infty (C, v^*N)$, as
\begin{equation}
\overline{\partial}_J v  = 0, \quad \overline{\partial} \phi = (\widetilde{v})^* \nu \in C^\infty \left(C, \mathrm{Hom}^{0, 1} (TC, v^*N) \right).
\end{equation} In other words, the section $\phi$ can be understood as the vertical deviation of the perturbed solution $u : C \to X$ from the zero section $\mathbb{P}^1 \subseteq X$. 

The discussion above allows a new characterization of the moduli problem as
\begin{equation}\label{eqn:dc-moduli}
\mathcal{M}_A \cong \left\{ \left( v: C \to \mathbb{P}^1, \ \phi \in C^\infty (C, v^*N) \right) : \  \overline{\partial}_J v = 0,\  \overline{\partial} \phi = (\widetilde{v})^* \nu, \ v_*[C] = A \right\}.
\end{equation}

\subsubsection{Map moduli space: Transversality}

We now show that the moduli space is still regular for perturbation data \emph{chosen within the special class} of decoupling perturbation data. Fix a reference decoupling perturbation data $\nu_0 \in C^\infty(C \times \mathbb{P}^1, \mathrm{Hom}^{0,1} (TC , N))$ and consider its neighborhood $\mathcal{S} \ni \nu_0$ in the $C^\infty$-topology.

\begin{lem}
\label{lem:dc-map-transverse}
There exists a comeager subset $\mathcal{S}^{\mathrm{reg}} \subseteq \mathcal{S}$ of decoupling perturbation data near $\nu_0$, such that for any $\nu \in \mathcal{S}^{\mathrm{reg}}$, the moduli space $\mathcal{M}_A = \mathcal{M}_A(X;\nu)$ is regular of dimension $\dim_{\mathbb{R}} X + 2 c_1(A) = 6$.
\end{lem}
\begin{proof}[Proof]
Fix a real number $\mathrm{p} > 2$. For the reference perturbation data $\nu_0 \in C^\infty(C \times \mathbb{P}^1, \mathrm{Hom}^{0,1} (TC , N))$, take $\mathcal{S}_{\varepsilon} \ni \nu_0$ to be its neighborhood in Floer's $C^\infty_\varepsilon$-topology. Recall that this is a Banach topology defined for a fixed sequence $\varepsilon = \{\varepsilon_n \}$ such that $\varepsilon_n \to 0$, as $\|\nu\|_{C^\infty_\varepsilon} = \sum_m \epsilon_m \| \nu\|_{C^m}$. We will choose the sequence $\varepsilon$ later in the proof; any discussion prior to that point will hold for any choice of $\varepsilon$.

Now consider the universal moduli space
\begin{equation}
    \label{universal-map-moduli-noneq}
    \mathcal{M}^{univ}_A = \{ (v, \phi ;  \nu \in \mathcal{S}_{\varepsilon}) : \ \overline{\partial}_J v = 0,\  \overline{\partial} \phi = (\widetilde{v})^* \nu, \  v_*[C]=A \}.
\end{equation}
We will show that the regular locus of the universal moduli space, where the linearized operator is surjective, is nonempty. To define the linearized operator at $(v, \phi, \nu) \in \mathcal{M}_A^{univ}$, fix a connection $\nabla$ for $\mathrm{Hom}^{0,1}(TC, N)$. The linearized operator, under suitable Banach completions, is a Fredholm operator which has the form
\begin{align}
    \label{universal-map-moduli-noneq-linearization}
    D^{univ}: L^\mathrm{p}_1 (C, v^* T \mathbb{P}^1) \oplus L^\mathrm{p}_1 (C, v^* N) \oplus T_\nu \mathcal{S}_\varepsilon &\to L^\mathrm{p} (C, \mathrm{Hom}^{0,1}(TC, v^* T \mathbb{P}^1)) \oplus L^\mathrm{p} (C, \mathrm{Hom}^{0,1}(TC, v^* N)) \\
    \left(\xi, \psi, \alpha \right) &\mapsto \left( D_v \xi, \  \overline{\partial} \psi - \nabla_\xi \nu + (\widetilde{v})^* \alpha  \right).
\end{align}
The decoupling of our perturbation data is reflected in that there is a corresponding decomposition
\begin{equation}\label{eqn:universal-linearized-decouples}
D^{univ} = \begin{pmatrix} D_1 & 0 \\ D' & D_2 \end{pmatrix},
\end{equation}
where
\begin{align}
    D_1 : L^\mathrm{p}_1 (C, v^* T \mathbb{P}^1) \to L^\mathrm{p} (C, \mathrm{Hom}^{0,1}(TC, v^* T \mathbb{P}^1)), \quad & D_1(\xi)= D_v\xi \\
    D' : L^\mathrm{p}_1 (C, v^* T \mathbb{P}^1) \to L^\mathrm{p} (C, \mathrm{Hom}^{0,1}(TC, v^* N)), \quad  & D'(\xi)= -\nabla_{\xi} \nu, \\
    D_2 : L^\mathrm{p}_1 (C, v^* N) \oplus T_\nu \mathcal{S}_\varepsilon \to L^\mathrm{p} (C, \mathrm{Hom}^{0,1}(TC, v^* N)), \quad & D_2(\psi, \alpha) = \overline{\partial} \psi +(\widetilde{v})^* \alpha .  
\end{align}

It suffices to show the surjectivity of $D_1$ and $D_2$ at $\nu = \nu_0$.

Note that since $J$ is integrable, the linearized operator $D_v$ is also just the Dolbeault operator $\overline{\partial}_J$. Consider the Dolbeault cohomology group $H^{0,1}_{\overline{\partial}}(C, v^*T\mathbb{P}^1)$, which is isomorphic to the cokernel of $D_1$. The bundle $v^*T\mathbb{P}^1 \cong \mathcal{O}(2 \cdot \deg v)$ has non-negative degree, so the sheaf cohomology group $H^1(C, v^*T\mathbb{P}^1) \cong H^{0,1}_{\overline{\partial}}(C, v^*T\mathbb{P}^1)$ vanishes. In particular, $D_1$ is surjective.

We show the surjectivity of $D_2$ at $\nu = \nu_0$. First we show that the image of $D_2$ is dense. Suppose the image of $D_2$ is not dense. For $\mathrm{q}$ such that $1/\mathrm{p} + 1/\mathrm{q} = 1$, fix a nonzero section $\beta \in L^{\mathrm{q}} (\mathrm{Hom}^{0,1}(TC, v^*N))$ which annihilates the image of $D_2$ (by the Hahn--Banach theorem). Then for any $\alpha \in T_{\nu} \mathcal{S}$ (in $C^\infty$ topology), it follows that
\begin{equation}
    \int_C \langle (\widetilde{v})^*\alpha, \beta \rangle \ d \mathrm{vol}_C = 0. 
\end{equation}
Fix $z \in C$ where $\beta(z) \neq 0$ ($\beta$ solves the adjoint of $D_2$, hence is smooth). There exists $\alpha \in C^\infty(C \times \mathbb{P}^1; \mathrm{Hom}^{0,1}(TC, N))$ such that $\langle (\widetilde{v})^*\alpha, \beta \rangle(z) \neq 0$. Since $\widetilde{v} : C \to C \times \mathbb{P}^1$ is an embedding, we can scale $\alpha$ by a bump function near $(z, v(z))$ so that rescaled $\alpha$ lies in $T_\nu \mathcal{S}$ and $\int_C \langle (\widetilde{v})^*\alpha, \beta \rangle \ d \mathrm{vol}_C \neq 0$. This is a contradiction, so $\beta \equiv 0$. But $\beta$ was assumed to be nonzero, so $D_2$ (with $T_\nu \mathcal{S}$ summand in the domain equipped with the $C^\infty$-topology) must have dense image. 

Since $L^\mathrm{p} (C, \mathrm{Hom}^{0,1}(TC, v^*N))$ is separable, we can choose dense subsequence $D_2(\psi_1, \alpha_1)$, $D_2 (\psi_2, \alpha_2), \dots \in L^\mathrm{p} (C, \mathrm{Hom}^{0,1}(TC, v^*N))$  in the image such that $\alpha_k \in T_{\nu} \mathcal{S}$ for every $k$. We can arrange the sequence $\varepsilon = \{\varepsilon_n\}$ so that all the $\alpha_k$'s have finite $C^\infty_\varepsilon$-norm. Therefore $D_2$ (with $T_\nu \mathcal{S}_\varepsilon$ summand in the domain equipped with the $C^\infty_\varepsilon$-topology, with that choice of $\varepsilon$) has closed dense image and hence is surjective. By elliptic regularity, one can enhance the spaces $L_1^{\mathrm{p}}$ and $L^{\mathrm{p}}$ to $L_{k}^{\mathrm{p}}$ and $L_{k-1}^{\mathrm{p}}$ resp. for all $k \ge 1$ and deduce that solutions are smooth.

Hence for a suitably chosen $\varepsilon$, we see that any curve in the moduli space $\mathcal{M}_A(X;\nu_0)$ is smooth and has surjective linearized operator in the $C^\infty_\varepsilon$-topology. This concludes the proof that the regular locus $\mathcal{M}_A^{univ, \mathrm{reg}} \subseteq \mathcal{M}_A^{univ}$ of the universal moduli space in the $C^\infty_\varepsilon$-topology is indeed nonempty.

Now by the Sard--Smale theorem applied to the regular locus of the universal moduli space and the projection map $\mathcal{M}_A^{univ, \mathrm{reg}} \to \mathcal{S}_\varepsilon$, we conclude that there exists a comeager subset $\mathcal{S}_\varepsilon^{reg} \subseteq \mathcal{S}_\varepsilon$ such that for every $\nu \in \mathcal{S}$, the corresponding moduli space $\mathcal{M}_A(X;\nu)$ is Fredholm regular. To pass from genericity in the $C^\infty_\varepsilon$-topology to genericity in the $C^\infty$-topology, we can employ the ``Taubes trick'' (see \cite[{\S}4.4.2]{Wen-lec}). The Taubes trick depends on the ability to exhaust the moduli space $\mathcal{M}(\nu)$ by countable collection of compact subsets, in a way that depends continuously on the data $\nu$. This is easily achieved in our setting, by exhausting the moduli spaces with increasing $C^1$-bounds ($\|du\| \le K$ for $K \to \infty$) on the curves.
\end{proof}

Note also that for a given holomorphic map $v : C \to \mathbb{P}^1$, there generically is no $\phi \in C^\infty(C, v^*N)$ such that $\overline{\partial} \phi = \widetilde{v}^* \nu$ is satisfied: the Dolbeault cohomology class of $[v^*\nu] \in H^{0,1}_{\overline{\partial}}(C, v^*N)$ must vanish for such $\phi$ to exist, but for $v$ of non-negative degree the space $H^{0,1}_{\overline{\partial}}(C, v^*N)$ is not trivial. Moreover, when there exists such $\phi$, there is a unique one, as the difference of two such $\phi$ lies in $H^{0,0}_{\overline{\partial}}(C, v^*N) \cong H^0(C, v^*N) = 0$.

\subsubsection{Map moduli space: Compactness}\label{sssec:map-moduli-cptness}
Note that as a perturbation data for $X$ in the usual sense, $\nu_X = \pi^*\nu$ is supported everywhere. To address compactness issues stemming from this, we recall (a much easier special case of) the theory of holomorphic curves for spaces with conical ends, as explained in \cite[{\S}5]{Rit14}.

Consider $\widetilde{X} = C \times X$ and take the symplectic form $\widetilde{\omega} = \mathrm{proj}_X^* \omega  + \mathrm{proj}_C^* \mathrm{vol}_C \in \Omega^2(\widetilde{X})$.

Let $\beta : X \to \mathbb{R}$ be a compactly supported function that is constantly equal to $1$ in some relatively compact neighborhood $U_\beta$ of the zero section $\mathbb{P}^1 \subseteq X$. Denote its pullback to $\widetilde{X}$ as $\widetilde{\beta}$. Then consider the almost complex structure on $C \times X$ defined by
\begin{equation}
\label{acs-on-graph-space}
    \widetilde{J} := \begin{pmatrix} j & 0 \\ \widetilde{\beta}(\nu_X \circ j - J \circ \nu_X) & J  \end{pmatrix}.
\end{equation}

The following is a result of straightforward computation.
\begin{lem}
\label{lem:acs-compatibility}
    The almost complex structure $\widetilde{J}$ is $\widetilde{\omega}$-compatible. A graph $\widetilde{u} : C \to \widetilde{X}$ is a solution of unperturbed $\widetilde{J}$-holomorphic curve equation if and only if $u : C \to X$ is a solution of the perturbed $J$-holomorphic curve equation $\overline{\partial}_J u = \widetilde{\beta} \cdot \nu_X$. 
\end{lem}

The cutoff $\beta$ is introduced so that a maximum principle applies to the graphs $\widetilde{u} : C \to \widetilde{X}$, see \cite[Lemma 32]{Rit14}. In the terminology of \cite[{\S}5]{Rit14}, it renders $\widetilde{J}$ to be \emph{admissible}. Indeed, the solutions can never touch the region where $\widetilde{\beta} \equiv 0$, which suffices for compactness.

However, the desired decoupling of the moduli space uses the description of the equation as $\overline{\partial}_J u = \nu_X$, not $\overline{\partial}_J u = \widetilde{\beta} \cdot \nu_X$, and therefore only applies to the neighborhood $U_\beta$ of $\mathbb{P}^1$ where $\beta \equiv 1$. To bypass this issue, we fix $\beta$ throughout, and consider the sequence of perturbation data $\nu_{X,n}$ such that $\nu_{X,n} \to 0$. In the limit $\nu_{X, \infty} = 0$, all solutions are contained in the zero section $\mathbb{P}^1 \subseteq X$ again by maximum principle. Hence for all large $n$, the solutions to $\overline{\partial}_{J} u = \widetilde{\beta} \cdot \nu_{X, n}$ lie in the neighborhood $U_\beta$. In particular, we can always assume that the solutions lie in the region where $\beta \equiv 1$ by scaling $\nu_X$ by a sufficiently small constant. We will always assume that this is ensured for our choices of $\nu$. 

The bordism class of the moduli space of $\widetilde{J}$-holomorphic curves is independent of the choice of $\beta$ by a standard cobordism argument (\cite[Lemma 40]{Rit14}).


The moduli space still carries evaluation maps at the distinguished marked points, $\mathrm{ev}: \mathcal{M}_A \to X \times X^p \times X $. By Gromov compactness, the limit points of the evaluation map is covered by simple stable maps, and in particular involves bubbling. Since we have fixed our $J$ to be integrable, the $J$-holomorphic spheres in $X$ will in general not be of expected dimension. Indeed, the image of non-constant holomorphic bubbles have to be contained in the zero section $\mathbb{P}^1 \subseteq X$ by the maximum principle, and the dimension of such curves in $\mathbb{P}^1$ deviates from the expected dimension of the curves when they are considered as curves in $X$.

Despite the incorrect dimension of the bubbling spheres, it could be achieved in our situation that the limit points of the evaluation map have codimension at least $2$, hence the evaluation map defines a pseudocycle. We must verify this property in order to use our pair $(J, \nu)$ to define the Gromov--Witten invariants. The following observation is simple but essential.

\begin{lem}
\label{lem:vertical-vanishing-at-marked-points}
Fix a solution $u = (v, \phi) \in \mathcal{M}_A$. Then a point on the graph $\widetilde{u}(z) = (z, u(z))$ is lies in $C \times \mathbb{P}^1 \subseteq C \times X$ if and only if the corresponding section $\phi \in C^\infty(C, v^*N)$ vanishes at $z$. 
\end{lem}
\begin{proof}
Note that $u(z) \in X$ is determined by the value $\pi(u(z)) = v(z) \in \mathbb{P}^1$ and $\phi(z) \in (\mathcal{O}(-1) \oplus \mathcal{O}(-1))_{v(z)}$, so that $u(z)$ as an element of the fiber of $X \to \mathbb{P}^1$ over $v(z)$ is exactly $\phi(z)$. Therefore the condition that $u(z)$ lies in the zero section is equivalent to the condition that $\phi(z) = 0$.
\end{proof}

Now recall that the bubbles in the fiber $\{ z\} \times X \subseteq \widetilde{X}$ are $J$-holomorphic, and therefore contained in the zero section $\mathbb{P}^1 \subseteq X$. Consider the principal component, i.e. the component of the stable map that carries the perturbation datum. For the principal component to connect to the bubbles, the lemma above imposes that the primitive $\phi$ of the perturbation datum $\nu$ must vanish at the nodal points. This is a codimension $4$ condition, leading to the pseudocycle property for the evaluation map.

Let us fix some notation. Let $T$ be a tree with a distinguished root vertex $0 \in T$. Given $A \in H_2(X)$, consider a decomposition of the degrees
\begin{equation}
    A = A_0 + \sum_{0 \neq b \in T} A_b, \quad A_b = n_b A'_b, \ n_b \in \mathbb{Z}_{\ge 0}
\end{equation}
where $A'_b$ is a primitive class (i.e. not a multiple cover of another class in $H_2(X)$). Given $T$, any such decomposition yields the tuple of degrees $\{A_b'\}_{b \in T} = \{A_0\} \sqcup \{A_b'\}_{0 \neq b \in T}$. Note that the same degree $A$ can be decomposed in many different ways.

\begin{defn}
\label{defn:simple-stable-maps}
    Fix a star-shaped tree $T$ with a distinguished root vertex $0 \in T$, with exactly one edge between $0$ and every $0 \neq b \in T$, and no other edges. Denote by $k = |T| - 1$ the number of non-root vertices. Fix also a decomposition of $A \in H_2(X)$ into $\{A_b'\}_{b \in T}$. 
    
    The \emph{(non-equivariant) simple stable map moduli space of degree $\{A_b'\}_{b \in T}$} is a moduli space
    \begin{equation}
        \mathcal{M}_{T, \{A_b'\}} (X;\nu) \subseteq \mathcal{M}_{k-\ell, A_0}(X;\nu) \times \prod_{b=1}^\ell \mathcal{M}_2 (X ; A_b') \times \prod_{b=\ell+1}^{k} \mathcal{M}_1(X; A_b' )
    \end{equation}
    consisting of tuples of
    \begin{itemize}
        \item $u_0: C \to X$ such that $(u_0)_*[C] = A_0$, and $\overline{\partial}_J u_0 = \nu$, with distinct marked points $z_1', \dots, z_k' \in C$ such that $\{z_1', \dots, z_\ell'\} \subseteq \{ z_0, z_1, \dots, z_p, z_\infty\}$, and $\{z_{\ell +1}' , \dots, z_k'\} \cap \{ z_0, z_1 , \dots, z_p, z_\infty \} = \emptyset$.
        \item $v_1, \dots, v_\ell : C \to X$ such that $(v_b)_* [C] = A_b'$, and $\overline{\partial}_J v_b = 0$, with two marked points $z''_{b}$ and $Z_b$, up to reparametrization (type (I) bubbles, bubbling off at the special marked points).
        \item $v_{\ell+1}, \dots, v_k : C \to X$ such that $(v_b)_* [C] = A_b'$, and $\overline{\partial}_J v_b = 0$, with a unique marked point $z''_{b}$, up to reparametrization (type (II) bubbles, bubbling off at a point that is not one of the special marked points).
    \end{itemize}
    satisfying the nodal connecting constraint $u_0(z_b') = v_b (z''_{b})$ for every $0 \neq b \in T$. In particular, the $(p+2 - \ell)$ points $\left(\{z_0, z_1, \dots, z_p, z_\infty\} \setminus \{ z_1', \dots, z_\ell'\}\right)$ on the principal component and $\ell$ points $\{ Z_1, \dots, Z_\ell\}$ on the bubbles together consist the marked points. There is a corresponding evaluation map $\mathrm{ev}_T : \mathcal{M}_{T, \{A_b'\}} \to X \times X^p \times X$ which evaluates at $u(z_j)$ or $v_b(Z_b)$.
\end{defn}

Finally, denote by
\begin{equation}
    \mathcal{M}_{T, \{A_b'\}} (X;\nu) ( - \Sigma z_j') \subseteq \mathcal{M}_{T, \{A_b'\}} (X;\nu)
\end{equation}
to be the subspace of the simple stable maps where the principal component $(u_0 : C \to X) = (v_0 : C \to \mathbb{P}^1, \phi_0 \in C^\infty(C, v_0^* N)$ satisfies the vanishing condition $\phi_0 (z_1') = \cdots = \phi(z_k') = 0$ at the nodal points $z_1', \dots, z_k'$ where bubbles connect. The moduli space $\mathcal{M}_{k-\ell, A_0}(X;\nu)(-\Sigma z_j') \subseteq \mathcal{M}_{k-\ell, A_0}(X;\nu)$ (which only consists of the principal component) is defined in the same way.

\begin{prop}[Gromov compactness, compare \cref{prop:gromov-generic}]\label{prop:gromov-decoupling}
For a generic choice of decoupling perturbation data $\nu$, the following are satisfied.

\begin{enumerate}[topsep=-2pt, label=(\roman*)]
\item The limit set of the evaluation map from $\mathcal{M}_A$ is covered by evaluation maps from simple stable maps of the form
\begin{align}
\mathcal{M}_{T,  \{ A_b ' \}} (X; \nu) (- \Sigma z_j' ) 
\end{align}

\item Moreover, the moduli spaces of simple stable maps with at least one bubble are smooth manifolds of dimension $\le \dim \mathcal{M}_A - 2$.

\end{enumerate}

In particular, for the integrable $J$ and generic decoupling perturbation data $\nu$, the evaluation map from $\mathcal{M}_A$ still defines a pseudocycle.
\end{prop}

The union of all such simple stable maps will be denoted

\begin{equation}
    \partial \overline{\mathcal{M}}_A := \bigcup_{T} \bigcup_{\{A_b'\}_{b \in T}} \mathcal{M}_{T,  \{ A_b ' \}} (X; \nu) (- \Sigma z_j' ).
\end{equation}

\begin{proof}

(i) By \cref{lem:acs-compatibility}, we can use the product symplectic form $\widetilde{\omega}$ to tame $\widetilde{J}$ of \eqref{acs-on-graph-space} and apply Gromov compactness. By Gromov compactness, the image of the evaluation map is covered by the image from simple stable maps with spheres bubbling off on fibers of the trivial bundle $C \times X \to C$. All such bubbles are $J$-holomorphic, and in particular holomorphic maps into $\mathbb{P}^1 \subseteq X$. There are two types of such bubbles, which may (I) bubble off over points $z \in C$ that are one of the special points $\{z_0, z_1, \dots, z_p, z_\infty\} \subseteq C$ or (II) bubble off over points $z \in C$ that are not one of the special points. (In the former case, the marked point (the special point on $C$) moves to the bubble.) That is, the limit set of evaluation map from $\mathcal{M}_A(X;\nu)$ is covered by evaluation maps from $\mathcal{M}_{T, \{A_b'\}} (X; \nu)$.

Note that for the bubbles to connect to the principal component, the nodal connecting constraint $u_0(z_b') = v_b(z''_b)$ must be satisfied for each $0 \neq b \in T$. (This is part of the definition of $\mathcal{M}_{T, \{A_b'\}}$.) Since bubbles are contained in the zero section $\mathbb{P}^1 \subseteq X$, we have $v_b (z''_b) \in \mathbb{P}^1$. \cref{lem:vertical-vanishing-at-marked-points} implies that for $u_0 = (v_0, \phi_0)$, the section $\phi_0$ must vanish at every $z_1', \dots, z_k'$. Hence the simple stable maps that cover the evaluation map in fact come from $\mathcal{M}_{T, \{A_b'\}}(-\Sigma z_j') \subseteq \mathcal{M}_{T, \{A_b'\}}$, as desired.

(ii) The vanishing condition on $\phi$ for the principal component is clearly a transversal, real codimension $4$ condition for every marked point $z_1', \dots, z_k'$. Therefore the expected dimension of the moduli space of the principal component, assuming there are $\ell$ type (I) bubbles and $k-\ell$ type (II) bubbles, is
\begin{equation}
\label{vdim-noneq-principal-component}
\mathrm{vdim}_{\mathbb{R}} \mathcal{M}_{k-\ell, A_0} \left(- \Sigma z_j' \right) = \dim_{\mathbb{R}} X + 2c_1(A_0) + 2(k-\ell) - 4k = 6 -2k -2 \ell. 
\end{equation}
Note that $H_2(X;\mathbb{Z}) = H_2(\mathbb{P}^1;\mathbb{Z})$ has a unique primitive class that supports a holomorphic curve, namely the fundamental class of $\mathbb{P}^1$. Therefore the degrees at the bubbles, $A_b'$ for $b=1, \dots, k$, must all equal $A_b' = [\mathbb{P}^1]$ for a simple stable map. The dimension of the moduli space of type (I) bubbles and type (II) bubbles are then
\begin{equation}
\label{vdim-noneq-bubbles}
\dim_{\mathbb{R}} \mathcal{M}_2 (\mathbb{P}^1 ; A_b') = \dim_{\mathbb{R}} \mathbb{P}^1 + 2 \chi(\mathbb{P}^1) + 2 \cdot 2 - 6 = 4, \quad \dim_{\mathbb{R}} \mathcal{M}_1 (\mathbb{P}^1 ; A_b') = 2,
\end{equation}
respectively. Since $v_b: C \to \mathbb{P}^1$ is degree $1$ and the marked points $z_b''$ can move around freely, the connecting constraints $u_0(z_b') = v_b(z''_b)$ transversally cut down the dimension by $2$ for every $b = 1, \dots, k$. Therefore
\begin{equation}
\mathrm{vdim}_{\mathbb{R}} \mathcal{M}_{T, \{A_b'\}}(-\Sigma z_j') = (6 -2k - 2\ell) + 4 \ell + 2 (k-\ell) - 2k = 6 - 2k \le \dim_{\mathbb{R}} \mathcal{M}_A - 2,
\end{equation}
where $\dim_{\mathbb{R}} \mathcal{M}_A = 6$ from \cref{lem:dc-map-transverse}. Then $\nu$ can be chosen generically to ensure that the moduli space is a smooth manifold of the expected dimension.
\end{proof}

\subsection{Aspinwall--Morrison--Voisin compactification}

The previous compactification of the moduli space was based on stable maps. This compactification allows to define a pseudocycle via the evaluation map, whose limit set is covered by images from simple stable maps.

We now describe an alternative compactification of the moduli space, due to \cite{AM93} and \cite{Voi96} who used it in the proof of multiple cover formulae for Gromov--Witten invariants of the quintic threefold. It is closely related to the qausimap compactification which featured in \cite{Giv95}.

The main advantage of using the alternative compactification is that it reduces the computation of Gromov--Witten invariants to certain Chern class computations. The main technical obstacle is to verify that the counts defined using the alternative compactification agree with the counts defined earlier using the stable map compactification.

The main observation is the following. A degree $d$ holomorphic map $v : C \to \mathbb{P}^1$ has an associated graph $\widetilde{v}: C \to C \times \mathbb{P}^1$, which defines a holomorphic subvariety $\mathrm{im}(\widetilde{v})\subseteq C \times \mathbb{P}^1$. The graph is cut out by a holomorphic section of $\mathcal{O}_{C \times \mathbb{P}^1}(d,1) \to C \times \mathbb{P}^1$. Let us denote this section by $s_v : C \times \mathbb{P}^1 \to \mathcal{O}(d,1)$, so that $s_v^{-1}(0) = \mathrm{im}(\widetilde{v})$.

A natural compactification for the zero sets of such sections is given by taking the whole linear system:

\begin{defn}
\label{defn:thick-noneq-section-moduli}
Let $A = d [\mathbb{P}^1] \in H_2(X;\mathbb{Z}) \cong H_2(\mathbb{P}^1; \mathbb{Z})$ for $d \ge 0$. The \emph{thickened (non-equivariant) section moduli space of degree $A$} is the complete linear system $|\mathcal{O}(d,1)|$, namely
\begin{equation}
\mathcal{P}_A = \mathcal{P}_d  := \mathbb{P} \left( H^0 (C \times \mathbb{P}^1, \mathcal{O}_{C \times \mathbb{P}^1} (d, 1)) \right) \cong \mathbb{P}^{2d+1}.
\end{equation}
\end{defn}

Hence a point in $\mathcal{P}_d$ represents a nonzero section of $\mathcal{O}(d,1)$ up to scalar multiplication. We will nevertheless denote the points by $s \in \mathcal{P}_d$.
The thickened section moduli space admits an open subset $\mathcal{P}_d^\circ \subseteq \mathcal{P}_d$ where $s \in \mathcal{P}_d$ cuts out an irreducible subvariety, i.e. where the zero locus of $s = s_v$ is genuinely a graph of a degree $d$ holomorphic map $v: C \to \mathbb{P}^1$. In fact, this open locus is the open dense stratum in the following natural stratification.

\begin{prop}\label{prop:thick-section-moduli-stratification}
The thickened section moduli space $\mathcal{P}_d$ admits a stratification as follows:
\begin{equation}
\mathcal{P}_d = \mathcal{P}_d^\circ \sqcup \bigsqcup_{\alpha \neq 0} \mathcal{P}_d^\alpha, \quad \alpha = (\alpha_1, \alpha_2, \dots), \ \alpha_k = 0 \mbox{ for } k \gg 0, \ \sum_{k \ge 0} k \cdot \alpha_k \le d.
\end{equation}
The stratum $\mathcal{P}_d^\alpha$ consists of zero locus of sections $s \in H^0(\mathcal{O}(d,1))$ such that the subscheme $Z(s) := \{ s = 0 \} \subseteq C \times \mathbb{P}^1$ has irreducible components given by
\begin{enumerate}[topsep=-2pt, label=(\roman*)]
\item (the \emph{principal component}) a graph of a map of degree $d - \sum_{k \ge 0} k \cdot \alpha_k $,
\item (the \emph{(possibly nonreduced) vertical bubbles}) for each $k \ge 1$, copies of $\{p_j \} \times \mathbb{P}^1$, nilpotently thickened to order $k-1$ (i.e. scheme-theoretically isomorphic to $\mathrm{Spec}(\mathbb{C}[x]/(x^k)) \times \mathbb{P}^1$), attached at $\alpha_k$ distinct points $p_{k,1}, \dots, p_{k, \alpha_k} \in C$. 
\end{enumerate}
\end{prop}
Concretely, note that a holomorphic section $s \in H^0(\mathcal{O}(d,1))$ can be described as a multi-homogeneous polynomial with bidgree $(d,1)$ in variables $X_0, X_1$ and $Y_0, Y_1$:
\begin{equation}\label{eqn:section-of-O(d,1)-concretely}
s(X, Y) = P_0 (X_0, X_1) Y_0 + P_1 (X_0, X_1) Y_1, \quad \deg P_0 = \deg P_1 = d.
\end{equation}
Then $\{ P_0 = P_1 = 0 \} \subseteq C$ defines a subscheme that is topologically a union of $\sum_{k \ge 0} \alpha_k$ points, where exactly $\alpha_k$ many of the common roots of $P_0 = P_1 = 0$ have multiplicity $k$. The pullback of this subscheme along the projection $C \times \mathbb{P}^1 \to C$ is exactly the locus of the  vertical bubbles (which are nonreduced if $k > 1$). The remaining factor $s(X,Y) / \mathrm{gcd}(P_0, P_1)$ cuts out the principal component. 

Over $\mathcal{P}_d$, there exists a vector bundle $\mathrm{Obs}$ defined below. We must first define the universal divisor, which is an analogue of the universal curve for the stable map compactification.

\begin{defn}\label{defn:universal-divisor}
The \emph{universal divisor} $\mathcal{D} \subseteq \mathcal{P}_d \times ( C \times \mathbb{P}^1)$ is a divisor such that
\begin{equation}
    \mathcal{O}(\mathcal{D}) = \mathcal{O}_{\mathcal{P}_d}(1) \boxtimes \mathcal{O}_{C \times \mathbb{P}^1} (d, 1),
\end{equation}
where it is cut out by the tautological section coming from the canonical identification $H^0(\mathcal{O}_{\mathcal{P}_d}(1)) \cong H^0(\mathcal{O}_{C \times \mathbb{P}^1}(d,1))^\vee$. 
\end{defn}
Set-theoretically, $\mathcal{D}$ is the union of $\{s\} \times (\{s=0\} \subseteq C \times \mathbb{P}^1)$.

To define $\mathrm{Obs}$, consider the projection maps
\begin{equation}
    \begin{tikzcd}
    & \mathcal{D} \dar[phantom, sloped, "\subseteq"] && \\
        \mathcal{P}_d & \lar["\mathrm{proj}_{\mathcal{P}_d}"'] \mathcal{P}_d \times (C \times \mathbb{P}^1) \rar["\mathrm{proj}_{C \times \mathbb{P}^1}"] & C \times \mathbb{P}^1 \rar["\mathrm{proj}_{\mathbb{P}^1}"] & \mathbb{P}^1
    \end{tikzcd}.
\end{equation}
\begin{defn}\label{defn:obstruction-bundle}
    The \emph{(non-equivariant) obstruction bundle over the section moduli space} is the coherent sheaf
    \begin{equation}\label{eqn:obstruction-bundle}
        \mathrm{Obs} = \mathrm{Obs}_d := R^1(\mathrm{proj}_{\mathcal{P}_d})_* \left( \mathcal{O}_{\mathcal{D}} \otimes  (\mathrm{proj}_{\mathbb{P}^1} \circ \mathrm{proj}_{C \times \mathbb{P}^1})^* N \right).
    \end{equation}
    Here, $\mathcal{O}_\mathcal{D} \cong \mathcal{O}_{\mathcal{P}_d \times (C \times \mathbb{P}^1)} / \mathcal{O}(-\mathcal{D})$ is the structure sheaf of $\mathcal{D}$, hence $\mathcal{O}_\mathcal{D} \otimes$ is restriction to $\mathcal{D}$.
\end{defn}
\begin{lem}[{\cite[Lemma 2.2]{Voi96}}]\label{lem:obs-description}
    There is an isomorphism 
    \begin{equation}
    \mathrm{Obs} \cong \mathcal{O}_{\mathcal{P}_d}(-1) \otimes H^2(C \times \mathbb{P}^1, \mathcal{O}(-d, -2))^{\oplus 2} \cong \mathcal{O}_{\mathcal{P}_d}(-1) \otimes \left(H^0(C, \mathcal{O}_C(d-2))^{\vee}\right)^{\oplus 2} 
    \end{equation}
    determined by a choice of an isomorphism $K_{C \times \mathbb{P}^1} \cong \mathcal{O}_{C \times \mathbb{P}^1}(-2,-2)$. In particular, $\mathrm{Obs}$ is a vector bundle on $\mathcal{P}_d$ of complex rank $2(d-1)$.
\end{lem}
\begin{proof}[Proof sketch]
    The first isomorphism is canonical, coming from the long exact sequence of derived pushforward associated to the sequence of sheaves obtained by tensoring $(\mathrm{proj}_{\mathbb{P}^1} \circ \mathrm{proj}_{C \times \mathbb{P}^1})^* N$ to the ideal sheaf sequence $0 \to \mathcal{O}(-\mathcal{D}) \to \mathcal{O} \to \mathcal{O}_\mathcal{D} \to 0$. The second isomorphism is Serre duality and depends on the choice of an isomorphism $K_{C \times \mathbb{P}^1} \cong \mathcal{O}_{C \times \mathbb{P}^1}(-2,-2)$.
\end{proof}

It is useful to think of $\mathrm{Obs}$ in terms of its fibers. At a point $s \in \mathcal{P}_d^\circ$, there is a corresponding holomorphic curve $v: C \to \mathbb{P}^1$ such that $s=s_v$ cuts out the graph of $v$. Then the fiber of $\mathrm{Obs}|_{s}$ is canonically isomorphic to $H^1(C, v^*N)$. To see this, note that the definition of $\mathrm{Obs}$ is via the derived pushforward from the universal divisor $\mathcal{D} \subseteq \mathcal{P}_d \times (C \times \mathbb{P}^1)$. The fiber of $s \in \mathcal{P}_d$ under the projection $\mathcal{D} \to \mathcal{P}_d$ is exactly the zero locus of the section $s$ inside $C\times \mathbb{P}^1$. When $s = s_v$ is graphical (i.e. cuts out a graph of a map $v$, which is the case for $s \in \mathcal{P}_d^\circ$), the corresponding fiber in $\mathcal{D}$ is exactly the graph of $v$, identified with the domain $C$. The pullback of $N$ to the graph of $v$ under the projection to $\mathbb{P}^1$ is therefore exactly $v^*N$ over $C$, and taking the derived pushforward $R^1$ corresponds to taking the sheaf cohomology $H^1$.

The Dolbeault cohomology class of the perturbation datum $\nu$ pulled back by $\widetilde{v}$ is naturally an element of $H^1(C, v^*N)$. (Compare the discussion after \cref{lem:dc-map-transverse}.) The choice of the decoupling perturbation data $\nu$ therefore induces a smooth section of $\mathrm{Obs}$ \emph{over the open subset} $\mathcal{P}_d^\circ$, which associates to the graph of holomorphic $v : C \to \mathbb{P}^1$ the corresponding Dolbeault cohomology class of $\widetilde{v}^* \nu$:
\begin{align}\label{eqn:obs-section}
[\nu] : \mathcal{P}_d^\circ &\to \mathrm{Obs}, \\
s_v & \mapsto [ (\widetilde{v})^* \nu ] \in H^{0,1}_{\overline{\partial}} (C, v^*N) \cong H^1(C, v^*N) \cong \mathrm{Obs}|_{s_v}.
\end{align}

\begin{defn}
\label{defn:noneq-sec-moduli}
Fix $A = d [\mathbb{P}^1] \in H_2(\mathbb{P}^1;\mathbb{Z})$ as before. The \emph{(non-equivariant) section moduli space of degree $A$} is
\begin{equation}
\mathcal{V}_A = \mathcal{V}_d := \mathcal{P}_d^\circ \cap [\nu]^{-1}(0) =  \left\{ s_v \in \mathcal{P}_d^\circ : \ [\nu] (s_v) = 0 \right\},
\end{equation}
the zero locus of the partially defined section $[\nu]$ of $\mathrm{Obs}$ over $\mathcal{P}_d^\circ$.
\end{defn}

\subsubsection{Section moduli space: Transversality}

A generic choice of $\nu$ gives a generic choice of the section $[\nu]$, so we have the following analogue of \cref{lem:dc-map-transverse}.
\begin{lem}\label{lem:section-transverse}
For a generic choice of decoupling perturbation data $\nu$, the moduli space $\mathcal{V}_A$ is a smooth manifold of dimension $\dim_{\mathbb{R}} X + 2 c_1(A) = 6$.
\end{lem}

We moreover orient $\mathcal{V}_A$ using its description as a zero locus of a section of a complex vector bundle over a complex manifold.

The section moduli space $\mathcal{V}_A$ cut out by the obstruction bundle should be considered as the analogue of the main stratum $\mathcal{M}_A$ in the stable map compactification. Indeed, recall the decoupling \eqref{eqn:dc-moduli} of equations for $\mathcal{M}_A$ which we described earlier. The moduli of holomorphic curves $v: C \to \mathbb{P}^1$ corresponds to the thickened moduli space $\mathcal{P}_A^\circ \subseteq \mathcal{P}_A$, while the moduli of solutions of the perturbed equation for the normal perturbation corresponds to the locus where the section $[\nu]$ of $\mathrm{Obs}$ vanishes. This discussion can be summarized as follows.

\begin{prop}[{\cite[Lemma 4.1, 4.2]{Voi96}}]\label{prop:noneq-map-equal-section}
Let $A = d [\mathbb{P}^1] \in H_2(X;\mathbb{Z})$. Assume that $\nu$ is chosen so that the section moduli space $\mathcal{V}_A$ and the map moduli space $\mathcal{M}_A$ are both regular. Then $\mathcal{V}_A$ is isomorphic to $\mathcal{M}_A$ of \eqref{eqn:map-moduli-decoupling} as smooth oriented manifolds of dimension $\mathrm{dim}_{\mathbb{R}} X + 2 c_1(A) = 6$.
\end{prop}
\begin{proof}
We first establish the set-theoretic bijection. Fix an element $s \in \mathcal{V}_A$. Since $s \in \mathcal{P}_d^\circ$, the zero locus $Z(s) = \{ s= 0\} \subseteq C \times \mathbb{P}^1$ is a graph of some holomorphic map $v : C \to \mathbb{P}^1$ of degree $d$, i.e. $s = s_v$. Since $[\nu](s_v) = 0$, the Dolbeault cohomology class of $(\widetilde{v})^* \nu$ vanishes and there is a section $\phi \in C^\infty(C, v^*N)$ such that $\overline{\partial} \phi = (\widetilde{v})^*\nu \in C^\infty(C, \mathrm{Hom}^{0,1}(TC, v^*N))$. Moreover, such $\phi$ is unique since $H^0(C, v^*N) = 0$. Hence there is a well-defined element $(v, \phi) \in \mathcal{M}_A$. For the other direction, take the holomorphic section that cuts out the graph of $v$.

To see that this bijection is in fact a diffeomorphism, consider the linearizations of $\mathcal{V}_A$ and $\mathcal{M}_A$. The tangent space of the map moduli at $(v, \phi) \in \mathcal{M}_A$ is modelled on the kernel of the linearized operator: 
\begin{align}
D_{(v, \phi)} : C^\infty( C, v^* T\mathbb{P}^1) \oplus C^\infty(C, v^*N) &\to C^\infty(C, \mathrm{Hom}^{0,1}(TC, v^*T\mathbb{P}^1)) \oplus C^\infty (C, \mathrm{Hom}^{0,1}(TC, v^*N)).
\end{align}
As in \eqref{universal-map-moduli-noneq-linearization}, this decomposes as
\begin{align}\label{eqn:linearization-splitting}
D_{(v, \phi)} (\xi, \psi) =  \left( \overline{\partial}_{v^* T\mathbb{P}^1} \xi , \ \overline{\partial}_{v^*N} \psi - \nabla_{\xi} \nu \right)
\end{align}
where the Dolbeault operator $\overline{\partial}_{v^* T\mathbb{P}^1}$ is surjective and $\overline{\partial}_{v^* N}$ is injective. In particular, the linearization $\nabla_\xi \nu$ induces an isomorphism
\begin{align}
T_{(v, \phi)} \mathcal{M}_A \cong \ker D_{(v, \phi)} \cong \ker \left( \nabla \nu: \ker(\overline{\partial}_{v^*T\mathbb{P}^1}) \to \mathrm{coker} (\overline{\partial}_{v^*N}) \right).
\end{align}
Compare this with the tangent space of $\mathcal{V}_A$.  Writing $s = s_v = (v, \phi) \in \mathcal{V}_A$ using the set-theoretic bijection, the tangent space is given by
\begin{align}
T_{(v,\phi)} \mathcal{V}_A  \cong \ker \left( d[\nu] : T_{{s_v}} (\mathcal{P}_d)^\circ \to \mathrm{Obs}|_{s_v} \right).
\end{align}
Now note that a one-parameter family $v_t : C \to \mathbb{P}^1$, $t \in (-\varepsilon, \varepsilon)$ induces a one-parameter family of sections $s_t = s_{v_t} \in (\mathcal{P}_d)^\circ$. Conversely, a one-parameter family of sections $s_t \in (\mathcal{P}_d)^\circ$ induces a one-parameter family of graphs $\widetilde{v}_t = \{ s_t = 0 \} \subseteq C \times \mathbb{P}^1$ (all graphical for small $t$ if $s = s_0$ is graphical), by (analytic) implicit function theorem. Hence we identify the deformation space of holomorphic maps near $v : C \to \mathbb{P}^1$ and the tangent space of the sections cutting out their graphs near $s_v \in \mathcal{P}_d^\circ$, that is we identify $\ker(\overline{\partial}_{v^*T\mathbb{P}^1}) \cong T_{s_v}(\mathcal{P}_d)^\circ$. Moreover, $\mathrm{coker}(\overline{\partial}_{v^*N}) \cong H^1(C, v^*N) \cong \mathrm{Obs}|_{s_v}$ from the definition of the obstruction bundle (\cref{defn:obstruction-bundle}). This identifies the tangent spaces $T_{(v,\phi)} \mathcal{M}_A \cong T_{(v,\phi)} \mathcal{V}_A$. 

For the agreement of orientations, consider the determinant lines (after passing to suitable Banach completions). Observe that the complex antilinear part of the linearized operator $D_{(v,\phi)} = (\overline{\partial}_{v^*T\mathbb{P}^1} , \overline{\partial}_{v^*N} - \nabla \nu)$ is of lower order (order zero). Hence $D_{(v,\phi)}$ can be homotoped to the complex linear operator $D^{1,0} = (\overline{\partial}_{v^*T\mathbb{P}^1}, \overline{\partial}_{v^*N}$). It remains to observe that the determinant line of $D$ determines the orientation of $\mathcal{M}_A$, and the determinant line of $D^{1,0}$ determines the orientation of $\mathcal{V}_A$ (as the zero locus of a section of a complex vector bundle over a complex manifold). 
\end{proof}

\subsubsection{Section moduli space: Compactness}
The section moduli space $\mathcal{V}_A$ is, by \cref{lem:section-transverse} and \cref{prop:noneq-map-equal-section}, a replacement for $\mathcal{M}_A$. The main result of \cite{Voi96} is reproduced below. It allows one to extract enumerative information from $\mathcal{V}_A = \mathcal{V}_d$ by means of cohomological computations.

\begin{thm}[{\cite[Theorem 3.2]{Voi96}}]\label{thm:noneq-section-compactification}
For a generic choice of $\nu$, the following holds. 
\begin{enumerate}[topsep=-2pt, label=(\roman*)]
    \item $\mathcal{V}_d \subseteq \mathcal{P}_d^\circ$ is smooth of real dimension $\dim_{\mathbb{R}} X + 2c_1(d[\mathbb{P}^1]) = 6$ (\cref{lem:section-transverse}), 
    \item The closure $\overline{\mathcal{V}}_d \subseteq \mathcal{P}_d$ of $\mathcal{V}_d$ is stratified as follows (compare \cref{prop:thick-section-moduli-stratification}):
    \begin{equation}
\overline{\mathcal{V}}_d = \mathcal{V}_d \sqcup \bigsqcup_{\alpha \neq 0} \mathcal{V}_d^\alpha, \quad \alpha = (\alpha_1, \alpha_2, \dots), \ \alpha_k = 0 \mbox{ for } k \gg 0, \ \sum_{k \ge 0} k \cdot \alpha_k \le d.
\end{equation}
    \item For each $\alpha \neq 0$, the stratum $\mathcal{V}_d^\alpha$  is contained in a locally closed subvariety of  real dimension $\le \dim \mathcal{V}_d - 2$, and hence $\overline{\mathcal{V}}_d$ defines a homology class,
    \item The homology class represented by $[\overline{\mathcal{V}}_d]$ is Poincar\'e dual to $c_{\mathrm{top}}(\mathrm{Obs}) \in H^*(\mathcal{P}_d)$.
\end{enumerate}
\end{thm}

The proof of \cref{thm:noneq-section-compactification} hinges on the following key lemma.

\begin{lem}[{\cite[Proposition 3.3]{Voi96}}]\label{lem:extension-of-obs-section}
    There exists a quotient sheaf $\overline{\mathrm{Obs}}$ of $\mathrm{Obs}$ such that the partially defined smooth section $[\nu]$, defined over $\mathcal{P}_d^\circ \subseteq \mathcal{P}_d$, extends continuously to a global section $\overline{[\nu]}$ of $\overline{\mathrm{Obs}}$:
    \begin{equation}\label{eqn:extension-of-obs-section}
        \begin{tikzcd}
            & \mathrm{Obs} \dar \arrow[r, twoheadrightarrow] & \overline{\mathrm{Obs}} \\
            \mathcal{P}_d^\circ \arrow[r, hook] \arrow[ru, "{[\nu]}"] & \mathcal{P}_d \arrow[ru, "{\overline{[\nu]}}"']&
        \end{tikzcd}.
    \end{equation}
\end{lem}

Below we give a description of the quotient sheaf $\overline{\mathrm{Obs}}$ and the corresponding section $\overline{[\nu]}$. 

To save notation, let us denote $\mathcal{N} = (\mathrm{proj}_{\mathbb{P}^1} \circ \mathrm{proj}_{C \times \mathbb{P}^1})^* N$ be the pullback of $N \to \mathbb{P}^1$ along the projection $ \mathcal{P}_d \times (C \times \mathbb{P}^1) \to \mathbb{P}^1$. Recall the description of the obstruction bundle \cref{eqn:obstruction-bundle}, reproduced below:
    \begin{equation}
        \mathrm{Obs} := R^1(\mathrm{proj}_{\mathcal{P}_d})_* \left( \mathcal{O}_{\mathcal{D}} \otimes  \mathcal{N} \right).
    \end{equation}
  Note that there is an exact sequence $0 \to \mathcal{N}(-\mathcal{D}) \to \mathcal{N} \to \mathcal{O}_{\mathcal{D}} \otimes \mathcal{N} \to 0$ coming from the ideal sheaf exact sequence for the universal divisor $\mathcal{D} \subseteq \mathcal{P}_d \times (C \times \mathbb{P}^1)$. Here, $\mathcal{N}(-\mathcal{D})$ denotes the tensor product of $\mathcal{N}$ with the ideal sheaf $\mathcal{O}(-\mathcal{D})$ of $\mathcal{D}$. The long exact sequence for the derived pushforward is then written as
  \begin{equation}
      \cdots \rightarrow R^1(\mathrm{proj}_{\mathcal{P}_d})_* \  \mathcal{N} \rightarrow R^1(\mathrm{proj}_{\mathcal{P}_d})_* \left( \mathcal{O}_{\mathcal{D}} \otimes \mathcal{N} \right) \rightarrow R^2(\mathrm{proj}_{\mathcal{P}_d})_* \left( \mathcal{N}(-\mathcal{D}) \right) \rightarrow R^2(\mathrm{proj}_{\mathcal{P}_d})_* \ \mathcal{N} \rightarrow \cdots.
  \end{equation}
  From Serre duality, we see that $H^2(C \times \mathbb{P}^1; \mathcal{O}(0, -1)) \cong H^0(C \times \mathbb{P}^1; \mathcal{O}(-2, -1))^\vee \cong  0$. From Riemann--Roch, $\chi(C \times \mathbb{P}^1 ; \mathcal{O}(0,-1)) = 0$, it moreover follows that $H^1(C \times \mathbb{P}^1; \mathcal{O}(0,-1)) = 0$. This implies the vanishing $R^1(\mathrm{proj}_{\mathcal{P}_d})_* \  \mathcal{N} \cong R^2(\mathrm{proj}_{\mathcal{P}_d})_* \  \mathcal{N} \cong 0$ and therefore
     \begin{equation}
         \mathrm{Obs} = R^1(\mathrm{proj}_{\mathcal{P}_d})_* \left( \mathcal{O}_{\mathcal{D}} \otimes \mathcal{N} \right) \cong R^2(\mathrm{proj}_{\mathcal{P}_d})_* \left( \mathcal{N}(-\mathcal{D}) \right).
     \end{equation}
     Applying Serre duality, we obtain
    \begin{equation}\label{eqn:obs-dual}
        \mathrm{Obs}^\vee \cong R^0(\mathrm{proj}_{\mathcal{P}_d})_* \left( \mathcal{N}^\vee (\mathcal{D}) \otimes K_{\mathcal{P}_D \times C \times \mathbb{P}^1 / \mathcal{P}_d} \right).
    \end{equation}

Recall the stratification from \cref{prop:thick-section-moduli-stratification} where each stratum $\mathcal{P}_d^\alpha$ corresponds to the locus where the number of vertical bubbles (with their multiplicity) is fixed. As in the discussion following \eqref{eqn:section-of-O(d,1)-concretely}, for a fixed section $s = P_0 Y_0 + P_1 Y_1 \in H^0(\mathcal{O}(d,1))$, the locus of vertical bubbles $\{ P_0 = P_1 = 0 \}$ defines a closed subscheme of the vanishing locus $Z(s) = \{s = P_0 Y_0 + P_1 Y_1 = 0\}$. As $s$ varies, this fits into a closed subscheme of the universal divisor $\mathcal{D}$:

\begin{defn}\label{defn:universal-bubbling-locus}
    The \emph{universal bubbling locus} $\mathcal{B} \subseteq \mathcal{D} \subseteq \mathcal{P}_d \times (C \times \mathbb{P}^1)$ is the union $\bigcup_{s \in \mathcal{P}_d} \{ P_0 = P_1 = 0\}$ with the natural scheme structure. The \emph{universal reduced bubbling locus} is the reduced subscheme $\mathcal{B}_{\mathrm{red}} \subseteq \mathcal{B}$. Note $\mathcal{B}_{\mathrm{red}}$ is naturally a subscheme of $\mathcal{D}_{\mathrm{red}}$, the reduced subscheme of $\mathcal{D}$. The \emph{universal nonreduced bubbling locus} is the union $\mathcal{B}_{\mathrm{nonred}} = \bigcup_{s \in \mathcal{P}_d} \{ P_0 = P_1 = \partial P_0 = \partial P_1 = 0 \}$.
\end{defn}
For a fixed $s \in \mathcal{P}_d$, note that $\mathcal{B}|_{s} \subseteq C \times \mathbb{P}^1$ is a product of $\mathbb{P}^1$ with an effective divisor in $C$, and therefore can be identified by that divisor in $C$. If we write this divisor as $\sum n_i p_i$ for points $p_i \in C$, then $\mathcal{B}_{\mathrm{red}}|_s$ and $\mathcal{B}_{\mathrm{nonred}}|_s$ correspond to $\sum p_i$ and $\sum (n_i -1) p_i$, respectively. 

\begin{rem}\label{rem:notations-from-Voi96}
    There is a change in the notation from that of \cite{Voi96}. The thickened moduli space $\mathcal{P}_d$ is denoted $M_k$ in \cite{Voi96}. In \cite[p.140]{Voi96}, the analogue of our bubbling $\mathcal{B} \subseteq \mathcal{P}_d \times (C \times \mathbb{P}^1)$ is denoted $Z \subseteq M_k \times \mathbb{P}^1$. We have chosen to pull it back so that it becomes a subscheme of $\mathcal{D}$. Our $\mathcal{B}_{\mathrm{nonred}}$ is her $B$. In \cite[p.144]{Voi96}, more notation is introduced: our $\mathcal{B}$ is the universal version of $D'$, and our $\mathcal{B}_{\mathrm{red}}$ is the universal version of $D''$.
\end{rem}

The main technical result of \cite{Voi96} is that for the subsheaf
\begin{equation}\label{eqn:subsheaf-of-obs-dual}
    \mathcal{F} := R^0(\mathrm{proj}_{\mathcal{P}_d})_* \left(\mathcal{N}^\vee(\mathcal{D}) \otimes K_{C \times \mathbb{P}^1} \otimes \mathcal{I}_{\mathcal{B}_{\mathrm{nonred}}}  \right) \\ 
    \hookrightarrow \mathrm{Obs}^\vee \cong R^0 (\mathrm{proj}_{\mathcal{P}_d})_* \left(\mathcal{N}^\vee (\mathcal{D}) \otimes K_{C \times \mathbb{P}^1} \right) 
\end{equation}
where the last isomorphism is \eqref{eqn:obs-dual}, we have that for a section $\phi$ of $\mathcal{F}$, the function given by the pairing $\langle [\nu], \phi \rangle$ \emph{extends continuously to all of $\mathcal{P}_d$}. Then $\overline{\mathrm{Obs}}$ is defined as the dual of $\mathcal{F}$, which is naturally a quotient sheaf of $\mathrm{Obs}^{\vee \vee} \cong \mathrm{Obs}$. The continuous extension $\overline{[\nu]}$ of $[\nu]$ as a section of $\overline{\mathrm{Obs}}$ is defined by the means of $\langle [\nu], \phi \rangle$.

\begin{rem}[{Geometric interpretation of $\overline{\mathrm{Obs}}$}]\label{rem:geometric-interpretation-quotient-obs}
Restricted to the stratum $\mathcal{P}_d^\alpha \subseteq \mathcal{P}_d$, the sheaf $\mathcal{F}|_{\mathcal{P}_d^\alpha}$ (and hence also $\overline{\mathrm{Obs}}|_{\mathcal{P}_d^\alpha}$) has constant rank so is locally free (as a sheaf over $\mathcal{P}_d^\alpha$). Fix $s \in \mathcal{P}_d^\alpha$. Recall that $\{s = 0\} \subseteq C \times \mathbb{P}^1$ has irreducible components corresponding to (i) the graph of a principal component $v_{\mathrm{prin}} : C \to \mathbb{P}^1$ and (ii) the vertical bubbles attached at $p_{k, 1}, \dots, p_{k, \alpha_k}$ with multiplicity $k$ for every $k \ge 1$.  The fiber of the sheaf $\overline{\mathrm{Obs}}$ at $s \in \mathcal{P}_d^\alpha$ is given by
\begin{align}
    \overline{\mathrm{Obs}}|_{s} &\cong  R^0(\mathrm{proj}_{\mathcal{P}_d})_* \left(\mathcal{N}^\vee(\mathcal{D}) \otimes K_{C \times \mathbb{P}^1} \otimes \mathcal{I}_{\mathcal{B}_{\mathrm{nonred}}} \right)^\vee|_s \\
    & \cong H^0 \left( C \times \mathbb{P}^1 , \mathcal{N}^\vee \otimes \mathcal{I}_{\mathcal{D}}^\vee \otimes K_{C \times \mathbb{P}^1} \otimes \mathcal{I}_{\mathcal{B}_{\mathrm{nonred}}} |_{\{s=0\}}\right)^\vee   \\
       & \cong H^0\left( \mathrm{im}(\widetilde{v}_{\mathrm{prin}}), (\mathcal{N}^\vee \otimes \mathcal{I}_{\mathcal{D}}^\vee \otimes K_{C \times \mathbb{P}^1} \otimes \mathcal{I}_{\mathcal{B}_{\mathrm{nonred}}}) \mid_{\mathrm{im}(\widetilde{v}_{\mathrm{prin}})} \right)^\vee \\
    &\cong H^0\left(C,  v_{\mathrm{prin}}^* N^\vee \left(\sum_{k=1}^\infty \sum_{a=1}^{\alpha_k} k p_{k, a}-\sum_{k=1}^\infty \sum_{a=1}^{\alpha_k} (k-1) p_{k, a}\right) \otimes K_C \right)^\vee \\
    &\cong H^1\left(C, v_{\mathrm{prin}}^* N \left( - \sum_{k=1}^\infty \sum_{a=1}^{\alpha_k}  p_{k, a} \right) \right) \quad (\mbox{Serre duality})
\end{align}
The last cohomology group can be computed as the quotient of $C^\infty(C, \mathrm{Hom}^{0,1}(TC, v_{\mathrm{prin}}^* N ))$ by the image under $\overline{\partial}$ of the subspace $C^\infty(C, v_{\mathrm{prin}}^*N)( - \sum p_{k,a}) \subseteq C^\infty(C, v_{\mathrm{prin}}^*N )$ of sections which vanish \emph{only to zeroth order} along the divisor where the vertical bubbles are attached. On the other hand, the fiber of the obstruction bundle $\mathrm{Obs}|_s$ is the quotient of the same space $C^\infty(C, \mathrm{Hom}^{0,1}(TC, v_{\mathrm{prin}}^* N ))$ but by the $\overline{\partial}$ image of sections $C^\infty(C, v_{\mathrm{prin}}^*N)( - \sum k p_{k,a})$ which vanish along the bubbling locus with the correct multiplicity. This gives the geometric description of $\overline{\mathrm{Obs}}$: it still supports the class of the perturbation datum $[\nu]$ carried by the principal component $v_{\mathrm{prin}}: C \to \mathbb{P}^1$, but the class is considered in a cokernel of $\overline{\partial}$ from a larger collection of sections.
\end{rem}

\begin{proof}[{Proof sketch of \cref{thm:noneq-section-compactification}, given \cref{lem:extension-of-obs-section}}]

    (For further details, see \cite[Proposition 3.6, Proposition 3.8]{Voi96}.)

    By choosing a Hermitian metric, we can choose a complex antilinear isomorphism of vector bundles $\mathrm{Obs} \cong \mathrm{Obs}^\vee$. Note that the subsheaf $\mathcal{F}$ of \eqref{eqn:subsheaf-of-obs-dual} injects into $\mathrm{Obs}^\vee$. Fix a surjection $E \to \mathcal{F}$ from a vector bundle $E$. After choosing a Hermitian metric, there exists a complex antilinear isomorphism of vector bundles $E \cong E^\vee$. We obtain the diagram
    \begin{equation}\label{eqn:extending-obs-section}
    \begin{tikzcd}
        & \mathrm{Obs}^\vee \dar["\cong", sloped] & \mathcal{F} \lar[hook'] & E \lar[twoheadrightarrow] \\
        & \mathrm{Obs} \rar[twoheadrightarrow] \dar \arrow[rr, bend left=30] & \overline{\mathrm{Obs}} & E^\vee \uar["\cong", sloped] \\
        \mathcal{P}_d^\circ \rar[hook] \arrow[ur, "{[\nu]}"] & \mathcal{P}_d \arrow[ur, "{\overline{[\nu]}}"'] 
    \end{tikzcd}.
    \end{equation}
    Here the map $\mathrm{Obs} \to E^\vee$ is obtained by dualizing the composition $E \to \mathcal{F} \to \mathrm{Obs}^\vee$. The monodromy around the upper right rectangle gives a complex linear endomorphism $\Theta: \mathrm{Obs} \to \mathrm{Obs}$.

    It follows from \cref{lem:extension-of-obs-section} that $\Theta[\nu]$ extends to a global continuous section of $\mathrm{Obs}$ over $\mathcal{P}_d$. Moreover, we see that the vanishing locus of $\Theta[\nu]$ is exactly the vanishing locus of $\overline{[\nu]}$, as $\overline{[\nu]}$ is by definition $\langle -, \Theta[\nu]^\vee \rangle \in \mathcal{F} \cong \overline{\mathrm{Obs}}$. From the description in \cref{rem:geometric-interpretation-quotient-obs}, we see that over the stratum $\mathcal{P}_d^\alpha$, $\overline{[\nu]}$ defines a section of the vector bundle $\mathrm{Obs}^\alpha \to \mathcal{P}_d^\alpha$, where the fiber of $\mathrm{Obs}^\alpha = \overline{\mathrm{Obs}}|_{\mathcal{P}_d^\alpha}$ at $s \in \mathcal{P}_d^\alpha$ is exactly given by $H^1(C, v_{\mathrm{prin}}^*N ( - \sum p_{k, a}))$. Hence, the vanishing locus of $\Theta[\nu]$, equivalently the vanishing locus of $\overline{[\nu]}$, is stratified into subsets $\mathcal{W}_d^\alpha := \{ \overline{[\nu]} = 0 \} \cap \mathcal{P}_d^\alpha$.

    We can compute
    \begin{equation}\label{eqn:dim-of-strata-alpha}
        \dim_{\mathbb{R}} \mathcal{P}_d^\alpha = 2 \left( 2 \deg (v_{\mathrm{prin}}) +1 + \sum \alpha_k \right) = 2 \left( 2 \left( d - \sum_k k \alpha_k\right) +1 + \sum_k \alpha_k \right)
    \end{equation}
    where $\sum \alpha_k$ is the number of marked points where bubbles can be attached. Also note that
    \begin{equation}\label{eqn:rank-of-obs-alpha}
        \mathrm{rank}_{\mathbb{R}} \ \mathrm{Obs}^\alpha = 2 \left( (2d-2) - 2 \sum_{k} (k-1) \alpha_k \right)
    \end{equation}
    since $C^\infty(C, v_{\mathrm{prin}}^*N)( - \sum k p_{k,a}) \subseteq C^\infty(C, v_{\mathrm{prin}}^*N)( - \sum p_{k,a})$ is an inclusion of complex codimension $2\sum_k (k-1)\alpha_k$. In particular,
    \begin{equation}
        \dim_{\mathbb{R}} \mathcal{P}_d^\alpha - \mathrm{rank}_{\mathbb{R}} \ \mathrm{Obs}^\alpha = \dim_{\mathbb{R}} \mathcal{V}_d - 2 \sum \alpha_k \le \dim_{\mathbb{R}} \mathcal{V}_d - 2.
    \end{equation}

    If $\nu$ is chosen generically, then for every $\alpha \neq 0$ the stratum $\mathcal{W}_d^\alpha = \{ \overline{[\nu]} = 0 \} \cap \mathcal{P}_d^\alpha$ is a smooth submanifold of $\mathcal{P}_d^\alpha$ of real dimension $\le \mathrm{dim}_{\mathbb{R}} \mathcal{V}_d - 2$. Since $\mathcal{V}_d = \{ [\nu] = 0 \} \subseteq \{ \overline{[\nu]} = 0 \}$, it follows that the closure $\overline{\mathcal{V}}_d \subseteq \{ \overline{[\nu]} = 0 \}$. Setting $\mathcal{V}_d^\alpha := \overline{\mathcal{V}}_d \cap \mathcal{W}_d^\alpha$ establishes (ii) and (iii). 
    
    For (iv), it remains to note that the homology class of $[\overline{\mathcal{V}}_d]$ is also represented by the homology class of the zero locus of a smooth perturbation $[\nu]'$ of $\Theta[\nu]$. Then $[\nu]'$ is a smooth section of $\mathrm{Obs} \to \mathcal{P}_d$, so the zero locus represents the Poincar\'e dual of the top Chern class $c_{\mathrm{top}} (\mathrm{Obs})$, as desired.
\end{proof}

\subsubsection{Section moduli space: Incidence constraints}

For the map moduli space $\mathcal{M}_A$, the incidence constraints were implemented as \eqref{eqn:noneq-cycle} by choosing a submanifold $\mathcal{Y} \subseteq X \times X^p \times X$. The map moduli spaces are rigidified by considering the evaluation map from the moduli space to the target $X \times X^p \times X$ and taking the preimage of $\mathcal{Y}$ assuming (strong) transversality; see \cref{defn:cycle-strongly-transverse}.

For the section moduli space $\mathcal{V}_d$, the following viewpoint provides a different method to implement incidence constraints. As in \eqref{eqn:noneq-cycle}, choose an oriented submanifold
\begin{equation}
    \mathcal{Y} = Y_0 \times (Y_1 \times \cdots \times Y_p) \times Y_\infty \subseteq X \times X^p \times X
\end{equation}
as the (non-equivariant) incidence cycle. For our example $X = \mathrm{Tot}(\pi : \mathcal{O}(-1)^{\oplus 2} \to \mathbb{P}^1)$, we will fix each $Y_j$ (for $j \in \{0, 1, \dots, p, \infty\}$) to be either $X$ (the total space) or $\pi^{-1}(y) \subseteq X$ for $y \in \mathbb{P}^1$ (the fiber of $\pi$) as embedded submanifolds. The following is an immediate consequence of our definition of decoupling perturbation data $\nu$.

\begin{lem}\label{lem:incidence-on-X-same-as-incidence-on-P1}
    Fix $z \in C$ and $u = (v, \phi) \in \mathcal{M}_A \cong \mathcal{V}_A$. Fix $y \in \mathbb{P}^1$ and the corresponding fiber $Y = \pi^{-1}(y) \subseteq X$. Then $u(z) \in Y$ if and only if $v(z) = y$.
\end{lem}
\begin{proof}
    Follows from $v = \pi \circ u$.
\end{proof}

The condition $v(z) = y$ for a fixed $y \in \mathbb{P}^1$ can be rewritten as follows.

\begin{lem}\label{lem:point-constraint-as-line-bundle}
    Fix $z \in C$, $y \in \mathbb{P}^1$ and let $s \in \mathcal{P}_d$ be a section of $\mathcal{O}(d,1) \to C \times \mathbb{P}^1$. Then $(z,y)$ determines a smooth section $\mathrm{ev}_{(z,y)}$ of $\mathcal{O}(1) \to \mathcal{P}_d$ such that $(z, y) \in \{ s= 0 \} \subseteq C \times \mathbb{P}^1$ if and only if $\mathrm{ev}_{(z,y)}$ vanishes at $s$.
\end{lem}
\begin{proof}
    Define the linear functional to be
    \begin{align}
        \label{eqn:incidence-section}
        \mathrm{ev}_{(z,y)}: H^0(C \times \mathbb{P}^1, \mathcal{O}(d,1)) &\to \mathcal{O}(d,1)|_{(z,y)} .\\
        s &\mapsto s(z,y)
    \end{align}
    Considering this as an element of $H^0(C \times \mathbb{P}^1, \mathcal{O}(d,1))^\vee$ involves a choice of an isomorphism $\mathcal{O}(d,1)|_{(z,y)} \cong \mathbb{C}$. Nevertheless, the vanishing locus of $\mathrm{ev}_{(z,y)}$ is well-defined independent of the choice, and indeed $\mathrm{ev}_{(z,y)}$ defines a well-defined section of the hyperplane bundle over $\mathcal{P}_d = \mathbb{P}H^0(\mathcal{O}(d,1))$. Clearly $s(z,y) = 0$ if and only if $\mathrm{ev}_{(z,y)}([s])=0$, where $[s]$ is the point in $\mathcal{P}_d$ corresponding to $s$.
\end{proof}

Fix $\mathcal{Y} = Y_0 \times (Y_1 \times \cdots \times Y_p) \times Y_\infty$, such that each $Y_j$ is either the total space $X$ or fiber cycles of the form $\pi^{-1}(y_j)$ for $y_j \in \mathbb{P}^1$. For each $j \in \{0, 1, \dots, p, \infty\}$, let
\begin{equation}
    \label{eqn:incidence-line-bundle}
    \mathcal{L}_j := \begin{cases} 0 & \mbox{ if } Y_j \cong X \\ \mathcal{O}_{\mathcal{P}_d}(1) & \mbox{ if } Y_j \cong \pi^{-1}(y_j) \mbox{ for } y_j \in \mathbb{P}^1 \end{cases}.
\end{equation}

\begin{defn}
    \label{defn:incidence-bundle}
    Fix $\mathcal{Y} \subseteq X \times X^p \times X$ and let $\mathcal{L}_j$ be as in \eqref{eqn:incidence-line-bundle}. The \emph{(non-equivariant) incidence constraint bundle over the section moduli space} is the vector bundle over $\mathcal{P}_d$ defined as
    \begin{equation}\label{eqn:incidence-bundle}
        \mathrm{IC} = \mathrm{IC}_d := \mathcal{L}_0 \oplus \left( \mathcal{L}_1 \oplus \cdots \oplus \mathcal{L}_p \right) \oplus \mathcal{L}_\infty.
    \end{equation}
\end{defn}
Consider all tuples $(z_j, y_j)$ for $j \in \{ 0, 1, \dots, p, \infty \}$ such that $Y_j$ is a fiber over $y_j$ and $z_j$ are the distinguished marked points on $C$. The associated sections $\mathrm{ev}_{(z_j, y_j)}$ of $\mathcal{L}_j$ together determine a global section $\mathrm{ev}_{\mathcal{Y}}: \mathcal{P}_d \to \mathrm{IC}$ of the incidence constraint bundle $\mathrm{IC}$.

The curves $u=(v,\phi) \in \mathcal{V}_d$ which pass through the fibers $Y_j$ at points $z_j \in C$ are identified with the intersection of the zero locus $\mathrm{ev}_{\mathcal{Y}}^{-1}(0)$ and $\mathcal{V}_d$. Using the homological characterization $[\overline{\mathcal{V}}_d] = c_{\mathrm{top}}(\mathrm{Obs})$ from \cref{thm:noneq-section-compactification}, one can compute the number of curves satisfying the incidence constraint by computing the integral of Chern classes coming from and the incidence constraint bundle:

\begin{lem}\label{lem:section-counts-equal-to-Chern-integral}
    Suppose $\mathcal{Y} = Y \times (Y_1 \times\cdots \times Y_p) \times Y$ is chosen so that $\mathrm{ev}_{\mathcal{Y}}^{-1}(0)$ intersects every $\mathcal{V}_d^\alpha$ transversely and that the intersection only happens along $\mathcal{V}_d \subseteq \overline{\mathcal{V}}_d$. Then
    \begin{equation}
        \# \mathcal{V}_d \pitchfork \mathrm{ev}_{\mathcal{Y}}^{-1}(0) = \int_{\mathcal{P}_d} c_{\mathrm{top}}(\mathrm{Obs}) \cup c_{\mathrm{top}}(\mathrm{IC}). 
    \end{equation}
\end{lem}

\begin{rem}
    In \cite{AM93} and \cite{Voi96}, this lemma is executed where $\mathcal{Y}$ represents three point constraints on the base $\mathbb{P}^1$, and hence yields (a form of) the celebrated multiple cover formula,
    \begin{equation}
        \langle \mathrm{pt}, \mathrm{pt}, \mathrm{pt} \rangle_{0, 3, d [\mathbb{P}^1]} = \int_{\mathcal{P}_d} c_{\mathrm{top}}(\mathrm{Obs})  \cup c_1(\mathcal{O}(1))^3 = \int_{\mathbb{P}^{2d+1}} H^{2d-2} \cdot H^{3} =  1. 
    \end{equation}
\end{rem}

The following proposition implies that, if the relevant moduli spaces are rigidified to be $0$-dimensional, then the counts from the map moduli spaces and section moduli spaces must be equal.

\begin{prop}\label{prop:transversality-is-equal}
    Fix the incidence cycle $\mathcal{Y} = Y_0 \times (Y_1 \times \cdots \times Y_p) \times Y_\infty$ such that every $Y_j$ for $j \in \{ 0, 1, \dots, p, \infty\}$ is either the total space $Y_j = X$ or a fiber cycle, $Y_j := \pi^{-1}(y_j)$ for $y_j \in \mathbb{P}^1$. Assume the moduli spaces $\mathcal{M}_A$ and $\mathcal{V}_A = \mathcal{V}_d$ are both regular. Then the evaluation map $\mathrm{ev}: \mathcal{M}_A \to X \times X^p \times X$ is transverse to the inclusion of the incidence cycle $\mathcal{Y} \hookrightarrow X \times X^p \times X$ if and only if the section $\mathrm{ev}_{\mathcal{Y}} : \mathcal{P}_d \to \mathrm{IC}$ is transverse to $\mathcal{V}_d$ (in the zero section of $\mathrm{IC}$). 
\end{prop}
\begin{proof}
    Let us assume for simplicity that all $Y_j$ are fiber cycles.
    
    Set-theoretically, \cref{lem:incidence-on-X-same-as-incidence-on-P1} and \cref{lem:point-constraint-as-line-bundle} shows that $\mathcal{M}_A \cap \mathrm{ev}^{-1}(\mathcal{Y}) \cong \mathcal{V}_A \cap \mathrm{ev}_{\mathcal{Y}}^{-1}(0)$.

    Assume the transversality in the sense of map moduli: $\mathcal{M}_A \pitchfork \mathrm{ev}^{-1}(\mathcal{Y})$. Fix $s = s_v \in \mathcal{V}_A \cap \mathrm{ev}_{\mathcal{Y}}^{-1}(0)$, which cuts out the graph of $v: C \to \mathbb{P}^1$. Fix $s_t$ for $t \in (-\varepsilon, \varepsilon)$, a $1$-parameter family of sections such that $s_0 = s$ with $\partial_t s_t|_{t=0} \in T_s \mathcal{V}_A$. For sufficiently small $\varepsilon$, all $s_t$ are graphical: $s_t = s_{v_t}$ for a $1$-parameter family of holomorphic curves $v_t : C \to \mathbb{P}^1$ of degree $d$, such that $v_0 = v$.

    Consider
    \begin{equation}
        \mathrm{ev}_{\mathcal{Y}}(s_{v_t}) = \left(s_{v_t}(z_0, y_0), \bigoplus_{j=1}^p s_{v_t} (z_j, y_j), s_{v_t}(z_\infty, y_\infty) \right).
    \end{equation}
    We will show that transversality in the sense of section moduli holds at $s$. First consider the first marked point $z_0 \in C$ and the fiber cycle $Y_0 = \pi^{-1}(y_0)$.

    Since $s_v$ satisfies the incidence constraint, $y_0 = v_0(z_0)$.     Recall that $\mathrm{ev}_{(z_0,y_0)} (s) = s(z_0, y_0) = 0 \in \mathcal{O}(d,1)|_{z_0, y_0}$. For transversality, we need to show that the derivative of $\mathrm{ev}_{\mathcal{Y}}$ surjects onto $\mathcal{L}_0|_s \cong \mathcal{O}_{\mathcal{P}_d}(1)|_{s} \cong \mathcal{O}_{C \times \mathbb{P}^1}(d,1)|_{(z_0, y_0)}$, i.e. $\partial_t s_{v_t}|_{t=0} \in T_s \mathcal{V}_A \mapsto \partial_t s_{v_t} (z_0, y_0)$ (over all choices of $v_t$) surjects onto $\mathcal{O}(d,1)|_{(z_0, y_0)}$.

    Fix a connection $\nabla$ on the line bundle $\mathcal{O}(d,1) \to C \times \mathbb{P}^1$. One has the tautological identity $s_{v_t}(z_0, v_t(z_0)) = 0$, as $s_{v_t}$ cuts out the graph of $v_t$. Differentiating this in $t$ gives
    \begin{equation}
        \partial_t s_{v_t} (z_0, y_0) + \nabla_t s_{v_0} (z_0, v_t(z_0))|_{t=0} = 0.
    \end{equation}

    Using the identification from \cref{prop:noneq-map-equal-section}, the family of curves $v_t : C \to \mathbb{P}^1$ associates a (holomorphic) vector field $\xi \in C^\infty(C, v^*T\mathbb{P}^1)$. Observe that for the vertical differential $\mathrm{d} s_v^{\mathrm{vert}} : T_{z_0} C \times T_{y_0} \mathbb{P}^1 \to \mathcal{O}(d,1)|_{(z_0,y_0)}$ (which is defined even without a choice of a connection, since $s_v(z_0, y_0) = 0$), we have the identity
    \begin{equation}
        \nabla_t s_v(z_0, v_t(z_0))|_{t=0} = \mathrm{d} s_v^{\mathrm{vert}} (0, \xi(z_0)).
    \end{equation}
    Hence it suffices to show that $\mathrm{d}s_v^{\mathrm{vert}}(0, \xi(z_0))$, as $\xi$ ranges over all vector fields coming from all choices of $v_t$, surjects onto $\mathcal{O}(d,1)|_{z_0,y_0}$.
    
    Because each $Y_j = \pi^{-1}(y_j)$ is a fiber cycle, the transversality assumption that $\mathrm{ev}: \mathcal{M}_A \to X \times X^p \times X$ is transversal to the inclusion of $(Y_0; Y_1, \dots, Y_p; Y_\infty)$ is equivalent to the claim that $\pi \circ \mathrm{ev} : \mathcal{M}_A \to \mathbb{P}^1 \times (\mathbb{P}^1)^p \times \mathbb{P}^1$ is transversal to the inclusion of the point $(y_0; y_1, \dots, y_p ; y_\infty)$. The derivative of $\pi \circ \mathrm{ev}: (v, \phi) \mapsto (v(z_0); v(z_1), \dots, v(z_p); v(z_\infty))$ at $(v,\phi) \in \mathcal{M}_A \cap \mathrm{ev}^{-1}(\mathcal{Y})$ is given as
    \begin{align}
        \mathrm{d} (\pi \circ \mathrm{ev}) : T_{(v, \phi)} \mathcal{M}_A &\to T_{y_0}\mathbb{P}^1 \oplus \bigoplus_{j=1}^p T_{y_j} \mathbb{P}^1 \oplus T_{y_\infty} \mathbb{P}^1. \\
        (\xi, \psi) &\mapsto \left( \xi (z_0); \xi(z_1), \dots, \xi(z_p);  \xi(z_\infty) \right) 
    \end{align}
    Therefore transversality implies that for any $\xi_0 \in T_{y_0} \mathbb{P}^1$, there exists $(\xi, \psi) \in T_{(v,\phi)} \mathcal{M}_A$ such that $\xi(z_0) = \xi_0$. Hence the vector fields $\xi \in C^\infty(\mathbb{P}^1, v^*T\mathbb{P}^1)$ coming from $v_t : C \to \mathbb{P}^1$ can take any value in $T_{y_0} \mathbb{P}^1$. We are reduced to showing that $\mathrm{d}s_v^{\mathrm{vert}}|_{0 \oplus T_{y_0} \mathbb{P}^1}$ is surjective.

    But $s_v$ is a holomorphic section, so $\mathrm{d}s_v$ is in fact complex linear. Hence it suffices to show that $\mathrm{d} s_v|_{0 \oplus T\mathbb{P}^1} = \mathrm{d} (s_v |_{z_0 \times \mathbb{P}^1})$ is nonzero. But $s_v|_{z_0 \times \mathbb{P}^1}$ is a linear form on $\mathbb{P}^1$, so $\mathrm{d}(s_v|_{z_0 \times \mathbb{P}^1})$ is nonzero unless $s_v |_{z_0 \times \mathbb{P}^1}$ is constantly zero. This only happens if there is a vertical bubble attached at $z_0$, which never happens over $\mathcal{V}_A \subseteq \mathcal{P}_A^\circ$. Therefore $\mathrm{d}s_v^{\mathrm{vert}}|_{0 \oplus T_{y_0}\mathbb{P}^1}$ is surjective (in fact an isomorphism of complex dimension $1$ vector spaces), as desired.

    Other points $(z_j, y_j)$ can be treated in the same way, and from this follows the transversality in the section moduli sense $\mathcal{V}_A \pitchfork \mathrm{ev}_{\mathcal{Y}}^{-1}(0)$. The converse direction (from $\mathcal{V}_A \pitchfork \mathrm{ev}_{\mathcal{Y}}^{-1}(0)$ to $\mathcal{M}_A \pitchfork \mathrm{ev}^{-1}(\mathcal{Y})$) also follows the same argument, just in reverse order. The proof is the same when any $Y_j = \pi^{-1}(y_j)$ is replaced with $X$, with corresponding summands in vector bundles deleted.
\end{proof}

    \cref{prop:transversality-is-equal} is not strictly necessary to define the counts of $0$-dimensional counts, because $\mathcal{Y}$ can be perturbed so that it is transversal to $\mathcal{M}_A \cong \mathcal{V}_A$ in both senses. Rather, it shows that if it is transversal in one sense, it is automatically transversal in the other sense.

\section{The local $\mathbb{P}^1$ Calabi--Yau 3-fold ($\mathbb{Z}/p$-Equivariant)}\label{sec:eq-localP1}
We now introduce the equivariant versions of the constructions in the previous section, including the decoupling perturbation data. This is analogous to the construction of equivariant moduli spaces in \cref{ssec:eq-moduli-spaces}. Most constructions are verbatim, being parametrized versions of the ordinary construction. We therefore choose to focus on the aspects that newly arise in the equivariant setting.

Once we develop the equivariant versions of our previous constructions, we could fully compute the structure constants for the quantum Steenrod operations (loosely interpreted, due to non-compactness issues). This is done in the final subsection.

\subsection{Equivariant decoupling perturbation data and map moduli spaces}\label{ssec:eq-decoupling-nu}
The equivariant analogue of the decoupling perturbation data \eqref{noneq-nu-dc} is
\begin{equation}
    \nu^{eq}  \in C^\infty \left( S^\infty \times_{\mathbb{Z}/p} (C \times \mathbb{P}^1), \mathrm{Hom}^{0,1 }(TC, N)\right). 
\end{equation}

As earlier, equivariant data are maybe best understood as parametrized data over $B\mathbb{Z}/p \simeq S^\infty / (\mathbb{Z}/p)$. For $w \in S^\infty$, let
\begin{equation}
\nu_w^{eq} = (w \times \mathrm{id}_{C \times \mathbb{P}^1})^* \nu^{eq} \in C^\infty \left(C \times \mathbb{P}^1, \mathrm{Hom}^{0,1}(TC, N) \right),
\end{equation}
and this satisfies the equivariance relation $\nu^{eq}_{\sigma \cdot w} = \sigma_C^* \  \nu^{eq}_{w}$, where we recall that $\sigma_C$ denotes the rotation action the marked curve $C$ defined in \cref{ssec:eq-moduli-spaces}.

As before, using equivariant perturbation data (pulled back from equivariant decoupling perturbation data), we can define the equivariant moduli spaces. Fix $A = d [\mathbb{P}^1] \in H_2(X;\mathbb{Z}) $, and an equivariant degree $i \ge 0$.
\begin{defn}[Compare \cref{defn:eq-map-moduli}, \eqref{eqn:dc-moduli}]
\label{defn:eq-map-dc-moduli}
The \emph{$i$th equivariant map moduli space of degree $A$} is the space
    \begin{equation}
        \mathcal{M}_A^{eq, i} := \left\{ \left(w \in \Delta^i, v: C \to \mathbb{P}^1, \phi \in C^\infty(C, v^*N) \right) : \ \overline{\partial}_J v = 0, \ \overline{\partial} \phi  = (w \times \widetilde{v})^* \nu^{eq} , \ v_*[C] = A \right\}.
    \end{equation}
of solutions using equivariant decoupling perturbation data.
\end{defn}

With the definition of equivariant versions of our previous constructions, we can now discuss the equivariant versions of the issues discussed in the non-equivariant case: regularity, compactness, and incidence constraints.

Fix a reference equivariant decoupling perturbation data $\nu_0^{eq}$ which is a pullback of a genuinely $\mathbb{Z}/p$-invariant decoupling perturbation data. Consider its neighborhood $\mathcal{S}^{eq} \ni \nu_0^{eq}$ in the $C^\infty$-topology. (Recall that the topology is given in terms of the restrictions to finite dimensional approximations, see \cref{defn:eq-perturbation-data}.) 

\begin{lem}[Compare \cref{lem:dc-map-transverse}]
\label{lem:eq-dc-map-transverse}
    There exists a comeager subset $\mathcal{S}^{eq, \mathrm{reg}} \subseteq \mathcal{S}^{eq}$, such that for equivariant decoupling perturbation data $\nu^{eq} \in \mathcal{S}^{eq, \mathrm{reg}}$, every $\mathcal{M}_A^{eq, i}$ for $i \ge 0$ is regular of dimension $i + \dim_{\mathbb{R}} X + 2 c_1(A) = i + 6$.
\end{lem}
\begin{proof}
    The proof follows the same strategy as in the construction of equivariantly strongly transverse incidence cycles, \cref{lem:eq-transverse-incidence-cycles} and \cref{rem:equivariant-genericity-cycle}. Namely, we start from a perturbation that makes $\mathcal{M}_A^{eq, 0}$ regular (\cref{lem:dc-map-transverse}), and continue inductively by (i) first extending the data constantly in the neighborhood of $\partial \overline{\Delta}^i \subseteq \overline{\Delta}^i$ and (ii) then extending into the interior $\Delta^i \subseteq \overline{\Delta}^i$ using parametrized transversality.

    Step (ii) involves a parametrized version of \cref{lem:dc-map-transverse}. This is straightforward, with the only difference being that the universal moduli space is given by
    \begin{equation}
        \mathcal{M}_A^{univ} = \left\{ (w \in \Delta^i, v, \phi, \nu^{eq, i} \in \mathcal{S}^{eq, i}) : \ \overline{\partial}_J v = 0, \ \overline{\partial} \phi = (w \times \widetilde{v})^* \nu^{eq, i}, \ v_*[C] = A \right\}
    \end{equation}
    where $\mathcal{S}^{eq, i}$ is a neighborhood of $\nu_0^{eq, i} = \nu_0^{eq}|_{\Delta^i \times C \times \mathbb{P}^1}$. The linearization at $(w, v, \phi, \nu^{eq, i})$ is given by
    \begin{align}
    \label{universal-map-moduli-eq-linearization}
    D^{univ, i}: T_w \Delta^i \oplus \Gamma(v^* T \mathbb{P}^1) \oplus \Gamma(v^* N) \oplus T_{\nu^{eq, i}} \mathcal{S}^{eq, i} &\to \Gamma^{0,1}(v^* T \mathbb{P}^1) \oplus \Gamma^{0,1}(v^* N) \\
    \left(\eta, \xi, \psi, \alpha \right) &\mapsto \left( D_v \xi, \  \overline{\partial} \psi - \nabla_{\eta} \nu^{eq, i} - \nabla_\xi \nu^{eq, i} + (w \times \widetilde{v})^* \alpha  \right).
    \end{align}
    (We have suppressed the notation $C^\infty(C, E)$, $C^\infty(C, \mathrm{Hom}^{0,1}(TC, E))$ to $\Gamma(E)$, $\Gamma^{0,1}(E)$ respectively for complex vector bundles $E \to C$.) 

    Indeed there is a decoupling of this linearized operator of the form (cf. \cref{eqn:universal-linearized-decouples})
    \begin{equation}
    D^{univ, i} = \begin{pmatrix} D_1 & 0 \\ D' & D_2 \end{pmatrix},
    \end{equation}
    where
    \begin{align}
    D_1 : T_w \Delta^i \oplus \Gamma(v^*T\mathbb{P}^1) \to \Gamma^{0,1}(v^*T\mathbb{P}^1), \quad & D_1(\eta, \xi)= D_v\xi \\
    D' : T_w \Delta^i \oplus \Gamma(v^*T\mathbb{P}^1) \to \Gamma^{0,1}(v^*N), \quad  & D'(\eta, \xi)= - \nabla_{\eta} \nu^{eq, i} -\nabla_{\xi} \nu^{eq, i}, \\
    D_2 : \Gamma(v^*N) \oplus T_{\nu^{eq, i}} \mathcal{S}^{eq, i} \to \Gamma^{0,1} (v^*N), \quad & D_2(\psi, \alpha) = \overline{\partial} \psi +(w \times \widetilde{v})^* \alpha .  
    \end{align}
    
    The proof of \cref{lem:dc-map-transverse} then applies verbatim, and provides a comeager subset $\mathcal{S}^{eq, i, \mathrm{reg}} \subseteq \mathcal{S}^{eq, i}$ such that for $\nu^{eq, i} \in \mathcal{S}^{eq, i, \mathrm{reg}}$, the corresponding moduli space $\mathcal{M}_A^{eq, i}$ is regular. As in \cref{rem:equivariant-genericity-cycle}, the preimage under the restriction map of the subsets $\mathcal{S}^{eq, i, \mathrm{reg}}$ can be intersected to define $\mathcal{S}^{eq, \mathrm{reg}} \subseteq \mathcal{S}^{eq}$. 
\end{proof}

Recall that the equivariant moduli spaces $\mathcal{M}_A^{eq, i}$ still carry evaluation maps associated to the $(p+2)$ special marked points:
\begin{align}
\mathrm{ev}^{eq, i}: \mathcal{M}^{eq, i}_A &\to X \times \left(\overline{\Delta}^i \times  X^p\right) \times X \\
(w, u) &\mapsto \left( u(z_0), (w, u(z_1), \dots, u(z_p)), u(z_\infty) \right).
\end{align}

\begin{lem}[Compare \cref{prop:gromov-decoupling}]\label{lem:eq-dc-stable-map-transverse}
    For a generic choice of equivariant decoupling perturbation data $\nu^{eq}$, the limit set of $\mathrm{ev}^{eq, i}$ is covered by $\Delta^i$-parametrized analogues of the simple stable maps
    \begin{equation}
        \mathcal{M}^{eq, i}_{T, \{A_b'\}} (X;\nu^{eq})(-\Sigma z_j'),
    \end{equation}
    together with moduli spaces of lower equivariant degrees,
    \begin{equation}
        \bigsqcup_{k=1}^{p} \sigma^k \mathcal{M}^{eq, j}_A \mbox{ and } \bigsqcup_{k=1}^p \sigma^k \mathcal{M}^{eq, j}_{T, \{A_b'\}} (X;\nu^{eq})(-\Sigma z_j')
    \end{equation}
    for $j < i$.
\end{lem}

For the incidence constraints, fix $\mathcal{Y} = Y_0 \times (Y_1 \times \cdots \times Y_p) \times Y_\infty \subseteq X \times X^p \times X$ so that each $Y_j$ is either diffeomorphic to $X$ or $\pi^{-1}(y_j)$ for $y_j \in \mathbb{P}^1$. For equivariance to hold, assume that $Y_1 = \cdots = Y_p$ are all fibers, $Y_j = \pi^{-1}(y_j)$. By \cref{lem:eq-transverse-incidence-cycles}, there is an equivariant incidence cycle $\mathcal{Y}^{eq}: S^\infty \times \mathcal{Y} \to X \times (S^\infty \times X^p) \times X$ with each $\mathcal{Y}^{eq}_w$ being a perturbation of $\mathcal{Y}$, such that $\mathcal{Y}^{eq, i} := \mathcal{Y}^{eq}|_{\Delta^i \times \mathcal{Y}}$ is transverse to evaluation maps from equivariant map moduli spaces (and simple stable maps). 

One can in fact do better, by requiring that every cycle $Y_j(w) = \mathrm{proj}_j \circ \mathcal{Y}^{eq}_w  : Y_j \to X$ is not only isotopic in $X$ to a fiber $\pi^{-1}(y_j) \subseteq X$ but is also a fiber of the form $\pi^{-1}(y_j(w)) \subseteq X$ itself, where $y_j(w) \in \mathbb{P}^1$ is a $S^\infty$-dependent family of points in $\mathbb{P}^1 = \pi(X)$. This is not strictly necessary for rigidifying the map moduli spaces, but will help in implementing the equivariant incidence constraints in the discussion of section moduli spaces.

\begin{lem}[See \cref{defn:eq-cycle}, \cref{defn:eq-strong-transverse}, compare \cref{lem:eq-transverse-incidence-cycles}]
\label{lem:eq-incidence-cycle-only-fibers}
    Suppose that the incidence cycle $\mathcal{Y} = Y_0 \times (Y_1 \times \cdots \times Y_p) \times Y_\infty \subseteq X \times X^p \times X$ satisfies that $Y_1 = \cdots = Y_p = \pi^{-1}(y_1)$ are all fibers. Then there is an equivariantly strongly transverse equivariant incidence cycle $\mathcal{Y}^{eq}$ such that every $\mathcal{Y}^{eq}_w$ is of the form $Y_0(w) \times Y_1(w) \times \cdots \times Y_p(w) \times Y_\infty(w) \subseteq X \times X^p \times X$, where each $Y_j(w) = \pi^{-1}(y_j(w))$ is a fiber of $y_j(w) \in \mathbb{P}^1$.
\end{lem}
\begin{proof}
    As in \cref{lem:eq-transverse-incidence-cycles}, we will apply \cref{lem:eq-transversality} to $A^{eq} := \bigcup_{i \ge 0} \bigcup_{k=1}^p \sigma^k \mathcal{M}_A^{eq,i}$ and $(A')^{eq} := \bigcup_{i \ge 0} \bigcup_{k=1}^p \sigma^k \partial\overline{\mathcal{M}}_A^{eq,i}$. Take $e^{eq}: A^{eq} \to \mathbb{P}^1 \times (S^\infty \times (\mathbb{P}^1)^p) \times \mathbb{P}^1$ and $(e')^{eq} : (A')^{eq} \to \mathbb{P}^1 \times (S^\infty \times (\mathbb{P}^1)^p) \times \mathbb{P}^1$ this time to be the evaluation maps composed with the projection $\pi: X \to \mathbb{P}^1$.

    Take $f: \mathrm{pt} \to \mathbb{P}^1 \times (\mathbb{P}^1)^p \times \mathbb{P}^1$ to be the fixed incidence cycle $\mathcal{Y}$ projected down, i.e. the inclusion of the point $(y_0, y_1, \dots, y_p, y_\infty)$ (where $y_1 = \cdots = y_p$). Then \cref{lem:eq-transversality} provides a choice of $f^{eq} : S^\infty \to \mathbb{P}^1 \times (S^\infty \times (\mathbb{P}^1)^p) \times \mathbb{P}^1$ (i.e. an $S^\infty$-dependent family of points) such that for each $i \ge 0$,  $f^{eq, i}$ is transverse to both (projected) evaluation maps $e^{eq, i}$, $(e')^{eq, i}$ from the equivariant moduli spaces and the simple stable maps.

    As in the proof of \cref{prop:transversality-is-equal}, the usual evaluation map $\mathrm{ev}: \mathcal{M}_A \to X \times X^p \times X$ is transversal to the inclusion of $(Y_0; Y_1, \dots, Y_p; Y_\infty)$ for fiber cycles $Y_j = \pi^{-1}(y_j)$ if and only if the projected evaluation map $\pi \circ \mathrm{ev} : \mathcal{M}_A \to \mathbb{P}^1 \times (\mathbb{P}^1)^p \times \mathbb{P}^1$ is transversal to the inclusion of the point $(y_0; y_1, \dots, y_p ; y_\infty)$. Hence by taking $\mathcal{Y}^{eq}_w \subseteq X \times X^p \times X$ to be given as $Y_0(w) \times Y_1(w) \times Y_p(w) \times Y_\infty(w)$ where $Y_j (w) = \pi^{-1}(\mathrm{proj}_j \circ f^{eq}(w))$ we obtain an equivariant incidence cycle satisfying the desired conditions.
\end{proof}

\subsection{Equivariant section moduli spaces}\label{ssec:eq-section-moduli}
The section moduli spaces of \cref{defn:noneq-sec-moduli} also admit equivariant extensions. First note that the action of $\mathbb{Z}/p$ on $C$ given by the rotation induces an action on $C \times \mathbb{P}^1$ by taking the action on target $\mathbb{P}^1$ to be trivial. The induced action on the tautological line bundle $\mathcal{O}_C(-1)$ further induces an action on the line bundle $\mathcal{O}_{C \times \mathbb{P}^1} (d, 1)$, which acquires the structure of a $\mathbb{Z}/p$-equivariant line bundle over $C \times \mathbb{P}^1$. 

Consequently, $H^0(\mathcal{O}(d,1))$ is a $\mathbb{Z}/p$-representation, so its projectivization $\mathcal{P}_d = \mathbb{P} H^0(\mathcal{O}(d,1))$ admits a $\mathbb{Z}/p$-action. The tautological divisor $\mathcal{D} \subseteq \mathcal{P}_d \times (C \times \mathbb{P}^1)$ is also $\mathbb{Z}/p$-invariant under the diagonal action on $\mathcal{P}_d \times (C \times \mathbb{P}^1)$, so the obstruction bundle $\mathrm{Obs} \to \mathcal{P}_d$ also admits the structure of an equivariant vector bundle. Note that the stratification of $\mathcal{P}_d = \mathcal{P}_d^\circ \sqcup \bigsqcup_{\alpha \neq 0} \mathcal{P}_d^\alpha$ into bubbling types is also $\mathbb{Z}/p$-invariant.  By gluing all data to the principal $\mathbb{Z}/p$-bundle $S^\infty \to B\mathbb{Z}/p$, we obtain the following.

\begin{defn}[Compare \cref{defn:thick-noneq-section-moduli}]
\label{defn:thick-eq-sec-moduli}
    The \emph{thickened equivariant section moduli space of degree $A$} is the homotopy orbits of $\mathcal{P}_d$ under the $\mathbb{Z}/p$-action, namely
\begin{equation}
\mathcal{P}_A^{eq} = \mathcal{P}_d^{eq} := S^\infty \times_{\mathbb{Z}/p} \mathcal{P}_d. 
\end{equation}
The \emph{equivariant obstruction bundle} is the homotopy orbits of the total space of $\mathrm{Obs} \to \mathcal{P}_d$, which is a vector bundle over $\mathcal{P}_d^{eq}$:
\begin{equation}
\mathrm{Obs}^{eq} := S^\infty \times_{\mathbb{Z}/p} \mathrm{Obs} \to S^\infty \times_{\mathbb{Z}/p} \mathcal{P}_d = \mathcal{P}_d^{eq}.
\end{equation}
In particular the rank of the equivariant obstruction bundle is equal to the rank of the non-equivariant obstruction bundle with the same description of the fibers. We define a smooth section (cf. \cref{eqn:obs-section})
\begin{align}\label{eq-obs-section}
[\nu^{eq}] : (\mathcal{P}_d^{eq})^\circ &\to \mathrm{Obs}^{eq}, \\
(w, s_v) & \mapsto [ (w \times \widetilde{v})^* \nu^{eq} ] \in H^{0,1}_{\overline{\partial}} (C, v^*N) \cong H^1(C, v^*N) \cong \mathrm{Obs}^{eq}|_{(w,s_v)}.
\end{align}
defined over the open locus $(\mathcal{P}_d^{eq})^\circ := S^\infty \times_{\mathbb{Z}/p} \mathcal{P}_d^\circ \subseteq \mathcal{P}_d^{eq}$.
\end{defn}

\begin{defn}[Compare \cref{defn:noneq-sec-moduli}]
\label{defn:eq-sec-moduli}
    The \emph{equivariant section moduli space of degree $A$} is the zero locus of the section $[\nu^{eq}]$ ,
    \begin{equation}
        \mathcal{V}_A^{eq} = \mathcal{V}_d^{eq} := (\mathcal{P}_d^\circ)^{eq} \cap [\nu^{eq}]^{-1}(0) =  \left\{ (w \in S^\infty, s_v \in \mathcal{P}_d^\circ) : \ [\nu^{eq}] (w, s_v) = 0 \right\} / (\mathbb{Z}/p).
    \end{equation}
\end{defn}

By restricting to the open cell $\Delta^i \subseteq B\mathbb{Z}/p$, we obtain the analogues $\mathcal{P}_d^{eq, i}$, $\mathrm{Obs}^{eq, i}$, and $\mathcal{V}_d^{eq, i}$ for equivariant degree $i$. The section $[\nu^{eq}]$ restricts to a section of $\mathrm{Obs}^{eq, i} \to (\mathcal{P}_d^{eq, i})^\circ$, which is equal to the section induced by $\nu^{eq,i} := \nu^{eq}|_{\Delta^i \times C \times \mathbb{P}^1}$.

We can now discuss equivariant versions of regularity, compactness, and incidence constraints for the section moduli spaces. 

\begin{lem}[Compare \cref{lem:section-transverse}, \cref{prop:noneq-map-equal-section}]\label{lem:eq-map-equal-section}
For a generic equivariant decoupling perturbation data $\nu^{eq}$, every $\mathcal{V}_A^{eq, i}$ for $i \ge 0$ is a smooth oriented manifold of dimension $i + \dim_{\mathbb{R}} X + 2c_1(A) = i + 6$, and moreover $\mathcal{M}_A^{eq, i} \cong \mathcal{V}_A^{eq, i}$ as smooth oriented manifolds.
\end{lem}
\begin{proof}
    The first statement is just a parametrized version of \cref{lem:section-transverse}. The proof of the second statement is the same as that of \cref{prop:noneq-map-equal-section}. The linearized operator for $(w, v, \phi) \in \mathcal{M}_A^{eq, i}$ now has the form (again, write $\Gamma(E)$, $\Gamma^{0,1}(E)$ in place of $C^\infty(C, E)$ and $C^\infty(C, \mathrm{Hom}^{0,1}(TC,E))$)
    \begin{align}
        D_{(w,v,\phi)}^{eq, i}: T_{w}\Delta^i \oplus \Gamma(v^*T\mathbb{P}^1) \oplus \Gamma(v^*N) &\to \Gamma^{0,1}(v^*T\mathbb{P}^1) \oplus \Gamma^{0,1}(v^*N) \\
        (\eta, \xi, \psi) &\mapsto ( \overline{\partial} \xi , \overline{\partial}\psi - \nabla_{\eta} \nu^{eq, i} - \nabla_{\xi} \nu^{eq,i})
    \end{align}
    and accordingly
    \begin{align}
        T_{(w, v, \phi)} \mathcal{M}_A^{eq, i} \cong \ker\left(\nabla \nu^{eq, i} : T_w \Delta^i \oplus \ker (\overline{\partial}_{v^*T\mathbb{P}^1} ) \to \mathrm{coker}(\overline{\partial}_{v^*N}) \right).
    \end{align}
    This is identified with
    \begin{align}
        T_{(w,v,\phi)}\mathcal{V}_A^{eq, i} \cong \ker \left( d[\nu^{eq, i}]: T_{(w, s_v)}(\mathcal{P}_d^{eq,i})^\circ \to \mathrm{Obs}^{eq, i}|_{(w,s_v)} \right),
    \end{align}
    because the fibration $(\mathcal{P}_d^{eq, i})^\circ \to \Delta^i$ is locally trivial. 
\end{proof}

From here, the natural next step would be to show the analogue of \cref{thm:noneq-section-compactification}, that is to show that $\mathcal{V}_A^{eq} \subseteq \mathcal{P}_A^{eq}$ can be bordified so that $[\overline{\mathcal{V}}_A^{eq}]$ defines a $\mathbb{F}_p$-coefficient equivariant homology cycle (see \cref{rem:no-Fp-pseudocycle}). We will bypass that general theory, proving a statement slightly weaker but enough for our applications.

Denote the fiber of the projection $\mathcal{P}_d^{eq} \to B\mathbb{Z}/p$ over $w \in \Delta^i \subseteq B\mathbb{Z}/p$ by $\mathcal{P}_{d,w}^{eq}$. Clearly $\mathcal{P}_{d,w}^{eq} \cong \mathcal{P}_d$ and it admits an induced stratification from the stratification of $\mathcal{P}_d$ (which is $\mathbb{Z}/p$-invariant, and hence the stratification of $\mathcal{P}_{d,w}^{eq}$ assembles into a stratification of $\mathcal{P}_d^{eq}$). Denote $\mathcal{V}_{d, w}^{eq}$ for the corresponding fiber of $\mathcal{V}_d^{eq} \to B\mathbb{Z}/p$.

\begin{prop}[Compare \cref{thm:noneq-section-compactification}]\label{prop:eq-section-compactification}
    For a generic choice of $\nu^{eq}$, the following holds.

    \begin{enumerate}[topsep=-2pt, label=(\roman*)]
    \item $\mathcal{V}_d^{eq, i} \subseteq (\mathcal{P}_d^{eq, i})^\circ$ is smooth of real dimension $i + \dim_{\mathbb{R}} X + 2c_1(d[\mathbb{P}^1]) = i+6$ (\cref{lem:eq-map-equal-section}), 
    \item There is a fiberwise compactification $\overline{\mathcal{V}}_d^{eq, i} \subseteq \mathcal{P}_d^{eq, i}$ of $\mathcal{V}_d^{eq, i}$ over the projection to the equivariant parameter space $\mathcal{P}_d^{eq, i} \to \Delta^i$. That is, the fiber of $\overline{\mathcal{V}}_d^{eq, i}$ over $w \in \Delta^i$ is the closure of $\mathcal{V}_{d,w}^{eq, i} \subseteq \mathcal{P}_{d,w}^{eq, i}$. It is stratified as follows:
    \begin{equation}
\overline{\mathcal{V}}_d^{eq, i} = \mathcal{V}_d^{eq, i} \sqcup \bigsqcup_{\alpha \neq 0} (\mathcal{V}_d^{eq, i})^\alpha, \quad \alpha = (\alpha_1, \alpha_2, \dots), \ \alpha_k = 0 \mbox{ for } k \gg 0, \ \sum_{k \ge 0} k \cdot \alpha_k \le d.
\end{equation}
    \item For each $\alpha \neq 0$, the stratum $(\mathcal{V}_d^{eq,i})^\alpha$  is contained in a locally closed subvariety (of $\mathcal{P}_d^{eq, i}$) of dimension $\le \dim \mathcal{V}_d^{eq, i} -2$.
\end{enumerate}
\end{prop}

\begin{proof}
    The proof follows that of \cref{thm:noneq-section-compactification}. The equivariant analogue of the key lemma \cref{lem:extension-of-obs-section} is immediate: $\nu^{eq}$ induces the diagram
    \begin{equation}\label{eqn:eq-extension-of-obs-section}
        \begin{tikzcd}
            & \mathrm{Obs}^{eq} \dar \arrow[r, twoheadrightarrow] & \overline{\mathrm{Obs}}^{eq} \\
            (\mathcal{P}_d^\circ)^{eq} \arrow[r, hook] \arrow[ru, "{[\nu^{eq}]}"] & \mathcal{P}_d^{eq} \arrow[ru, "{\overline{[\nu^{eq}]}}"']&
        \end{tikzcd}.
    \end{equation}
    The next step is to extend $[\nu^{eq}]$ to a global continuous section of $\mathrm{Obs}^{eq} \to \mathcal{P}_d^{eq}$, i.e. to establish an equivariant analogue of \cref{eqn:extending-obs-section}. The construction of the global extension depended on the ability to fix a surjection from a vector bundle $E \to \mathcal{F}$ for the subsheaf $\mathcal{F} = R^0(\mathrm{proj}_{\mathcal{P}_d})_* \left(\mathcal{N}^\vee (\mathcal{D}) \otimes K_{D/\mathcal{P}_d} \otimes \mathcal{I}_{\mathcal{B}_{\mathrm{nonred}}} \right)$ of $\mathrm{Obs}^\vee$. Any non-equivariant surjection $q: E \to \mathcal{F}$ from a vector bundle $E$ can be extended to an equivariant surjection from a vector bundle by taking $E \otimes \mathbb{C}[\mathbb{Z}/p] \to \mathcal{F}$, defined as $e \otimes \sigma \mapsto \sigma \cdot q(e)$. Hence the $\mathbb{Z}/p$-equivariant sheaf $\mathcal{F}$ admits a surjection from a $\mathbb{Z}/p$-equivariant vector bundle. This suffices to construct the desired continuous extension of $[\nu^{eq}]$ to $\mathcal{P}_d^{eq}$, and the argument from \cref{thm:noneq-section-compactification} passes through to establish (ii) and (iii).
\end{proof}

For the incidence constraints, choose $\mathcal{Y} \subseteq X \times X^p \times X$ as in the discussion for the equivariant map moduli, so that $Y_1 = \cdots = Y_p$ are all fiber cycles of the form $Y_j = \pi^{-1}(y_j)$, with $y_1 = \cdots = y_p$. By \cref{lem:eq-incidence-cycle-only-fibers}, we can choose an equivariant incidence cycle $\mathcal{Y}^{eq}$ such that each $Y_j (w) := \mathrm{proj}_j \circ \mathcal{Y}^{eq}_w$ is also a fiber cycle of the form $Y_j(w) = \pi^{-1}(y_j(w))$ for $y_j(w) \in \mathbb{P}^1$. 

By definition $\mathcal{L}_1 \cong \cdots \mathcal{L}_p \cong \mathcal{O}(1)$. This renders $\mathrm{IC} = \mathcal{L}_0 \oplus \left( \bigoplus_{j=1}^p \mathcal{L}_j \right) \oplus \mathcal{L}_\infty$ into an equivariant vector bundle over $\mathcal{P}_{d}$, where the $\mathbb{Z}/p$ acts on $\mathrm{IC}$ by cyclically permuting the $p$ summands in $\bigoplus_{j=1}^p \mathcal{L}_j$ and fixing $\mathcal{L}_0$, $\mathcal{L}_\infty$. Extending $\mathrm{IC}^{eq} \to \mathcal{P}_d^{eq}$, we see that the tuples $(z_j, y_j(w))$ and the associated sections $\mathrm{ev}_{(z_j, y_j(w))}$ assemble to a global section $\mathrm{ev}_{\mathcal{Y}^{eq}} : \mathcal{P}_d^{eq} \to \mathrm{IC}^{eq}$, namely

\begin{equation}
        \mathrm{ev}_{\mathcal{Y}^{eq}}(w, s) = \left(s(z_0, y_0(w)), \bigoplus_{j=1}^p s (z_j, y_j(w)),  s(z_\infty, y_\infty(w)) \right).
    \end{equation}

As $\mathbb{Z}/p$-equivariant bundles, $\mathrm{Obs}^{eq}$, $\mathrm{IC}^{eq}$ have (mod $p$ reductions of) equivariant Euler classes $c_{\mathrm{top}}^{eq} (\mathrm{Obs}^{eq})$, $c_{\mathrm{top}}^{eq}(\mathrm{IC}^{eq})$, which are equivariant cohomology classes in $H^*_{\mathbb{Z}/p}(\mathcal{P}_d ; \mathbb{F}_p)$.

\begin{lem}[Compare \cref{lem:section-counts-equal-to-Chern-integral}]\label{lem:eq-section-counts-equal-to-Chern-integral}
    Suppose $\mathcal{Y}^{eq}$ is chosen so that for each $i \ge 0$, the restrictions of  $\mathrm{ev}_{\mathcal{Y}^{eq}}: \mathcal{P}_d^{eq} \to \mathrm{IC}^{eq}$ to $\mathcal{P}_d^{eq, i}$ are transverse to every $(\mathcal{V}_d^{eq, i})^{\alpha}$. For the unique $i$ such that $\dim_{\mathbb{R}} \mathcal{V}_d^{eq, i} = i + \dim_{\mathbb{R}} \mathcal{V}_d = \mathrm{codim}_{X \times X^p \times X} \mathcal{Y}$, so that $\mathcal{V}_d^{eq, i} \pitchfork \mathrm{ev}^{-1}_{\mathcal{Y}^{eq}}(0)$ is a $0$-dimensional manifold, we have
    \begin{equation}
        \# \mathcal{V}_d^{eq, i} \pitchfork  \mathrm{ev}_{\mathcal{Y}^{eq}}^{-1}(0) = \left\langle \Delta^i, 
 \int_{\mathcal{P}_d^{eq}} c_{\mathrm{top}}^{eq}(\mathrm{Obs}^{eq}) \cup c_{\mathrm{top}}^{eq}(\mathrm{IC}^{eq}) \right\rangle \in \mathbb{F}_p,
    \end{equation}
    where the equivariant integral is valued in $H^*_{\mathbb{Z}/p}(\mathrm{pt};\mathbb{F}_p) := \mathbb{F}_p[\![t, \theta]\!]$.
\end{lem}
\begin{proof}
    The well-definedness of $\# \mathcal{V}_d^{eq, i} \pitchfork \mathrm{ev}_{\mathcal{Y}^{eq}}^{-1}(0)$ is exactly as in \cref{lem:eq-0-dim-count-map-moduli}: the zero locus of $\mathrm{ev}_{\mathcal{Y}^{eq}}$ only meets $\overline{\mathcal{V}}_d^{eq, i}$ along $\mathcal{V}_d^{eq, i}$ by dimension reasons, and in a $1$-parameter choice of $\mathcal{Y}^{eq}$ the count can only jump by multiples of $p$ (from intersections with $\sigma^k \mathcal{V}_d^{eq, i-1}$). The argument uses the dimension count from \cref{prop:eq-section-compactification} (iii).

    Let $[\nu^{eq}]'$ denote a smooth perturbation of a continuous global extension $\Theta[\nu^{eq}]$ of $[\nu^{eq}]$. Then the intersection $\mathcal{V}_d^{eq, i} \pitchfork \mathrm{ev}^{-1}_{\mathcal{Y}^{eq}}(0)$ is exactly the zero locus of the section $[\nu^{eq}]' \oplus \mathrm{ev}_{\mathcal{Y}^{eq}}$ of $\mathrm{Obs}^{eq} \oplus \mathrm{IC}^{eq}$ restricted to $\mathcal{P}_d^{eq, i} \subseteq \mathcal{P}_d^{eq}$. Since $\mathcal{P}_d^{eq, i}$ is the preimage of $\Delta^i$ under the projection $\mathcal{P}_d^{eq} \to B\mathbb{Z}/p$, this count is exactly given by
    \begin{equation}
        \left \langle \Delta^i,  \int_{\mathcal{P}_d^{eq}} c_{\mathrm{top}}^{eq}(\mathrm{Obs}^{eq} \oplus \mathrm{IC}^{eq}) \right\rangle = \left\langle \Delta^i, 
 \int_{\mathcal{P}_d^{eq}} c_{\mathrm{top}}^{eq}(\mathrm{Obs}^{eq}) \cup c_{\mathrm{top}}^{eq}(\mathrm{IC}^{eq}) \right\rangle
    \end{equation}
    by the general theory of Chern classes. (To bypass working with the whole polynomial $\int_{\mathcal{P}_d^{eq}} c_{\mathrm{top}}^{eq}(\mathrm{Obs}^{eq}) \cup c_{\mathrm{top}}^{eq}(\mathrm{IC}^{eq}) \in \mathbb{F}_p [\![t, \theta]\!]$ and equivariant Poincar\'e duality issues, one may restrict to $\bigcup_{j \le i} \mathcal{P}_d^{eq, j} \subseteq \mathcal{P}_d^{eq}$ which truncates the equivariant parameters after equivariant degree $i$, and apply non-equivariant theory to the restriction.)
\end{proof}

\subsection{Comparison of equivariant moduli spaces}\label{ssec:comparison-of-eq-moduli}

The final piece of our strategy for computation is the equivariant version of the comparison lemma,  \cref{prop:transversality-is-equal}. Fix $\mathcal{Y} = Y_0 \times (Y_1 \times \cdots \times Y_p) \times Y_\infty$ as before, with $Y_1 = \cdots = Y_p$ being fiber cycles, and let $\mathcal{Y}^{eq}$ be an $S^\infty$-parametrized equivariant family of perturbations such that each $Y_j(w) = \mathrm{proj}_j \circ \mathcal{Y}^{eq}_w$ is still a fiber cycle (see \cref{lem:eq-incidence-cycle-only-fibers}).

\begin{lem}[Compare \cref{prop:transversality-is-equal}]\label{lem:eq-transversality-is-equal}
    Consider the unique $i$ such that $\mathcal{M}_A^{eq, i}$, $\mathcal{V}_A^{eq, i}$ are both regular of dimension $\dim_{\mathbb{R}} \mathcal{M}_A^{eq, i} = \dim_{\mathbb{R}} \mathcal{V}_A^{eq, i} = \mathrm{codim}_{X \times X^p \times X} \mathcal{Y}$. Then the evaluation map $ev^{eq, i} : \mathcal{M}_A^{eq, i} \to X \times (\overline{\Delta}^i \times X^p) \times X$ is transverse to $\mathcal{Y}^{eq, i} : \Delta^i \times \mathcal{Y} \to X \times (\overline{\Delta}^i \times X^p) \times X$ if and only if the section $\mathrm{ev}_{\mathcal{Y}^{eq, i}} : \mathcal{P}_A^{eq, i} \to \mathrm{IC}^{eq, i}$ is transverse to $\mathcal{V}_A^{eq, i}$ (in the zero section).
\end{lem}
\begin{proof}
    The proof is the same as that of \cref{prop:transversality-is-equal}, incorporating the tangent spaces from the equivariant parameter spaces $\Delta^i$. We omit the details.
\end{proof}

By comparing the equivariant moduli spaces $\mathcal{M}_A^{eq, i} \cong \mathcal{V}_A^{eq, i}$, rigidified to $0$-dimensional counts by means of an equivariant incidence cycle $\mathcal{Y}^{eq}$, we can compute the structure constants for the quantum Steenrod operations. The structure constants are defined by means of the map moduli spaces, which we identify with the section moduli spaces, where the computations can be done by equivariant cohomology.

\begin{rem}\label{rem:QSt-BM-cycles}
    Strictly speaking, quantum Steenrod operations were only defined for closed manifolds, so we cannot directly consider quantum Steenrod operations on $X$. One way to make the theorem below precise is to compactify $X$ with its projective completion $\hat{X} = \mathbb{P}(\mathcal{O}(-1)^{\oplus 2} \oplus \underline{\mathbb{C}}) \to \mathbb{P}^1$ and observe that the curves we consider are always near $\mathbb{P}^1 \subseteq X$. Therefore we can work relatively to $\hat{X} \setminus X$. The fiber cycles $Y_j$ define locally finite homology cycles in $H_*^{\mathrm{BM}}(X) = H_*(\hat{X}, \hat{X}\setminus X)$, and therefore the theorem can be interpreted as computing the structure constants of quantum Steenrod operations on $\hat{X}$.
\end{rem}

We now proceed to compute the quantum Steenrod operations $Q\Sigma_b(b_0)$, determined by the structure constants $(Q\Sigma_b(b_0), b_\infty) \in \Lambda_{\mathbb{F}_p}[\![t, \theta]\!]$ where $(\cdot, \cdot)$ is the Poincar\'e pairing. Below we describe each of the coefficients $\langle \Delta^i, (Q\Sigma_b(b_0), b_\infty) \rangle \in \Lambda_{\mathbb{F}_p}$ of $(t,\theta)^i$.

\begin{thm}\label{thm:qst-equals-Chern-integral}
    Fix $b_0, b, b_\infty \in H^*(X)$ to be Poincar\'e dual to the locally finite cycles given by the embedded submanifolds $Y_0$, $Y \cong Y_1 \cong \cdots \cong Y_p$, and $Y_\infty$. The corresponding structure constants for the quantum Steenrod operations can be computed by the equivariant integral
    \begin{equation}
    (Q\Sigma_b (b_0), b_\infty) = \sum_i \sum_{A \in H_2(X)} \left\langle \Delta^i, 
    \int_{\mathcal{P}_A^{eq}} c_{\mathrm{top}}^{eq}(\mathrm{Obs}^{eq}) \cup c_{\mathrm{top}}^{eq}(\mathrm{IC}^{eq}) \right\rangle q^A \ (t,\theta)^i \in \Lambda_{\mathbb{F}_p}[\![t, \theta]\!].
    \end{equation}
\end{thm}
\begin{proof}
   We can choose $\nu^{eq}$ generically so that:
    \begin{itemize}
        \item (\cref{lem:eq-dc-map-transverse}) The map moduli spaces $\mathcal{M}_A^{eq, i}$ are regular.
        \item (\cref{lem:eq-dc-stable-map-transverse}) The simple stable maps $\partial \overline{\mathcal{M}}_A^{eq, i} = \bigcup_{T} \mathcal{M}_{T, \{ A_b'\}}^{eq, i} (-\Sigma z_j')$ are regular.
        \item (\cref{lem:eq-map-equal-section}) The section moduli spaces $\mathcal{V}_A^{eq, i}$ are regular and isomorphic to $\mathcal{M}_A^{eq, i}$.
        \item (\cref{prop:eq-section-compactification}) The compactifying strata $(\mathcal{V}_A^{eq, i})^\alpha$ are regular.
    \end{itemize}
    Moreover, following \cref{lem:eq-transverse-incidence-cycles}, we can choose $\mathcal{Y}^{eq}$, so that each $Y_j(w)$ is a fiber $\pi^{-1}(y_j(w))$, satisfying:
    \begin{itemize}
        \item  For every $i \ge 0$, $\mathcal{Y}^{eq, i} : \Delta^i \times \mathcal{Y} \to X \times (\overline{\Delta}^i \times X^p) \times X$ is transverse to evaluation maps from the map moduli spaces $\mathcal{M}_A^{eq, i}$ and simple stable maps $\partial \overline{\mathcal{M}}_A^{eq, i}$.
        \item For every $i \ge 0$, $\mathrm{ev}_{\mathcal{Y}^{eq, i}} : \mathcal{P}_d^{eq, i} \to \mathrm{IC}^{eq, i}$ is transverse to the section moduli spaces $\mathcal{V}_A^{eq, i}$ and the compactifying strata $(\mathcal{V}_A^{eq, i})^{\alpha}$.
    \end{itemize}

    The choices of such $\nu^{eq}$ and $\mathcal{Y}^{eq}$ are generic in the sense of \cref{rem:equivariant-genericity-cycle}. Now the desired result follows from the chain of equalities

    \begin{align}
        \label{eqn:thm-r1}
        \mathrm{Coefficient} \left(q^A (t,\theta)^i ; ( Q\Sigma_b (b_0), b_\infty )  \right) &= \# \ \mathrm{ev}^{eq} \left( \overline{\mathcal{M}}_A^{eq, i} \right) \pitchfork \mathcal{Y}^{eq} = \# \ \mathrm{ev}^{eq} \left( \mathcal{M}_A^{eq, i} \right) \pitchfork \mathcal{Y}^{eq} \\
        \label{eqn:thm-r2}
        &= \# \ \mathcal{V}_A^{eq, i} \pitchfork \mathrm{ev}_{\mathcal{Y}^{eq}}^{-1}(0) = \# \ \overline{\mathcal{V}}_A^{eq, i} \pitchfork \mathrm{ev}_{\mathcal{Y}^{eq}}^{-1}(0) \\
        \label{eqn:thm-r3}&= \left\langle \Delta^i, 
 \int_{\mathcal{P}_A^{eq}} c_{\mathrm{top}}^{eq}(\mathrm{Obs}^{eq}) \cup c_{\mathrm{top}}^{eq}(\mathrm{IC}^{eq}) \right\rangle \in \mathbb{F}_p.
    \end{align}
    The first equality \eqref{eqn:thm-r1} is the definition of quantum Steenrod operations (see \cref{defn:QSt-structure-constants}), together with the observation that decoupling perturbation data can be regarded as a special subclass of the usual perturbation data used in the definition (see the discussion after \cref{defn:noneq-nu-dc}). Note that there is no sign because $b_0, b, b_\infty$ all have even degrees. 
    
    The second equality follows from dimension counts and transversality assumptions for the map moduli spaces. Third equality \eqref{eqn:thm-r2} is \cref{lem:eq-map-equal-section}, \cref{lem:eq-transversality-is-equal}. Fourth equality follows from dimension counts and transversality assumptions for the section moduli spaces. Finally the last count \eqref{eqn:thm-r3} can be computed using equivariant cohomology, which is exactly the content of \cref{lem:eq-section-counts-equal-to-Chern-integral}. A different choice of $\nu^{eq}$ and $\mathcal{Y}^{eq}$ gives the same count in $\mathbb{F}_p$, by the argument in \cref{lem:eq-0-dim-count-map-moduli}.
\end{proof}

\subsection{Example computations}\label{ssec:computaiton-localP1}
Using \cref{thm:qst-equals-Chern-integral} (see also \cref{rem:QSt-BM-cycles}), one can compute the quantum Steenrod operations for the fiber class $b = \mathrm{PD}[Y = \pi^{-1}(y)] \in H^2(X)$. The interesting computations come from applying $Q\Sigma_b$ to classes $1 \in H^0$ and $b \in H^2$.

In light of \cref{thm:qst-equals-Chern-integral}, it suffices to compute the equivariant Chern classes of the obstruction bundle and the incidence constraint bundle. To compute the equivariant weights of the bundles involved, one must first choose a \emph{linearization} of the $\mathbb{Z}/p$-action $\sigma_C$ on $C = \mathbb{P}^1$, i.e. a lift of $\sigma_C$ to an action on the tautological bundle $\mathcal{O}_C(-1) \to C$. (The resulting answer should not depend on the choice, but the intermediate calculations do depend on the choice.)

Recall for $\zeta = e^{2\pi i/p}$, the action on $C$ is given by $\sigma_C[z:w] = [z: \zeta^{-1}w]$. We choose the linearization to be $\sigma (z, w) = (z, \zeta^{-1}w)$ on $\mathcal{O}_C(-1) \to C$. Note that this choice induces $\mathbb{Z}/p$-action on all bundles $\mathcal{O}_C(d) \to C$ for $d \in \mathbb{Z}$. In particular, $H^0(C, \mathcal{O}_C(d))$ becomes a $\mathbb{Z}/p$-representation for every $d$ and our choice is that $V := H^0(C, \mathcal{O}(1)) = \mathbb{C}_{1} \oplus \mathbb{C}_{\zeta}$. Here $\mathbb{C}_\lambda$ denotes the rank $1$ irreducible representation where the generator $\sigma \in \mathbb{Z}/p$ acts by character $\lambda \in \mathrm{U}(1)$. For $d \ge 0$, $H^0(C, \mathcal{O}(d)) \cong \mathrm{Sym}^d V$ as representations, where $\mathrm{Sym}^d$ denotes the $d$th symmetric power operation.

\begin{lem}\label{lem:eq-coh-section}
The equivariant cohomology algebra of $\mathcal{P}_d$ is computed as
\begin{equation}\label{eqn:eq-coh-section}
    H^*_{\mathbb{Z}/p}(\mathcal{P}_d ; \mathbb{F}_p) = H^*(\mathcal{P}_d^{eq} ; \mathbb{F}_p) \cong \mathbb{F}_p [ H, t, \theta]/H^2(H-t)^2\cdots(H-dt)^2,
\end{equation}
where $t = c_1^{\mathbb{Z}/p}(\mathcal{O}_{B\mathbb{Z}/p}(1)) \in H^2_{\mathbb{Z}/p}(\mathrm{pt};\mathbb{F}_p)$ is as in \eqref{eqn:equiv-coh-ground-ring} and $H = c_1^{\mathbb{Z}/p}(\mathcal{O}_{\mathcal{P}_d}(1)) = c_1(\mathcal{O}_{\mathcal{P}_d^{eq}}(1)) \in H^2_{\mathbb{Z}/p}(\mathcal{P}_d;\mathbb{F}_p)$ is the (equivariant) hyperplane class.
\end{lem}
\begin{proof}
    This is a corollary of the splitting principle formula for the cohomology of a projective bundle (see e.g. \cite[Example 9.1.2.1]{CK99}) applied to $\mathcal{P}_d = \mathbb{P} H^0(C \times \mathbb{P}^1 , \mathcal{O}_{C \times \mathbb{P}^1}(d,1))$. Note that $H^0(C \times \mathbb{P}^1 , \mathcal{O}_{C \times \mathbb{P}^1}(d,1)) \cong H^0(C, \mathcal{O}_C(d)) \otimes H^0(\mathbb{P}^1, \mathcal{O}_{\mathbb{P}^1}(1)) \cong \mathrm{Sym}^d V \otimes (\mathbb{C}_1)^{\oplus 2}$, because the $\mathbb{Z}/p$-action on $\mathbb{P}^1$ factor is trivial. Moreover, $\mathrm{Sym}^d V \cong \bigoplus_{k=0}^d \mathbb{C}_{\zeta^k}$. It remains to recall that our choice of convention (\cref{rem:equiv-convention}) is that the line bundle over $B\mathbb{Z}/p$ induced from the representation $\mathbb{C}_{\zeta}$ has weight $c_1^{\mathbb{Z}/p}(\mathcal{O}_{B\mathbb{Z}/p}(-1))=-t$.
\end{proof}

\begin{lem}\label{lem:eq-euler-obs}
    The equivariant top Chern class of the obstruction bundle is given by
    \begin{equation}\label{eqn:eq-euler-obs}
        c_{\mathrm{top}}(\mathrm{Obs}^{eq}) = (H-t)^2 (H-2t)^2 \cdots (H-(d-1)t)^2 \in H^*_{\mathbb{Z}/p}(\mathcal{P}_d ; \mathbb{F}_p).
    \end{equation}
\end{lem}
\begin{proof}
    This follows from \cref{lem:obs-description}, which establishes an isomorphism 
    \begin{equation}
        \mathrm{Obs}^\vee \cong \mathcal{O}_{\mathcal{P}_d}(1) \otimes \left( H^0(C \times \mathbb{P}^1, \mathcal{O}(d,2) \otimes K_{C \times \mathbb{P}^1} )\right)^{\oplus 2}
    \end{equation} 
    where $K_{C \times \mathbb{P}^1}$ denotes the canonical bundle (whose global sections are products of $1$-forms on $C$ and $\mathbb{P}^1$). To compute the Chern class, it therefore suffices to determine $H^0(C \times \mathbb{P}^1, \mathcal{O}(d,2) \otimes K_{C \times \mathbb{P}^1} ) \cong H^0(C, \mathcal{O}(d) \otimes K_C) \otimes H^0(\mathbb{P}^1, \mathcal{O}(2) \otimes K_{\mathbb{P}^1}) \cong H^0(C, \mathcal{O}(d) \otimes K_C)$ as a representation of $\mathbb{Z}/p$. Using the basis $z^k \otimes dz$ ($k = 1, \dots, d-2$) (where $z \in \mathbb{C} \subseteq C$ is the affine coordinate) for the global sections of $\mathcal{O}(d) \otimes K_C$, one sees that as representations $H^0(C, \mathcal{O}(d) \otimes K_C) = \bigoplus_{k=1}^{d-1} \mathbb{C}_{\zeta^k}$. From this the result follows.
\end{proof}

\begin{lem}\label{lem:eq-euler-IC}
    The equivariant top Chern class of the incidence constraint bundle is given by
    \begin{equation}
        c_{\mathrm{top}} (\mathrm{IC}^{eq}) = H(H-t)\cdots(H-(p-1)t) \cdot c_{0, \infty}(H,t) \in H^*_{\mathbb{Z}/p}(\mathcal{P}_d),
    \end{equation}
    where
    \begin{equation}
        c_{0, \infty}(H,t) = 
        \begin{cases}
            1 & \mbox{ if } b_0 = 1, \ b_\infty = 1 \\
            (H-dt) & \mbox{ if } b_0 = b, \ b_\infty = 1 \\
            H & \mbox{ if } b_0 = 1, \ b_\infty = b \\
            H(H-dt) & \mbox{ if } b_0 = b, \ b_\infty = b
        \end{cases}.
    \end{equation}
\end{lem}
\begin{proof}
    Recall the definition of $\mathrm{IC}$ (\cref{eqn:incidence-bundle} and \eqref{eqn:incidence-line-bundle}), given as $\mathrm{IC} = \mathcal{L}_0 \oplus (\mathcal{L}_1 \oplus \cdots \oplus \mathcal{L}_p) \oplus \mathcal{L}_\infty$ where each $\mathcal{L}_j$ is either $0$ or a line bundle isomorphic to $\mathcal{O}_{\mathcal{P}_d}(1)$.

    Considered equivariantly, $\mathcal{L}_1 \oplus \cdots \oplus \mathcal{L}_\infty \cong \mathrm{Reg}_{\mathbb{Z}/p} \otimes \mathcal{O}_{\mathcal{P}_d}(1)$ where $\mathrm{Reg}_{\mathbb{Z}/p}$ denotes the regular representation of $\mathbb{Z}/p$. In terms of characters, $\mathrm{Reg}_{\mathbb{Z}/p}$ is the rank $p$ representation decomposing as $\bigoplus_{k=0}^{p-1} \mathbb{C}_{\zeta^k}$. This accounts for the factor $\prod_{k=0}^{p-1}(H-kt)$.

    For $\mathcal{L}_0$ and $\mathcal{L}_\infty$, the first guess would be to associate $H = c_1^{\mathbb{Z}/p}(\mathcal{O}_{\mathcal{P}_d}(1))$ whenever the corresponding cohomology class is the (Poincar\'e dual of) the fiber cycle. But recall that the definition of $\mathcal{L}_0$ and $\mathcal{L}_\infty$ involves a choice of an isomorphism $\mathcal{O}(d,1)|_{(0, y_0)} \cong \mathbb{C}$ and $\mathcal{O}(d,1)|_{(\infty, y_\infty)} \cong \mathbb{C}$. In our case, $0, \infty \in C$ are fixed points of the $\mathbb{Z}/p$-action and therefore the fibers $\mathcal{O}(d,1)|_{(0, y_0)}$, $\mathcal{O}(d,1)|_{(\infty, y_\infty)}$ are $\mathbb{Z}/p$-representations. With our choice of linearization, $\mathcal{O}(d,1)|_{(0, y_0)} \cong \mathbb{C}_{\zeta^d}$, $\mathcal{O}(d,1)|_{(\infty, y_\infty)} \cong \mathbb{C}_1$ as representations. Therefore when $b_0=b$ the corresponding line bundle must be twisted by $\mathbb{C}_{\zeta^d}$, accounting for the shift $H \mapsto H-dt$. For $b_\infty = b$ no shift is needed.
\end{proof}

\begin{cor}\label{cor:computation}
The structure constants for $Q\Sigma_b$ are computed as follows. As usual we write $q^d$ for $q^A$ such that $A  = d[\mathbb{P}^1] \in H_2(X;\mathbb{Z})$ for fixed $d \ge 0$.
\begin{equation}
    Q\Sigma_b = \begin{pmatrix} (Q\Sigma_b (1), b) & (Q\Sigma_b (b), b) \\ (Q\Sigma_b (1 ), 1 )  & (Q\Sigma_b (b), 1) \end{pmatrix} = \begin{pmatrix} \sum_{d=1}^{\infty} -d^{p-2}t^{p-2} q^{d} & \sum_{d=1}^\infty -t^{p-1} q^{pd} \\ \sum_{d=1}^{\infty} -2d^{p-3}t^{p-3} q^{d} & \sum_{d=1}^{\infty} d^{p-2}t^{p-2} q^{d} \end{pmatrix}.
\end{equation}
\end{cor}
\begin{proof}
    \cref{lem:eq-coh-section}, \cref{lem:eq-euler-obs}, \cref{lem:eq-euler-IC} allows computation from \cref{thm:qst-equals-Chern-integral}.
\end{proof}

\begin{rem}
    The results in \cref{cor:computation} agrees with the results obtained by the recursion from covariant constancy, modulo the Novikov parameters annihilated by the derivative (i.e. monomials of the form $(\mathrm{const}\cdot q^{pA})$). But \cref{cor:computation} provides strictly more: $(Q\Sigma_b(b), b)$ carries $q^{pA}$ terms.
    
    Intriguingly, this count is nonzero for the curve with fiber insertions at every marked point, in fact exactly equal to $ - 1$ for each multiple of $p$. Such a curve arises from a $p$-fold multiple cover of a curve with three marked points with fiber insertions at every marked point, so it would be interesting to compare the results with the enumerative content of the Aspinwall--Morrison formula $\langle b, b, b \rangle_{pA} = 1$.
    
\end{rem}

\section{$S^1$-equivariant quantum Steenrod operations}\label{sec:S1-QSt}
The aim of this section is to provide a generalization of the definition of the quantum Steenrod operations to the setting where the target symplectic manifold $X$ admits an $S^1$-action. We use almost complex structures of $X$ parametrized by $BS^1 \simeq \mathbb{CP}^\infty$ to define the $(\mathbb{Z}/p \times S^1)$-equivariant analogues $\mathcal{M}^{eq, i, j} = \mathcal{M}_{S^1}^{eq, i, j}$ of the $\mathbb{Z}/p$-equivariant  moduli spaces $\mathcal{M}^{eq, i}$. The proof of the regularity of these moduli spaces and their transversal intersections with $S^1$-equivariant analogues of the incidence cycles are immediate generalizations of the results in \cref{sec:QSt}, so we will be brief with their discussion.

Using the $S^1$-equivariant moduli spaces, the definition of quantum Steenrod operations and the covariant constancy relation are extended to the $S^1$-equivariant setting. In particular, the computation of $S^1$-equivariant quantum Steenrod operations yields a covariantly constant endomorphism for the $S^1$-equivariant quantum connection (in $\mathbb{F}_p$-coefficients).

\subsection{$S^1$-equivariant cohomology}\label{ssec:S1-eq-coh}
As an analogue of \cref{ssec:eq-coh}, we describe our choice of the cellular model for $S^1$-equivariant cohomology. We need to fix a contractible space $ES^1$ with a free $S^1$-action. Again take
\begin{equation}
S^\infty = \{ \sw = (\sw_0, \sw_1, \dots ) \in \mathbb{C}^\infty : \ \sw_k = 0 \mbox{ for } k \gg 0, \|\sw\|^2 = 1 \}
\end{equation}
for the choice of our model. We will denote its elements by $\sw \in ES^1$ to distinguish them from $w \in E\mathbb{Z}/p \simeq S^\infty$. On $S^\infty \simeq ES^1$, an element $\vartheta \in S^1 := \{\vartheta \in \mathbb{C}: \|\vartheta\|^2 = 1 \}$ acts (on the left) by
\begin{equation}
    \vartheta \cdot \sw = \vartheta \sw = \left( \vartheta \sw_0, \vartheta \sw_1, \dots \right).
\end{equation}
We take $BS^1 := S^\infty / S^1 \cong \mathbb{CP}^\infty$ as the quotient.

There is a cellular model for this description, analogous to \eqref{eqn:equiv-param}. Fix $k \ge 0$. Consider
\begin{align}
    \sDelta^{2k} &= \{ \sw \in S^\infty : \sw_k > 0, \sw_{k+1} = \sw_{k+2} = \cdots = 0 \}, \\
    \sDelta^{2k+1} &= \{ \sw \in S^\infty : |\sw_k| > 0, \mathrm{arg}(\sw_k) \neq 0, \sw_{k+1} = \sw_{k+2} = \cdots = 0 \}.
\end{align}
Both $\sDelta^{2k}$ and $\sDelta^{2k+1}$ are homeomorphic to open disks of dimension $2k, (2k+1)$ respectively. Their orientations are given as in \eqref{eqn:equiv-param}.

Compactifications $\overline{\sDelta}^{2k}$ and $\overline{\sDelta}^{2k+1}$ are given by adding in boundaries
\begin{align}
    \partial \overline{\sDelta}^{2k} &= \{ \sw \in S^\infty : \sw_k = \sw_{k+1} = \cdots = 0 \}, \\
    \partial \overline{\sDelta}^{2k+1} &= \{ \sw \in S^\infty : \sw_k \ge 0, \ \sw_{k+1} = \cdots = 0 \}.
\end{align}
The boundaries $\partial \overline{\sDelta}^{2k}$ and $\partial \overline{\sDelta}^{2k+1}$ are stratified by cells of lower dimension. The union $\sDelta^{2k} \cup \sDelta^{2k+1} \cong S^1 \times \sDelta^{2k}$ is exactly the $S^1$-orbit of the cell $\sDelta^{2k}$. Correspondingly, the quotient $BS^1 = S^\infty / S^1 \cong \mathbb{CP}^\infty$ is given as the union of $\sDelta^{2j}$ for every even nonnegative integer $2j$ that defines a CW-structure. Each cell $\sDelta^{2j}$ becomes a cycle in cellular homology and provides an additive basis for $H_*^{S^1}(\mathrm{pt};\mathbb{Z}) := H_*(BS^1 ; \mathbb{Z})$.

In light of this, the $S^1$-equivariant analogues of the equivariant moduli spaces that we will introduce later will only use the even-dimensional cells $\sDelta^{2j}$ as parameter spaces.

For any commutative ring $R$, the $S^1$-equivariant ground ring accordingly admits a presentation
\begin{equation}
    H^*_{S^1}(\mathrm{pt} ; R) \cong R [\![ \teq ]\!],
\end{equation}
where $\teq \in H^2(BS^1;R)$ is the generator such that $\langle \teq, \sDelta^2 \rangle = -1$. We will later consider the case where $R = \mathbb{F}_p$ and $R = \Lambda_{\mathbb{F}_p}$.

\subsection{$\mathbb{Z}/p \times S^1$-equivariant moduli spaces}\label{ssec:S1-eq-moduli}
Using the cellular description of the parameter space $BS^1$, we describe the relevant moduli spaces for defining the $S^1$-equivariant quantum cup product and $S^1$-equivariant quantum Steenrod operations. As in the previous description in \cref{ssec:eq-moduli-spaces}, we will use the Borel model, where we allow the almost complex structure on $X$ to vary over the parameter space $BS^1$.

Let $(X, \omega)$ be a symplectic manifold equipped with a $S^1$-action that preserves the symplectic form. Denote by $\vartheta_X \in \mathrm{Diff}(X, \omega)$ the symplectomorphism given by the action of $\vartheta \in S^1$.
\begin{defn}\label{defn:S1-acs}
    A \emph{$S^1$-equivariant collection of almost complex structures} is the set $J^{eq} := \{ J_{\sw} \}_{\sw \in ES^1}$ where each $J_{\sw} \in \mathrm{End}(TX)$ is a $\omega$-compatible almost complex structure, smoothly depending on $\sw \in ES^1 = S^\infty$, satisfying the equivariance condition $J_{\vartheta \cdot \sw} = \vartheta_X^* J_\sw$. 
\end{defn}
Using the fact that the space of $\omega$-compatible almost complex structures is contractible, one can inductively (for the finite dimensional approximations of $ES^1 \simeq S^\infty$) choose perturbations of a fixed $J$ to establish the existence of $S^1$-equivariant collection of almost complex structures.

To every fixed $J_\sw$, one can repeat the construction of the $\mathbb{Z}/p$-equivariant perturbation data in \cref{defn:eq-perturbation-data}. Use the actions $\sigma_C \times \vartheta_X$ to define an action of $\mathbb{Z}/p \times S^1$ on $C \times X$.

\begin{defn}[Compare \cref{defn:eq-perturbation-data}]\label{defn:S1-eq-perturbation data}
    A collection of \emph{$(\mathbb{Z}/p \times S^1)$-equivariant perturbation data} is a perturbation data
    \begin{align}
        \nu_{S^1,X}^{eq} \in C^\infty \left( (S^\infty \times S^\infty) \times_{\mathbb{Z}/p \times S^1} (C \times X) ; \mathrm{Hom}^{0,1}(TC, TX) \right)
    \end{align}
    where $TC$ and $TX$ are pulled back to the homotopy orbits $(S^\infty \times S^\infty) \times_{\mathbb{Z}/p \times S^1} (C \times X)$, and $TX$ is given the structure of a complex vector bundle by means of $J_{S^1 \cdot \sw}$ along the restriction $(S^\infty \times (S^1 \cdot \sw)) \times_{\mathbb{Z}/p \times S^1} (C \times X)$ for the orbit $S^1 \cdot \sw$ of $\sw$.
\end{defn}

Indeed, the $(\mathbb{Z}/p \times S^1)$-equivariant perturbation data can be considered as $BS^1$-parametrized version of the previously discussed $\mathbb{Z}/p$-equivariant perturbation data, now valued in $J_{\sw}$-complex antilinear homomorphisms from $TC$ to $TX$ (as opposed to having fixed $J$).

\begin{defn}[Compare \cref{defn:eq-map-moduli}]\label{defn:S1-eq-map-moduli}
    Fix $A \in H_2(X;\mathbb{Z})$, $i \ge 0$, $j \ge 0$. The \emph{$(\mathbb{Z}/p \times S^1)$-equivariant map moduli space of degree $A$ and equivariant degree $(i, 2j)$}, or the \emph{$(i, 2j)$th equivariant map moduli space of degree $A$} denoted by $\mathcal{M}_{A}^{eq, i , 2j} := \mathcal{M}_{A}^{eq, i, 2j}(X; \nu_{S^1, X}^{eq})$ is the set
    \begin{equation}
        \mathcal{M}_{A}^{eq, i, 2j} := \left\{ \left( w \in \Delta^i, \sw \in \sDelta^{2j}, u : C \to X \right) : \overline{\partial}_{J_\sw} u = (w \times \sw \times \widetilde{u})^* \nu_{S^1, X}^{eq}, \ u_*[C] = A \right\}.
    \end{equation}
\end{defn}

A generic choice of $J^{eq}$ and $\nu^{eq}_{S^1, X}$ renders these moduli spaces regular of dimension $i + 2j + \dim_{\mathbb{R}}\mathcal{M}_A$, where $\mathcal{M}_A$ is the non-equivariant moduli space of \cref{defn:noneq-map-moduli}. The proof is again the combination of parametrized version of the standard transversality result and the observation that only countably many genericity conditions are imposed (so that their intersection remains generic in the Baire sense). 

These moduli spaces carry evaluation maps
\begin{align}
    \mathrm{ev}^{eq, i, 2j} : \mathcal{M}_A^{eq, i, 2j} &\to \overline{\sDelta}^{2j} \times \left( X \times (\overline{\Delta}^i \times X^p ) \times X \right) \\
    (w, \sw, u) &\mapsto \left( \sw ; u(z_0); (w, u(z_1), \dots, u(z_p)); u(z_\infty) \right)
\end{align}
and admit parametrized analogues of the stable maps for their compactification.

\begin{defn}[Compare \cref{defn:eq-cycle}]\label{defn:S1-eq-cycle}
    Fix a (non-equivariant) incidence cycle $\mathcal{Y} \subseteq X \times X^p \times X$ such that $Y_1 = \cdots = Y_p$. A \emph{$(\mathbb{Z}/p \times S^1)$-equivariant incidence cycle} corresponding to $\mathcal{Y}$ is a smooth map
    \begin{equation}
        \mathcal{Y}^{eq}_{S^1} : E\mathbb{Z}/p \times ES^1 \times \mathcal{Y} \to ES^1 \times \left( X \times (E\mathbb{Z}/p \times X^p) \times X \right)
    \end{equation}
    such that
    \begin{itemize}
        \item For each $(w, \sw) \in E\mathbb{Z}/p \times ES^1 = S^\infty \times S^\infty$, the restriction $\mathcal{Y}^{eq}_{w , \sw} := \mathcal{Y}^{eq}_{S^1}|_{(w, \sw) \times \mathcal{Y}}$ maps into $\{\sw\} \times \left( X \times (\{w\} \times X^p )\times X \right)$, and hence $\mathcal{Y}^{eq}_{w, \sw} : \mathcal{Y} \to X \times X^p \times X$;
        \item For each $\sw \in ES^1$, the restriction $\mathcal{Y}^{eq}_{S^1}|_{E\mathbb{Z}/p \times \{ \sw \} \times \mathcal{Y}}$ is a $\mathbb{Z}/p$-equivariant incidence cycle in the sense of \cref{defn:eq-cycle};
        \item Finally, understood as maps from $\mathcal{Y}$, we have $\mathcal{Y}^{eq}_{\sigma \cdot w, \vartheta \cdot \sw} = \vartheta_X \circ \sigma_{X^p} \circ \mathcal{Y}^{eq}_{w, \sw}$ where $\vartheta_X$ denotes the diagonal action of $\vartheta_X$ on all factors of $X \times X^p \times X$.
    \end{itemize}
\end{defn}

Given a $(\mathbb{Z}/p \times S^1)$-equivariant incidence cycle, denote by
\begin{equation}
    \mathcal{Y}^{eq, i, 2j} := \mathcal{Y}^{eq}_{S^1} |_{\Delta^i \times \sDelta^{2j} \times \mathcal{Y}} : \Delta^i \times \sDelta^{2j} \times \mathcal{Y} \to \overline{\sDelta}^{2j} \times \left( X \times (\overline{\Delta}^i \times X^p ) \times X \right)
\end{equation}
its restriction to the parameter space $\Delta^i \times \sDelta^{2j} \subseteq E\mathbb{Z}/p \times ES^1$. We refer to this as the $(i, 2j)$th equivariant incidence cycle. 

\begin{lem}[Compare \cref{lem:eq-transversality}, \cref{lem:eq-transverse-incidence-cycles}]\label{lem:S^1-eq-transversality}
    Given a sequence of regular $(\mathbb{Z}/p \times S^1)$-equivariant moduli spaces $\{\mathcal{M}_A^{eq, i, 2j}\}_{i, j \ge 0}$, regular equivariant simple stable maps, and a choice of incidence cycle $\mathcal{Y}$ such that $Y_1 = \cdots = Y_p$, there exists an $(\mathbb{Z}/p \times S^1)$-equivariant incidence cycle $\mathcal{Y}^{eq}_{S^1}$ such that each $\mathcal{Y}^{eq, i, 2j}$ is transverse to evaluation maps from $\mathcal{M}_A^{eq, i, 2j}$ and the associated simple stable maps.
\end{lem}
\begin{proof}
We proceed as in \cref{lem:eq-transversality}, proving parametrized transversality inductively for increasing $j$. The inductive step follows from the following two observations.

(i) First note that the transversality for (evaluation maps from) $(i, 2j)$th moduli spaces automatically implies transversality for $(i, 2j+1)$th moduli spaces, defined using the parameter spaces $\sDelta^{2j+1}$, by $S^1$-equivariance. That is, note that $\sDelta^{2j+1} = \bigcup_{\vartheta \neq 1} \vartheta \cdot \sDelta^{2j}$, and that transversality holds in the $\sDelta^{2j}$-parametrized sense for every $\vartheta \cdot \sDelta^{2j}$ by the inductive assumption. A fortiori this implies transversality in the $\sDelta^{2j+1}$-parametrized sense. 

(ii) Also note that the boundary of the even cell $\sDelta^{2j+2}$ is a copy of $S^{2j+1}$, so that the argument of \cref{lem:eq-transversality} (for the cells $\Delta^i$ with $i$ even) passes through. That is, the transversality for the $(i, 2j)$th and $(i, 2j+1)$th moduli spaces together imply the transversality for $(i, 2j+2)$th moduli spaces.
\end{proof}
Hence if the transversal intersection $\mathrm{ev}^{eq, i, 2j}(\mathcal{M}_A^{eq, i, 2j}) \pitchfork \mathcal{Y}^{eq, i, 2j}$ is $0$-dimensional, its count in $\mathbb{F}_p$ may be used to define the structure constants of operations on $(\mathbb{Z}/p \times S^1)$-equivariant quantum cohomology.

\subsection{$S^1$-equivariant quantum Steenrod operations}\label{ssec:S1-QSt-defn}
Using the $S^1$-equivariant analogues $\mathcal{M}_A^{eq, i, 2j}$ of the previous equivariant moduli spaces, we may now define the $S^1$-equivariant analogues of the quantum Steenrod operations. We begin by reviewing the $S^1$-equivariant quantum product.

Let $(X, \omega)$ be a symplectic manifold satisfying the positivity condition from \cref{assm:pos-condition}, admitting an $S^1$-action preserving the symplectic form $\omega$, such that the fixed locus $F = X^{S^1}$ is compact. As in \eqref{eqn:borel-eq-coh}, $H^*_{S^1}(X; \mathbb{F}_p) := H^*(ES^1 \times_{S^1} X ; \mathbb{F}_p)$ denotes the Borel equivariant cohomology ring of $X$.

We assume that $F = X^{S^1}$ is a connected embedded submanifold of $X$, and denote the inclusion by $\iota_F : F \to X$. Under such assumptions, the fixed locus $F$ has a well-defined normal bundle $N_{F/X}$ that is naturally a $S^1$-equivariant vector bundle over $F$. After choosing orientations for $X$ and $F$, the normal bundle has an equivariant Euler class $e_{S^1}(N_{F/X}) \in H^*_{S^1}(F)$. 

The $S^1$-equivariant Poincar\'e pairing is defined on $H^*_{S^1}(X;\mathbb{F}_p)$ by means of localization theorems in equivariant cohomology (as in e.g. \cite[Chapter 3]{AP93}). To apply the localization theorem for $S^1$ with $\mathbb{F}_p$-coefficients, we further assume the following about the $S^1$-action on $X$:
\begin{assm}\label{assm:S^1-action-semifree}
\begin{equation}
    \mbox{the isotropy subgroup of any orbit of the } S^1\mbox{-action is } \{1 \} \mbox { or } S^1.
\end{equation}
\end{assm}
Note that the assumption is always satisfied for our examples, where the $S^1$-action is given on $X = T^*M$ for a complex manifold $M$ as the rotation of cotangent fibers. 
\begin{lem}[Localization; see {\cite[Section 3.2, 5.3]{AP93}}]
    Under \cref{assm:S^1-action-semifree}, the natural restriction map $\iota_F^* : H^*_{S^1}(X) \to H^*_{S^1}(F)$ and the Gysin map $(\iota_F)_{!} : H^*_{S^1}(F) \to H^{*+\mathrm{rk}(N_{F/X})}(X)$ become isomorphisms after tensoring with the fraction field $\mathrm{Frac}(H^*_{S^1}(\mathrm{pt};\mathbb{F}_p)) = \mathbb{F}_p(\!(\teq)\!)$. 
\end{lem}
\begin{proof}
By choosing an $S^1$-invariant metric, fix a $S^1$-invariant tubular neighborhood of $F = X^{S^1}$. The assumption ensures that the complement $X \setminus F$ is covered by union of (open neighborhoods of) free $S^1$-orbits. (Without the assumption, there may be orbits with nontrivial proper isotropy groups.) The free orbits have $S^1$-equivariant cohomology that are $\teq$-torsion. The result follows from Mayer--Vietoris sequence for equivariant cohomology.
\end{proof}

By the localization theorem, by passing to the fraction field $\mathrm{Frac}(H^*_{S^1}(\mathrm{pt};\mathbb{F}_p)) = \mathbb{F}_p(\!(\teq)\!)$, the Euler class $e(N_{F/X}) = (\iota_F)^*(\iota_F)_! \ 1 \in H^*_{S^1}(F)$ becomes invertible.

\begin{defn}\label{defn:S1-Poincare-pairing}
    The \emph{$S^1$-equivariant Poincar\'e pairing} is a pairing $(\cdot, \cdot)_{S^1} : H^*_{S^1}(X;\mathbb{F}_p) \otimes H^*_{S^1}(X;\mathbb{F}_p) \to \mathbb{F}_p(\!(\teq)\!)$ given by
    \begin{equation}\label{eqn:S1-Poincare-pairing}
        (a, b)_{S^1} = \int_{X} a \cup b := \int_{F} \frac{\iota_F^*(a \cup b)}{e_{S^1}(N_{F/X})} \in \mathbb{F}_p(\!(\teq)\!).
    \end{equation}
\end{defn}

The $S^1$-equivariant Poincar\'e pairing is nondegenerate if $F$ is compact oriented, and for a fixed $b \in H^*_{S^1}(X;\mathbb{F}_p)\otimes \mathbb{F}_p(\!(\teq)\!)$ its linear dual under the Poincar\'e pairing is again denoted $b^\vee \in H^*_{S^1}(X;\mathbb{F}_p) \otimes \mathbb{F}_p(\!(\teq)\!)$.

The $S^1$-equivariant quantum product is defined using the $S^1$-equivariant Gromov--Witten invariants and the Poincar\'e pairing (see e.g. \cite[Section 4]{BMO11}). We further assume that $(X, \omega)$ is closed, for the Gromov--Witten invariants to be well-defined. (This assumption will be relaxed later for $X = T^*\mathbb{P}^1$.)

\begin{defn}\label{defn:S1-quantum-product}
    The \emph{$S^1$-equivariant quantum product} $\ast_{S^1}$ is defined for $a, b \in H^*_{S^1}(X;\mathbb{F}_p)$ by means of the Poincar\'e pairing $(\cdot, \cdot)_{S^1} : H^*_{S^1}(X;\mathbb{F}_p) \otimes H^*_{S^1}(X;\mathbb{F}_p) \to \mathbb{F}_p(\!(\teq)\!)$ as
    \begin{equation}
        a \ast_{S^1} b = \sum_{A} \langle a, b, c \rangle_{A}^{S^1} \ q^A \ c^\vee \in \Lambda_{\mathbb{F}_p}(\!(\teq)\!).
    \end{equation}
    Here $\langle a, b, c \rangle_A^{S^1} \in \mathbb{F}_p[\![\teq]\!]$ is the $S^1$-equivariant $3$-pointed Gromov--Witten invariant of degree $A$, defined from $S^1$-equivariant versions of the moduli space of $3$-pointed genus $0$ curves and equivariant incidence cycles coming from Poincar\'e dual cycles of $a, b, c$. The sum is well-defined as an element of $\Lambda_{\mathbb{F}_p}(\!(\teq)\!)$ by Gromov compactness. The product is extended $q$-linearly to $H^*_{S^1}(X;\Lambda)$. 
\end{defn}

\begin{defn}\label{defn:S1-QH}
    The \emph{$(\mathbb{Z}/p \times S^1)$-equivariant quantum cohomology algebra of $X$} is the ring
    \begin{equation}
        QH^*_{\mathbb{Z}/p \times S^1}(X;\mathbb{F}_p) := \left( H^*_{S^1}(X;\Lambda)[\![t, \theta]\!] \otimes \mathbb{F}_p(\!(\teq)\!), \ast_{S^1} \right),
    \end{equation}
    with the quantum product $\ast_{S^1}$ extended linearly for the $\mathbb{Z}/p$-equivariant parameters $t, \theta$.
\end{defn}

Take $QH^*_{\mathbb{Z}/p \times S^1}(X;\mathbb{F}_p)$, and fix a cohomology class $b \in \mathrm{PD}[Y] \in H^*(X;\mathbb{F}_p)$ that is (mod $p$ reduction of) Poincar\'e dual to a homology cycle represented by an $S^1$-invariant embedded submanifold $Y \subseteq X$. Similarly fix cohomology classes $b_0, b_\infty \in H^*(X;\mathbb{F}_p)$ Poincar\'e dual to $S^1$-invariant submanifolds $Y_0, Y_\infty \subseteq X$. Note that by $S^1$-invariance of the submanifolds $Y_j$, these cohomology classes may be considered as equivariant cohomology classes in $H^*_{S^1}(X;\mathbb{F}_p)$, dual to the Borel constructions $ES^1 \times_{S^1} Y_j \subseteq ES^1 \times_{S^1} X$.

Fix the corresponding non-equivariant incidence cycle $\mathcal{Y} = Y_0 \times (Y \times \cdots \times Y) \times Y_\infty \subseteq X \times X^p \times X$. Choose an equivariant incidence cycle $\mathcal{Y}_{S^1}^{eq}$. Note that by definition for any choice of $w \in E\mathbb{Z}/p$, the image of $\mathcal{Y}^{eq}_{S^1}|_{\{w\} \times ES^1 \times \mathcal{Y}}: ES^1 \times \mathcal{Y} \to ES^1 \times (X \times X^p \times X)$ has $S^1$-quotient that is isotopic to the fiberwise product of the Borel construction $ES^1 \times_{S^1} Y_j \to BS^1$. 

Assume that the data $(J^{eq}, \nu_{S^1, X}^{eq})$, and an equivariant incidence cycle $\mathcal{Y}^{eq}$ are chosen so that the counts $\mathcal{M}_A^{eq, i, 2j} \pitchfork \mathcal{Y}^{eq, i, 2j} \in \mathbb{F}_p$ of curves satisfying the incidence constraints is well-defined.
\begin{defn}\label{defn:S1-QSt}
    For $b \in H^*_{S^1}(X;\mathbb{F}_p)$, the \emph{$S^1$-equivariant quantum Steenrod operation} is a map
    \begin{equation}
        Q\Sigma_b^{S^1} : QH^*_{\mathbb{Z}/p \times S^1} (X;\mathbb{F}_p) \to QH^*_{\mathbb{Z}/p \times S^1} (X;\mathbb{F}_p) 
    \end{equation}
    defined for $b_0 \in H^*_{S^1}(X;\mathbb{F}_p)$ as the sum
    \begin{equation}
        Q\Sigma_b^{S^1}(b_0) = \sum_{b_\infty} \sum_A \sum_i (-1)^\star \left( \# \mathrm{ev}^{eq, i, 2j}(\mathcal{M}_A^{eq, i, 2j}) \pitchfork \mathcal{Y}^{eq, i, 2j} \right) q^A \ (t, \theta)^i \ \teq^{j} \ b_\infty^\vee \in QH^*_{\mathbb{Z}/p \times S^1} (X;\mathbb{F}_p).
    \end{equation}
    The first sum is over a chosen basis of $QH^*_{S^1}(X;\mathbb{F}_p)$ as a $\mathbb{F}_p(\!(\teq)\!)$-module, and $b_\infty^\vee \in H^*_{S^1}(X;\mathbb{F}_p) \otimes \mathbb{F}_p(\!(\teq)\!)$ denotes the dual of $b_\infty \in H^*_{S^1}(X;\mathbb{F}_p)$ under the $S^1$-equivariant Poincar\'e pairing. 
\end{defn}
The structure constants given by the counts are defined to be zero unless
\begin{equation}
    i + 2j + \dim_{\mathbb{R}} X + 2c_1(A) = |b_0| + p|b| + |b_\infty|,
\end{equation}
so that the intersection $ \mathrm{ev}^{eq, i, 2j}(\mathcal{M}_A^{eq, i, 2j}) \pitchfork \mathcal{Y}^{eq, i, 2j}$ is indeed discrete. We extend the operation linearly in the $q$-parameters and $(t, \theta)$-parameters to allow $b_0 \in QH^*_{\mathbb{Z}/p \times S^1}(X;\mathbb{F}_p)$ as an input. The sign $\star$ is as in \eqref{eqn:sign-in-QSt}.

By the dimension constraint, $Q\Sigma_b^{S^1}$ is a map of graded vector spaces (over the graded field $\mathbb{F}_p(\!(\teq)\!)$) of degree $p|b|$.

TQFT-type arguments identical to those in \cite[Section 4]{SW22} may be used to establish the following properties for the $S^1$-equivariant setting. The proofs rely on Morse-theoretic chain-level models for defining the ($S^1$-equivariant) quantum Steenrod operations, and are immediate generalizations of their analogues for the usual quantum Steenrod operations. For this reason, we have decided to highlight only the (very mildly) new aspects that arise in the $S^1$-equivariant case.
\begin{prop}[Compare \cref{prop:QSt-properties}]\label{prop:S1-QSt-properties}
    For $b \in H^*_{S^1}(X;\mathbb{F}_p)$, the operations $Q\Sigma_b^{S^1}$ satisfy the following properties.
        \begin{itemize}
        \item (i) $Q\Sigma_{1}^{S^1} = \mathrm{id}$ for $1 \in H^0(X;\mathbb{F}_p)$ the unit of cohomology algebra.
        \item (ii) $Q\Sigma_b^{S^1}(c)|_{\teq = 0} = Q\Sigma_b(c)$, hence $Q\Sigma_b^{S^1}$ deforms the quantum Steenrod operations of \cref{defn:QSt-structure-constants} in the $S^1$-equivariant parameters.
        \item (iii) $Q\Sigma_b^{S^1}(c)|_{t = \theta = 0} = \overbrace{b \ast_{S^1} \cdots \ast_{S^1} b}^{p} \ \ast_{S^1} \ c$, hence $Q\Sigma_b^{S^1}$ deforms $p$-fold $S^1$-equivariant quantum product in the $\mathbb{Z}/p$-equivariant parameters.
        \item (iv) $Q\Sigma_b^{S^1}(c)|_{q^{A \neq 0}=0} = \mathrm{St}(b) \cup_{S^1} c$, hence $S^1$-equivariant quantum Steenrod operations deform classical Steenrod operations (for the cohomology of $ES^1 \times_{S^1} X$) in quantum parameters.
    \end{itemize}
\end{prop}
\begin{proof}[Proof sketch]
    (ii) is a consequence of the observation that projecting to the $\teq$-constant term amounts to considering all cohomology classes $b \in H^*_{S^1}(X;\mathbb{F}_p)$ by their non-equivariant counterpart $b|_{h=0} \in H^*(X;\mathbb{F}_p)$ and discarding for all $j>0$ the $(i, 2j)$th $(\mathbb{Z}/p \times S^1)$-equivariant moduli spaces with incidence constraints. The $(i,0)$th $(\mathbb{Z}/p \times S^1)$-equivariant moduli spaces exactly recover the $i$th $\mathbb{Z}/p$-equivariant moduli spaces.

    (i), (iii), (iv) are immediate generalizations of the results \cite[Proposition 4.8, Lemma 3.5, Lemma 4.6]{SW22}, which use Morse-theoretic models for their proofs. The Morse model can be adopted immediately to the $S^1$-equivariant setting by taking $BS^1$-parametrized perturbations of a fixed Morse function to define Morse functions on $ES^1 \times_{S^1} X \to BS^1$ (compare \cref{defn:S1-eq-cycle}). 
    
    For statement (iii), we consider a degeneration of the domain curve $C$ to a length $p$ linear chain of $3$-pointed genus $0$ curves. Each of such $3$-pointed curve is the domain for the maps that define the structure constants for the (Morse-theoretic chain-level) quantum product. We claim that the associated degeneration of the moduli space (with incidence constraints $b$ along $z_1, \dots, z_p \in C$) indeed counts the $p$-fold iterated $S^1$-equivariant quantum product (with $b$). For this, it suffices to observe that linear subspaces $\mathbb{CP}^{j_1}, \dots, \mathbb{CP}^{j_p}$ satisfying $j_1 + \cdots + j_p = j$ generically intersect at a single point in $\mathbb{CP}^j$. Therefore for every contribution from a map from length $p$ chain of genus $0$ curves in the expression $b \ast_{S^1} \cdots \ast_{S^1} b \ast_{S^1} c$ where each $b\ast_{S^1}$ provides a factor of $\teq^{j_1}, \dots, \teq^{j_p}$, there is exactly one contribution from a map from $C$ in $Q\Sigma_b^{S^1}(c)|_{t=\theta=0}$ carrying $\teq^{j_1 + \cdots + j_p} = \teq^j$ obtained by reversing the degeneration, and this is the desired result.
\end{proof}

\subsection{Covariant constancy for the $S^1$-equivariant quantum connection}\label{ssec:S1-cov-constancy}
By replacing the quantum product $\ast$ with the $S^1$-equivariant quantum product $\ast_{S^1}$, we obtain the $S^1$-equivariant quantum connection. The assumption on the manifold $(X, \omega)$ and the $S^1$-action on $X$ is as in the previous \cref{ssec:S1-QSt-defn}.

Again fix $a \in H^2_{S^1}(X;\mathbb{Z})$ to define the differentiation operator $\partial_a : \Lambda[\![t, \theta]\!] \to \Lambda[\![t, \theta]\!]$. Note that choosing $\teq \in H^2_{S^1}(X;\mathbb{Z})$ amounts to the zero operator (since $\teq \cdot A$ = 0), so there is actually no loss of generality by restricting to the case $a \in H^2(X;\mathbb{Z})$ as in the previous definition, \cref{defn:quantum-connection}.

\begin{defn}\label{defn:S1-Qconn}
The \emph{(small) ($S^1$-)equivariant quantum connection} of $X$ is the collection of endomorphisms $\nabla^{S^1}_a : QH^*_{\mathbb{Z}/p \times S^1} (X;\mathbb{F}_p) \to QH^*_{\mathbb{Z}/p \times S^1}(X;\mathbb{F}_p)$  for each $a \in H^2_{S^1}(X;\mathbb{Z})$ defined by
\begin{equation}
    \nabla_a^{S^1} \beta = t \partial_a \beta - a \ast_{S^1} \beta.
\end{equation}
\end{defn}

The $S^1$-equivariant quantum connection is the main object for study in the equivariant mirror symmetry program, and above we consider its version with $\mathbb{F}_p$-coefficients. 

The following proposition generalizes the covariant constancy relation of \cite{SW22}, and establishes the fact that the operations $Q\Sigma_b^{S^1}$ may be used to find covariantly constant endomorphisms for the equivariant quantum connection.

\begin{prop}\label{prop:S1-cov-constancy}
    For any choice of $b \in H^*(X;\mathbb{F}_p)$, the $S^1$-equivariant quantum Steenrod operation $Q\Sigma_b^{S^1}$ is a covariantly constant endomorphism for the $S^1$-equivariant quantum connection, that is it satisfies
    \begin{equation}
        \nabla_a^{S^1} \circ Q\Sigma_b^{S^1} - Q\Sigma_b^{S^1} \circ \nabla_a^{S^1} = 0
    \end{equation}
    for any $a \in H^2_{S^1}(X;\mathbb{Z})$.
\end{prop}
\begin{proof}
    The proof is a direct generalization of the proof for the covariant constancy of $Q\Sigma_b$ with respect to $\nabla_a$, given in \cite[Section 5]{SW22}. The sketch of their proof was provided in \cref{thm:covariant-constancy}. Here we mention the necessary modifications for the $S^1$-equivariant setting, adopting the notation from \cref{thm:covariant-constancy}.

    The analogue of the moduli spaces $\mathcal{M}^{eq, i}_A(a)$ with an extra free marked point are given by $\mathcal{M}^{eq, i, 2j}_A(a)$, the only difference being that the almost complex structure $\{J_\sw\}_{\sw \in ES^1}$ is allowed to vary (over $BS^1$). Collecting the zero-dimensional moduli spaces $\mathcal{M}^{eq, i, 2j}_A(a)$ as structure constants defines the operation $\partial_a Q\Sigma_b^{S^1}$.

    We consider the two strata in the moduli spaces $\mathcal{M}^{eq, i, 2j}_A(a)$ where the free marked point collides with $z_0 = 0$ or $z_\infty = \infty$ in $C$. It remains to note that, as in the discussion in \cref{prop:S1-QSt-properties}, the operations obtained by taking the structure constants to be the counts of solutions from such strata can be identified with $Q\Sigma_b^{S^1} (a \ast_{S^1} -)$ and $a \ast_{S^1} Q\Sigma_b^{S^1}(-)$, respectively.

    By using the relationship $t [C] = [z_\infty] - [z_0]$ in a distinguished choice (\cite[Section 2c]{SW22}) of the cellular chain complex for $C$, one obtains
    \begin{equation}
        t \partial_a Q\Sigma_b^{S^1}(-) = a \ast_{S^1} Q\Sigma_b^{S^1}(-) - Q\Sigma_b^{S^1}(a\ast_{S^1} - )
    \end{equation}
    which is equivalent to the desired result.
\end{proof}

The theorem will be used in the next section to determine a covariantly constant endomorphism for the $S^1$-equivariant quantum connection of the simplest symplectic resolution (the case of $T^*\mathbb{P}^1$).

\section{Endomorphisms for the equivariant quantum connection of $T^*\mathbb{P}^1$}\label{sec:T*P1}
By the $S^1$-equivariant analogue of covariant constancy (\cref{prop:S1-cov-constancy}), the task of constructing a covariantly constant endomorphism for the equivariant quantum connection is reduced to the task of fully computing the operations $Q\Sigma_b^{S^1}$. We execute this task in this section.

\subsection{The $S^1$-equivariant quantum connection of $T^*\mathbb{P}^1$ in characteristic $p$}
We first compute the $S^1$-equivariant quantum connection for the geometry of $T^*\mathbb{P}^1$.

Fix $p >2$. Consider $X = T^*\mathbb{P}^1$ with its standard K\"ahler form. Note that this is \emph{not} the standard (complex) Liouville symplectic form on the cotangent bundle, and the zero section $\mathbb{P}^1$ is a symplectic, not a Lagrangian, submanifold for this symplectic structure. On $T^*\mathbb{P}^1$, we consider the $\mathbb{C}^*$-action which dilates the cotangent fibers and restrict to the $S^1$-action rotating the fibers under the equatorial inclusion $S^1 \subseteq \mathbb{C}^*$. 

Observe that $T^*{\mathbb{P}^1}$ deformation retracts to the base, and the deformation retract can also be chosen to be $S^1$-equivariant. Hence $H^*_{S^1}(T^*\mathbb{P}^1) \cong H^*_{S^1}(\mathbb{P}^1)$. The $S^1$-action on the base $\mathbb{P}^1$ is trivial, so by the K\"unneth formula we additively have
\begin{equation}
    H^*_{S^1}(T^*\mathbb{P}^1;\mathbb{F}_p) \cong H^*_{S^1}(\mathbb{P}^1;\mathbb{F}_p) = \mathbb{F}_p[\![\teq]\!] \cdot 1 \oplus \mathbb{F}_p [\![ \teq] \!] \cdot b,
\end{equation}
for $1 \in H^0_{S^1}(T^*\mathbb{P}^1)$ (the Poincar\'e dual of the fundamental cycle) and $b \in H^2_{S^1}(T^*{\mathbb{P}^1})$ (the Poincar\'e dual of the fiber). The ring structure is undeformed, so $b^2 =0$. We are interested in determining the operation $Q\Sigma_b^{S^1}$ for the fiber class $b \in H^2_{S^1}(T^*\mathbb{P}^1)$. 

\begin{lem}
    The $S^1$-equivariant Poincar\'e pairing on $T^*\mathbb{P}^1$ is given by
    \begin{equation}\label{eqn:S1-Poincare-pairing-T*P1}
        ( \cdot, \cdot)_{S^1} = \begin{pmatrix} (1,1)_{S^1} & (b,1)_{S^1} \\ (1,b)_{S^1}& (b,b)_{S^1} \end{pmatrix} = \begin{pmatrix}
           2 \teq^{-2} & \teq^{-1} \\ \teq^{-1} & 0
        \end{pmatrix}.
    \end{equation}
\end{lem}
\begin{proof}
    It suffices to describe the pairing for the basis elements $1$ and $b$, as the pairing is $\mathbb{F}_p[\![\teq]\!]$-bilinear. The pairing (see \cref{defn:S1-Poincare-pairing}) is defined in terms of the fixed point localization theorem, so we must determine the Euler class of the normal bundle of the fixed locus of the $S^1$-action. The fixed locus is just the zero section $\mathbb{P}^1 \subseteq T^*\mathbb{P}^1$.

    As a complex vector bundle, the normal bundle is given by $N_{\mathbb{P}^1/T^*\mathbb{P}^1} \cong \mathcal{O}_{\mathbb{P}^1}(-2)$, and this carries a weight $1$ action by our choice of the $S^1$-action on $T^*\mathbb{P}^1$. The equivariant Euler class can be computed by twisting the non-equivariant line bundlle $\mathcal{O}(-2)$ with this weight $1$ representation rotating the fibers, so that $e(N_{\mathbb{P}^1/T^*\mathbb{P}^1}) = -2b + \teq$ (under our conventions for $\teq$). In the localized ring $H^*_{S^1}(\mathbb{P}^1) \otimes \mathbb{F}_p(\!(\teq)\!)$, the inverse of the Euler class is given by $e(N_{\mathbb{P}^1/T^*\mathbb{P}^1})^{-1} = \teq^{-1} ( 1 + 2\teq^{-1} b)$. The desired result now follows from computing using \eqref{eqn:S1-Poincare-pairing}.
\end{proof}

We proceed to compute the $S^1$-equivariant quantum connection for $T^*\mathbb{P}^1$. Note that $H_2(T^*\mathbb{P}^1 ; \mathbb{Z}) \cong H_2(\mathbb{P}^1;\mathbb{Z}) \cong \mathbb{Z}$ so the symbols $q^A \in \Lambda$ can be identified with monomials $q^d$ for $d \ge 0$. Note that the target manifold $T^*\mathbb{P}^1$ is not compact. Nevertheless, by reduction to the local $\mathbb{P}^1$ geometry (\cref{lem:nbhd-of-T*P1}) and maximum principle for the local $\mathbb{P}^1$ (see discussion in \cref{sssec:map-moduli-cptness}), its $S^1$-equivariant Gromov--Witten invariants and hence the quantum connection are well-defined.

\begin{lem}
    The $S^1$-equivariant quantum connection for $T^*\mathbb{P}^1$ is given by
    \begin{equation}
        \nabla_b^{S^1} = t \partial_b - b \ \ast_{S^1} = t q \partial_q  - \begin{pmatrix}
            0 & \frac{\teq^2 q}{1-q} \\ 1 & \frac{-2\teq q}{1-q} 
        \end{pmatrix}  .
    \end{equation}
\end{lem}
\begin{proof}
    Since $1$ remains a unit for the $S^1$-equivariant quantum product, it suffices to compute $b \ast_{S^1} b$ to determine the connection matrix. Due to the key geometric observation \cref{lem:nbhd-of-T*P1}, the computation of the structure constants (the three-pointed $S^1$-equivariant Gromov--Witten invariants) is reduced to the Aspinwall--Morrison multiple cover formula for the local $\mathbb{P}^1$, from \cite{AM93}, \cite{Voi96}. The result is that
    \begin{equation}
        ( b \ast_{S^1} b, b )_{S^1} = \sum_{d \ge 1} \langle b, b, b \rangle_{d[\mathbb{P}^1]}^{S^1} q^d = \teq \sum_{d \ge 1} q^d = \frac{\teq q}{1-q}.
    \end{equation}
    By considering the dual $b^\vee = -2b + \teq$ for the basis $\{1, b\}$ under the equivariant Poincar\'e pairing \eqref{eqn:S1-Poincare-pairing-T*P1}, we obtain the desired result
    \begin{equation}
        b \ast_{S^1} b = \frac{\teq q}{1-q} (-2b + \teq) = \frac{\teq^2 q}{1-q} \cdot 1 + \frac{-2\teq q}{1-q} \cdot b.
    \end{equation}
\end{proof}

By \cref{prop:S1-cov-constancy}, $Q\Sigma_b^{S^1}$ will be a covariantly constant endomorphism for $\nabla_b^{S^1}$. Our next goal is to compute the structure constants to determine this endomorphism.

\subsection{Reduction to local $\mathbb{P}^1$ geometry}
The computation of all the structure constants of $Q\Sigma_b^{S^1}$ may appear to be a much harder task than that of its non-$S^1$-equivariant analogue $Q\Sigma_b$. Here we describe how the computation for the $T^*\mathbb{P}^1$ case can in fact be reduced to the previous example of local $\mathbb{P}^1$ CY3, and computed with the same method (mildly generalized to incorporate the $S^1$-action).

\subsubsection{Deformations of $T^*\mathbb{P}^1$}
The key observation for this reduction is the following, which is a classical result going back at least to Atiyah \cite{Ati58}.
\begin{lem}\label{lem:nbhd-of-T*P1}
    Consider $X = (T^*\mathbb{P}^1, J_X)$ with its integrable complex structure $J_X$. The local $\mathbb{P}^1$ CY3 $(Z = \mathrm{Tot}(\mathcal{O}(-1)^{\oplus 2} \to \mathbb{P}^1), J_Z)$ is biholomorphic to the total space of a family of smooth complex varieties $\mathcal{X} \to \mathbb{C}$, with every fiber $\mathcal{X}_s$, $s \in \mathbb{C}$ diffeomorphic to $X$, and the central fiber $\mathcal{X}_0$ biholomorphic to $(X, J_X)$. In particular, the complex structure on each fiber $\mathcal{X}_s$ can be thought of as an almost complex structure $J_s$ on $X$.
\end{lem}
\begin{proof}[Proof sketch]
    This is a well-known result which goes under the name of ``small/simultaneous resolution.'' The key observation is that the exceptional divisor in the blowup of the ordinary double point in dimension $3$, $\{x^2+y^2+z^2 = t^2 \} \subseteq \mathbb{C}^3_{x,y,z} \times \mathbb{C}_t$, is a smooth quadric $Q = \mathbb{P}^1 \times \mathbb{P}^1 \subseteq \mathbb{P}^3$. Hence by collapsing one factor of $\mathbb{P}^1$ in the exceptional divisor of the blowup $\mathrm{Bl} \subseteq (\mathbb{P}^1 \times \mathbb{P}^1) \times \mathbb{C}^4$ we obtain the small resolution $Z \subseteq \mathbb{P}^1 \times \mathbb{C}^4$. The corresponding projection maps $Z \to \mathbb{P}^1$ and $Z \to \mathbb{C}_t$ are identified with $Z = \mathrm{Tot}(\mathcal{O}(-1)^{\oplus 2} \to \mathbb{P}^1)$ and $\mathcal{X} \to \mathbb{C}$, respectively. There are two choices for $Z$ (corresponding to the choice of the factor of $\mathbb{P}^1 \times \mathbb{P}^1$ to collapse) related by a birational transformation; we choose one once and for all.
\end{proof}

To compute the operations $Q\Sigma_b^{S^1}$, one must fix a $S^1$-equivariant collection of almost complex structures (\cref{defn:S1-acs}) on $X = T^*\mathbb{P}^1$. By \cref{lem:nbhd-of-T*P1}, there is a family of almost complex structures $J_s$ on $X$ parametrized by $s \in \mathbb{C}$, given by the complex structure on the fiber $\mathcal{X}_s \to \mathbb{C}$ (the ``smoothing'' of the $A_1$ singularity for $s \neq 0$). The base $\mathbb{C}$ can in fact be identified with the $S^1$-equivariant miniversal deformation space of $(X, J_X)$ (see for example \cite[Theorem 1]{KM92}), and the $S^1$-action of $(X, J_X)$ extends to an action on the total space of the whole family $Z = \mathcal{X} \to \mathbb{C}$. We will choose our almost complex structures $\{J_\sw\}_{\sw \in ES^1}$ within this class of deformations, so that for every $\sw \in ES^1$ there exists $s \in \mathbb{C}$ such that $J_\sw = J_s$.

\subsubsection{Locus of the distinguished complex structure of $T^*\mathbb{P}^1$}
We now choose $J^{eq} = \{ J_\sw\}_{\sw \in ES^1}$ satisfying the equivariance condition $J_{\vartheta \cdot \sw} = \vartheta_X^* J_\sw$ within the set $\{J_s\}_{s \in \mathbb{C}}$, where $J_0 = J_X$ is the integrable structure of $X$. The choice $J^{eq}$ can equivalently be regarded as the data of a $S^1$-equivariant map
\begin{equation}
    \mathbf{J} : ES^1 = S^\infty \to \mathbb{C}, \quad \mathbf{J}(\sw) = s \iff J_\sw = J_s.
\end{equation}
where the action of $\vartheta \in S^1$ on $\mathbb{C}$ is by scalar multiplication (rotation).

Since $(X, J_s)$ with $s \neq 0$ is an affine variety (a smoothing of an $A_1$ singularity), the Gromov--Witten invariants for such almost complex structures are trivial. The nontrivial counts therefore only occur in the moduli spaces over $\sw \in BS^1$ such that $\mathbf{J}(\sw) = 0$. (The vanishing locus of $\mathbf{J}$ as a subset of $BS^1 = ES^1 / S^1$ is well-defined, since $0 \in \mathbb{C}$ is the fixed point of the $S^1$-action on $\mathbb{C}$.)

\begin{lem}\label{lem:locus-J-standard}
    Denote the locus of $\sw \in BS^1$ for which $\mathbf{J}(\sw) = 0$ by $V_{\mathbf{J}} \subseteq BS^1$. Restricted to each finite-dimensional approximation $\mathbb{CP}^j = \bigcup_{k=1}^{j} \sDelta^{2j}$, the locus $V_\mathbf{J} \cap \mathbb{CP}^{j}$ can be identified with a hyperplane $\mathbb{CP}^{j-1} \hookrightarrow \mathbb{CP}^j$. 
\end{lem}
\begin{proof}
    Consider $\mathbf{J} : ES^1 \cong S^\infty \to \mathbb{C}$, and extend this uniquely to a fiberwise linear function on the total space of the complex line bundle $\mathcal{O}_{\mathbb{CP}^\infty}(-1) \to \mathbb{CP}^\infty = BS^1$. The vanishing locus of $\mathbf{J}$ in $BS^1 = \mathbb{CP}^\infty$ can then be identified with the vanishing locus of a section of $\mathcal{O}_{\mathbb{CP}^\infty}(1)$, which is a hyperplane $V_{\mathbf{J}} \subseteq \mathbb{CP}^\infty$. Namely, $V_{\mathbf{J}}$ is another copy of $\mathbb{CP}^\infty$ with an embedding $V_\mathbf{J} \subseteq \mathbb{CP}^\infty$ that restricts to an embedding $\mathbb{CP}^{j-1} \hookrightarrow \mathbb{CP}^j$ of some hyperplane.
\end{proof}

\begin{lem}\label{lem:reduction-rank1}
    Consider the vanishing locus $V_{\mathbf{J}} \cap \mathbb{CP}^j \subseteq \mathbb{CP}^j$ from \cref{lem:locus-J-standard}, and consider $V_{\mathbf{J}} \cap \sDelta^{2j}$ as a subset in $ES^1$. For a sufficiently small transversal slice $D$ (biholomorphic to a neighborhood of $0 \in \mathbb{C}$) to the vanishing locus $V_{\mathbf{J}} \cap \sDelta^{2j}$ at a chosen point $\sw' \in V_{\mathbf{J}} \cap \sDelta^{2j}$, the family $(X, J_\sw)_{\sw \in D} \to D$ obtained by restriction can be identified biholomorphically with a neighborhood of $\mathbb{P}^1$ in the local $\mathbb{P}^1$ CY3, i.e. $Z = \mathrm{Tot}(\mathcal{O}(-1)^{\oplus 2} \to \mathbb{P}^1)$. 
\end{lem}
\begin{proof}
    The target $\mathbb{C}$ of the map $\mathbf{J}$ is the miniversal deformation space of $(X, J_X)$, and the family $(X, J_s)_{s \in \mathbb{C}}$ over $\mathbb{C}$ is exactly the local $\mathbb{P}^1$ CY3 by \cref{lem:nbhd-of-T*P1}. So a small neighborhood $U$ of $0 \in \mathbb{C}$ carries the family $(X, J_u)_{u \in U}$ which is a neighborhood of the zero fiber $(X, J_0) \cong (X, J_X)$ inside $Y$. By choosing the transversal slice $D$ to be the preimage $\mathbf{J}|_{\sDelta^{2j}}^{-1}(V)$ for an open disk $V \subset U$, we obtain the desired family.
\end{proof}

\subsubsection{$(\mathbb{Z}/p \times S^1)$-equivariant local $\mathbb{P}^1$ CY3}
\cref{lem:locus-J-standard} and \cref{lem:reduction-rank1} allows the computation of the structure constants of $Q\Sigma_b^{S^1}$ for the fiber class $b \in H^2_{S^1}(T^*\mathbb{P}^1)$ to be reduced to the local $\mathbb{P}^1$ case.

Fix $\mathcal{Y} = Y_0 \times (Y_1 \times \cdots \times Y_p) \times Y_\infty$ such that $Y_1 = \cdots = Y_p$ are fiber cycles, and $\mathcal{Y}^{eq}_{S^1}$ be a choice of $E\mathbb{Z}/p \times ES^1$-parametrized equivariant family of perturbations $\mathcal{Y}^{eq}_{w, \sw}$ such that each $Y_j(w, \sw) = \mathrm{proj}_j \circ \mathcal{Y}^{eq}_{w, \sw}$ is still a fiber cycle. (Compare the discussion with \cref{ssec:comparison-of-eq-moduli}.)

Recall that the $S^1$-action on $X = T^*\mathbb{P}^1$ extends to an action of the local $\mathbb{P}^1$ CY3 $Z = \mathrm{Tot}(\mathcal{O}(-1)^{\oplus 2} \to \mathbb{P}^1)$ (see discussion after \cref{lem:nbhd-of-T*P1}). Indeed, the constructions of the $\mathbb{Z}/p$-equivariant versions of the thickened moduli space $\mathcal{P}_A^{eq}$, the obstruction bundle $\mathrm{Obs}^{eq}$, and the incidence constraint bundle $\mathrm{IC}^{eq}$ from \cref{ssec:eq-section-moduli} all admit obvious generalizations into their $(\mathbb{Z}/p \times S^1)$-equivariant versions by means of the $(\mathbb{Z}/p \times S^1)$-equivariant Borel construction. Denote these by $\mathcal{P}^{eq}_{S^1, A}$, $\mathrm{Obs}^{eq}_{S^1}$, $\mathrm{IC}^{eq}_{S^1}$, respectively. 

\begin{prop}\label{prop:integration-formula-S1-QSt}
    Fix $b_0, b, b_\infty \in H^*(X)$ to be the Poincar\'e dual to $Y_0, Y = Y_1 = \cdots = Y_p,$ and $Y_\infty$. The corresponding structure constants for $Q\Sigma_b^{S^1} : QH^*_{\mathbb{Z}/p \times S^1}(X; \mathbb{F}_p) \to QH^*_{\mathbb{Z}/p \times S^1}(X;\mathbb{F}_p)$ from $b \in H^2_{S^1}(X;\mathbb{F}_p)$, given by
    \begin{equation}
       (Q\Sigma_b(b_0), b_\infty)_{S^1} = \sum_{i, j} \sum_A \# \ \mathrm{ev} (\mathcal{M}_A^{eq, i, 2j}) \pitchfork \mathcal{Y}^{eq, i, 2j} \  q^A \ (t,\theta)^i \ \teq^j
    \end{equation}
    for $Q\Sigma_b^{S^1} : QH^*_{\mathbb{Z}/p \times S^1}(X; \mathbb{F}_p) \to QH^*_{\mathbb{Z}/p \times S^1}(X;\mathbb{F}_p)$ from $b \in H^2_{S^1}(X;\mathbb{F}_p)$ are equal to
    \begin{equation}
        \sum_{i,j} \sum_{A} \teq \cdot  \left\langle \Delta^i \times \sDelta^{2j-2}, 
    \int_{\mathcal{P}_{S^1, A}^{eq}} c_{\mathrm{top}}^{eq}(\mathrm{Obs}_{S^1}^{eq}) \cup c_{\mathrm{top}}^{eq}(\mathrm{IC}_{S^1}^{eq}) \right\rangle q^A \ (t,\theta)^i \ \teq^{j-1}
    \end{equation}
    obtained from $(\mathbb{Z}/p \times S^1)$-equivariant Chern class integrals.
\end{prop}
\begin{proof}
    For $X = T^*\mathbb{P}^1$, choose generic equivariant perturbation data $\nu_X$ and consider the moduli space $\mathcal{M}_A^{eq, i, 2j}(X;\nu_X)$ with its projection to $\sDelta^{2j}$. By \cref{lem:locus-J-standard}, there is a distinguished hypersurface $V_{\mathbf{J}} \cap \sDelta^{2j}$ over which $J_\sw = J_X$ is the integrable almost complex structure on $X$. By the maximum principle, all solutions $(w, \sw, u: C \to X) \in \mathcal{M}_A^{eq, i, 2j}(X; \nu_X)$ for sufficiently small perturbation data are supported over a relatively compact subset of an open neighborhood $\mathcal{U}$ of $V_{\mathbf{J}} \cap \sDelta^{2j}$ in $\sDelta^{2j}$.

    Again let $Z = \mathrm{Tot}(\mathcal{O}(-1)^{\oplus 2} \to \mathbb{P}^1)$ denote the local $\mathbb{P}^1$ CY3. By \cref{lem:reduction-rank1}, by shrinking $\mathcal{U}$ if necessary, the total space of the family $(X, J_\sw)_{\sw \in \mathcal{U}} \to \mathcal{U}$ may be biholomorphically identified with the total space of $S^{2j-1} \times_{S^1} Z \to V_{\mathbf{J}}$ (restricted to the open subset $V_{\mathbf{J}} \cap \sDelta^{2j} \subseteq V_{\mathbf{J}}$), the finite dimensional approximation to the Borel construction of $Z$. Therefore the moduli space $\mathcal{M}_A^{eq, i, 2j}(X;\nu_X)$ (for sufficiently small $\nu_X$) can be identified as smooth oriented manifolds with $\mathcal{M}_A^{eq, i, 2j-2}(Z;\nu_Z)$, where $\nu_Z$ is the perturbation data for $Z$ induced from $\nu_X$, parametrized over $V_\mathbf{J} \cong BS^1$. The claim assumes regularity of these moduli spaces, achieved for generic perturbation data.

    We moreover claim that the incidence constraints $\mathcal{Y}_X^{eq, i, 2j}$ obtained from fiber classes (of $T^*\mathbb{P}^1 \to \mathbb{P}^1$) in $X$ can also be identified with the incidence constraints $\mathcal{Y}_Z^{eq, i, 2j-2}$ obtained from fiber classes (of $\mathcal{O}(-1)^{\oplus 2} \to \mathbb{P}^1$) in $Z$. Recall that we have chosen the $(\mathbb{Z}/p\times S^1)$-equivariant incidence cycle $\mathcal{Y}^{eq}_{S^1, X}$ to consist of  perturbations such that every $\mathcal{Y}^{eq}_{S^1, X}|_{w, \sw} =: \mathcal{Y}^{eq}_{w, \sw}$ is a product of fiber cycles. The fiber cycles in $X$ are represented by $S^1$-invariant embedded submanifolds of $X$. Therefore we can moreover assume that for each fixed $w' \in \Delta^i$, the $BS^1$-family of perturbations $\mathcal{Y}^{eq}_{w', \sw}$ (i.e. a family that varies over $\sw \in ES^1$ with equivariance condition) is given by the $S^1$-equivariant Borel construction $ES^1 \times_{S^1} \mathcal{Y}^{eq}_{w', 1}$ of $\mathcal{Y}^{eq}_{w', 1}$ for $1 = \sDelta^0 \in ES^1$:
    \begin{equation}
        \mathcal{Y}^{eq}_{w', \sw} := (\sw, \mathcal{Y}^{eq}_{w', 1}) \subseteq \{\sw\} \times (X \times X^p \times X).
    \end{equation}
    With that assumption, consider the embedding of fiber
    \begin{equation}
        Y_{X, j} (w', \sw) \subseteq X
    \end{equation}
    obtained by projecting the image of $\mathcal{Y}^{eq}_{w', \sw}$ to the factor of $X$ corresponding to $j = 0, 1, \dots, p, \infty$. Let $\sw$ vary over $D$, and consider the $1$-parameter family of fibers $\{Y_{X,j}(w', \sw) \}_{\sw \in D}$ of $X$ over the transversal slice $D$ to $\sw' \in \mathbb{CP}^{j-1} \cap \sDelta^{2j}$ from \cref{lem:reduction-rank1}. It remains to show that this family of fibers of $X$ can be identified with a single fiber $Y_{Z, j}(w', \sw')$ of $Z$, possibly after an isotopy. To see this, note that in the geometry of simultaneous resolution described in \cref{lem:nbhd-of-T*P1}, any fiber of $Z \to \mathbb{P}^1$ is a complex $2$-plane which intersects every $\mathcal{X}_s$ for $s \in \mathbb{C}$ in a single complex line. This line is transversal to the zero section $\mathbb{P}^1 \subseteq X$ along the diffeomorphism $X \cong \mathcal{X}_s$, and hence can be isotoped into any chosen fiber of $X \to \mathbb{P}^1$.
    
    Choose a diffeomorphism $f$ of $D$ with an open disk of $\mathbb{C}$, and one can choose such an isotopy between $Y_{Z,j}(w', \sw)  \cap \mathcal{X}_{f(\sw)}$ with $Y_{X,j}(w', \sw)$ for one $\sw \neq \sw' \in D$. A choice of such an isotopy for a single $\sw \in D$ induces a choice for all isotopy for $\sw \in D$, as the collection of diffeomorphisms $X \cong \mathcal{X}_{f(\sw)}$ is equivariant respect to the $S^1$-action on $D$. Using this family of isotopies, one can identify the family of fibers $\{Y_{X, j}(w', \sw)\}_{\sw \in D}$ of $X \to \mathbb{P}^1$ with a single fiber of $Z \to \mathbb{P}^1$.

    In sum, we are reduced to computing the following counts for $Z = \mathrm{Tot}(\mathcal{O}(-1)^{\oplus 2} \to \mathbb{P}^1)$:
    \begin{equation}
        \# \ \mathcal{M}_A^{eq, i, 2j}(X; \nu_X) \pitchfork \mathcal{Y}_X^{eq, i, 2j} = \# \ \mathcal{M}_A^{eq, i, 2j-2}(Z; \nu_Z
        ) \pitchfork \mathcal{Y}_Z^{eq, i, 2j-2}. 
    \end{equation}
    
    By the independence of counts on perturbation data, we now take $\nu_Z$ to be $(\mathbb{Z}/p \times S^1)$-equivariant analogues of the decoupling perturbation data from \cref{ssec:eq-decoupling-nu}. The same proof as in \cref{thm:qst-equals-Chern-integral} identifies this count with the integral
    \begin{equation}
        \left\langle \Delta^i \times \sDelta^{2j-2}, 
    \int_{\mathcal{P}_{S^1, A}^{eq}} c_{\mathrm{top}}^{eq}(\mathrm{Obs}_{S^1}^{eq}) \cup c_{\mathrm{top}}^{eq}(\mathrm{IC}_{S^1}^{eq}) \right\rangle \in \mathbb{F}_p.
    \end{equation}
\end{proof}

\subsection{Computation of the $S^1$-equivariant quantum Steenrod operations}
The results of the previous sections allow the computation for the $S^1$-equivariant quantum Steenrod operations to be reduced to the geometry of local $\mathbb{P}^1$ CY3 example. We only need to extend the results of \cref{ssec:computaiton-localP1} to incorporate $S^1$-weights to the computations in equivariant cohomology. Denote by $Y = \mathrm{Tot}(\mathcal{O}(-1)^{\oplus 2} \to \mathbb{P}^1)$.

\begin{lem}\label{lem:S^1-eq-coh-section}
    The $(\mathbb{Z}/p \times S^1)$-equivariant cohomology algebra of $\mathcal{P}_d$ is computed as
    \begin{equation}
        H^*_{\mathbb{Z}/p \times S^1}(\mathcal{P}_d) ; \mathbb{F}_p) = \mathbb{F}_p [\![H, t, \theta, \teq]\!]/H^2(H-t)^2\cdots(H-dt)^2,
    \end{equation}
    where the notations are as in \cref{lem:eq-coh-section}.
\end{lem}
\begin{proof}
    The $S^1$-action on the base $\mathbb{P}^1$ of $Y$ is trivial, so the $S^1$-action on $\mathcal{P}_d = \mathbb{P}H^0(C \times \mathbb{P}^1 , \mathcal{O}_{C \times \mathbb{P}^1} (d, 1))$ is also trivial. Therefore the $(\mathbb{Z}/p \times S^1)$-equivariant cohomology algebra is obtained from the $\mathbb{Z}/p$-equivariant cohomology algebra from \cref{lem:eq-coh-section} just by taking the (completed) tensor product with the ground ring of $S^1$-equivariant cohomology, $\mathbb{F}_p[\![\teq]\!]$.
\end{proof}

\begin{lem}\label{lem:S1-eq-euler-obs}
    The $(\mathbb{Z}/p \times S^1)$-equivariant top Chern class of the obstruction bundle is given by
     \begin{equation}\label{eqn:S1-eq-euler-obs}
        c_{\mathrm{top}}(\mathrm{Obs}^{eq}_{S^1}) = (H-t -\teq)^2 (H-2t - \teq)^2 \cdots (H-(d-1)t - \teq)^2 \in H^*_{\mathbb{Z}/p \times S^1}(\mathcal{P}_d ; \mathbb{F}_p).
    \end{equation}
\end{lem}
\begin{proof}
    From \cref{lem:obs-description}, we have a description
    \begin{equation}
        \mathrm{Obs}^\vee \cong \mathcal{O}_{\mathcal{P}_d}(1) \otimes \left( H^0(C \times \mathbb{P}^1, \mathcal{O}(d,1) \otimes \mathrm{proj}_{\mathbb{P}^1}^* \left( \mathcal{O}_{\mathbb{P}^1}(-1)^{\oplus 2} \right)^\vee \otimes K_{C \times \mathbb{P}^1} )\right).
    \end{equation}
    Here, $\mathcal{O}(-1)^{\oplus 2}$ fits into the Euler exact sequence of holomorphic vector bundles
    \begin{equation}
        0 \to \mathcal{O}(-2) \to \mathcal{O}(-1)^{\oplus 2} \to \mathcal{O} \to 0
    \end{equation}
    where $\mathcal{O}(-2) \to \mathbb{P}^1$ is the total space of $X = T^*\mathbb{P}^1$ and $\mathcal{O}$ parametrizes the deformations of $X = \mathcal{X}_0$ to nearby fibers in the simultaneous resolution $\mathcal{X} \to \mathbb{C}$ of \cref{lem:nbhd-of-T*P1}. Both $\mathcal{O}(-2)$ and $\mathcal{O}$ therefore have $S^1$-weight $1$. It follows that the middle term $\mathcal{O}(-1)^{\oplus 2}$ carries weight $\teq \oplus \teq$ under our conventions for the choice of $\teq$. Twisting the result from the $\mathbb{Z}/p$-equivariant case (\cref{lem:eq-euler-obs}) by this $S^1$-action gives the desired result.
\end{proof}

\begin{lem}\label{lem:S1-eq-euler-IC}
    The $(\mathbb{Z}/p \times S^1)$-equivariant top Chern class of the incidence constraint bundle is given by
    \begin{equation}
        c_{\mathrm{top}} (\mathrm{IC}^{eq}_{S^1}) = H(H-t)\cdots(H-(p-1)t) \cdot c_{0, \infty}(H,t) = (H^p - t^{p-1} H) c_{0, \infty}(H, t) \in H^*_{\mathbb{Z}/p \times S^1}(\mathcal{P}_d),
    \end{equation}
    where
    \begin{equation}\label{eqn:IC-0-infty}
        c_{0, \infty}(H,t) = 
        \begin{cases}
            1 & \mbox{ if } b_0 = 1, \ b_\infty = 1 \\
            (H-dt) & \mbox{ if } b_0 = b, \ b_\infty = 1 \\
            H & \mbox{ if } b_0 = 1, \ b_\infty = b \\
            H(H-dt) & \mbox{ if } b_0 = b, \ b_\infty = b
        \end{cases}.
    \end{equation}
\end{lem}
\begin{proof}
    This result is exactly as in \cref{lem:eq-euler-IC} with no modification, due to the fact that the $S^1$-action on the incidence constraint bundle is trivial (and hence no twisting).
\end{proof}

\begin{cor}\label{cor:S1-QSt-computation}
    The structure constants for $Q\Sigma_b^{S^1}$ are computed as follows. We write $q^d$ for $q^A$ such that $A  = d[\mathbb{P}^1] \in H_2(T^*\mathbb{P}^1;\mathbb{Z})$ for fixed $d > 0$. Write $d = \alpha p + \beta$ so that $\alpha \ge 0$ and  $0 \le \beta \le p-1$ are the quotient and remainder of $d$ mod $p$. For each $ 0 \le \ell \le p-1$, let
    \begin{equation}
        C_{d, \ell} := C_{d, \ell}(x, t, h) = (x^{p-1} - t^{p-1})^{1 - 2 \alpha} \ c_{0, \infty}(x + \ell t, t) \ \prod_{k=1}^{d-1} (x  + (\ell -k)t - \teq)^2 
    \end{equation}

    be a formal power series in $x$ with coefficients in $\mathbb{F}_p(\!(t, h )\!)$. Then we have the following formula
\begin{align}
    (Q\Sigma_b^{S^1}(b_0), b_\infty)_{S^1} \\
     = (\mathrm{St}(b) \cup b_0, b_\infty)_{S^1} +  \teq \  q^d \cdot  \sum_{d > 0} & \left[ \sum_{\ell=0}^\beta \mathrm{Coeff} \left( x^{2\alpha}:   C_{d, \ell} \ { \prod_{j=1}^\ell (x+jt)^{-2} \prod_{j=1}^{\beta-\ell} (x-jt)^{-2}  } \right) \right. \notag \\
         &\left. + \sum_{\ell = \beta + 1}^{p-1} \mathrm{Coeff} \left( x^{2\alpha - 2} : C_{d, \ell} \ {  \prod_{j=0}^\beta (x+(\ell-j)t)^{-2} } \right)  \right]  \notag.
\end{align}
Here $c_{0, \infty}(x,t)$ is the polynomial from \eqref{eqn:IC-0-infty}, depending on the choice of $b_0, b_\infty$.

In particular, the result can be computed for any $p > 2$ and up to any order in $(t,\theta)^i \teq^j$. Note that there are higher nonlinear terms in $\teq$, with $\teq^j$ with $j > 2$.
\end{cor}

\begin{proof}
    
    One can apply the integration formula for the fixed point localization theorem \cite[Proposition 5.3.18]{AP93} to compute the integral
    \begin{equation}
        (Q\Sigma_b^{S^1}(b_0), b_\infty)_{S^1} = \int_{\mathcal{P}_{S^1, d}^{eq}} c_{\mathrm{top}}^{eq}(\mathrm{Obs}_{S^1}^{eq}) \cup c_{\mathrm{top}}^{eq}(\mathrm{IC}_{S^1}^{eq}).
    \end{equation}
     The fixed locus of the $(\mathbb{Z}/p \times S^1)$-action on $\mathcal{P}_d$ has components given by the projectivizations on the isotypic decomposition (cf. \cref{lem:eq-coh-section}):
    \begin{align}
        &H^0(\mathcal{O}(d, 1)) \cong \bigoplus_{k=0}^d (\mathbb{C}_{\zeta^k})^{\oplus 2} = \bigoplus_{\ell =0}^{p-1} \left( \mathbb{C}_{\zeta^\ell} \right)^{\oplus 2 (\lfloor (d-\ell)/p \rfloor + 1) } , \\ 
        & (\mathcal{P}_d^{eq})^{\mathbb{Z}/p \times S^1} = \bigcup_{\ell=0}^{p-1} \mathbb{P}\left( \mathbb{C}_{\zeta^\ell} \right)^{\oplus 2 (\lfloor (d-\ell)/p \rfloor + 1)} = \bigcup_{\ell=0}^{p-1} \mathbb{P}^{2\lfloor (d-\ell)/p \rfloor +1 }.
    \end{align}
    
    For convenience, write $d =  \alpha p + \beta$ for $\alpha \ge 0$ and $0 \le \beta \le p-1$, so that $\lfloor d/p \rfloor = \alpha$. Denote the $\ell$th component of the fixed point locus by $\mathbb{P}_{d,\ell}$, which is a copy of a projective space of complex dimension
    \begin{equation}
        2 \lfloor (d-\ell)/p \rfloor + 1 = \begin{cases} 2\alpha +1 & \mbox{ if } \ell \le \beta \\ 2 \alpha -1 & \mbox{ if } \ell > \beta \end{cases}.
    \end{equation} Identify the equivariant cohomology of the $\ell$th component of the fixed locus with $H^*_{\mathbb{Z}/p \times S^1}(\mathbb{P}_{d, \ell}) = \mathbb{F}_p[\![ t, \theta, \teq]\!][x]/(x^{2\lfloor (d - \ell)/p\rfloor +2})$, so that the restriction map in cohomology for $\mathbb{P}_{d, \ell} \hookrightarrow \mathcal{P}_d$ sends $(H \mapsto x+\ell t)$.

    For localization, we first compute the Euler class of the normal bundle of $\mathbb{P}_{d, \ell} \hookrightarrow \mathcal{P}_d$. The normal bundle of $\mathbb{P}_\ell$ can be identified with the direct sum of the isotypic parts, twisted by the weight of the $\ell$th component:
    \begin{equation}\label{eqn:fixed-loc-euler-class}
        N_{\mathbb{P}_{d, \ell}} \cong \bigoplus_{0 \le k \le p-1, k \neq \ell} \left( \mathbb{C}_{\zeta^{k - \ell}} \right)^{\oplus 2 (\lfloor (d-k)/p \rfloor + 1)}
    \end{equation}
    Correspondingly one can compute
    \begin{align}
        c_{\mathrm{top}}( N_{\mathbb{P}_{d, \ell}} ) &= \prod_{0 \le k \le p-1, k \neq \ell} \left( x + (\ell - k) t \right)^{2 (\lfloor (d-k)/p \rfloor + 1)} \\
        &=  {\begin{cases} \prod_{j=1}^{\ell} (x + jt)^{2} \prod_{j=1}^{\beta - \ell} (x - jt)^{2} \prod_{j\neq \ell} (x+(\ell- j) t)^{2\alpha}  & \mbox{ if } \ell \le \beta \\ \prod_{j=0}^\beta (x + (\ell-j)t)^2 \prod_{j\neq \ell} (x+(\ell- j) t)^{2\alpha} & \mbox { if } \ell > \beta \end{cases}}.
    \end{align}

    One can simplify
    \begin{equation}\label{eqn:simplify-euler-class}
        \prod_{j \neq \ell} ( x+ (\ell -j)t) = \frac{\prod_{j=0}^{p-1} (x+ (\ell -j) t)}{ x + (\ell - \ell) t } =  \frac{x^p - t^{p-1}x}{x} = x^{p-1} - t^{p-1}.
    \end{equation}
    
    Now the integration formula \cref{prop:integration-formula-S1-QSt} gives, via fixed point localization,
    \begin{align}
        \mathrm{Coeff} \left( q^d : (Q\Sigma_b^{S^1}(b_0), b_\infty)_{S^1} \right) &= \int_{\mathcal{P}_{S^1, d}^{eq}} c_{\mathrm{top}}^{eq}(\mathrm{Obs}_{S^1}^{eq}) \cup c_{\mathrm{top}}^{eq}(\mathrm{IC}_{S^1}^{eq}) \\
        &= \sum_{\ell =0}^{p-1}  \int_{\mathbb{P}_{d, \ell}} \frac{c_{\mathrm{top}}^{eq}(\mathrm{Obs}_{S^1}^{eq}) \cup c_{\mathrm{top}}^{eq}(\mathrm{IC}_{S^1}^{eq}) }{ c_{\mathrm{top}} (N_{\mathbb{P}_{d, \ell}}) }. \label{eqn:integration-formula-loc-term}
    \end{align}
    The desired result now follows from \eqref{eqn:fixed-loc-euler-class}, \eqref{eqn:simplify-euler-class},  and \cref{lem:S1-eq-euler-obs}, \cref{lem:S1-eq-euler-IC}. We use that
    \begin{equation}
        c_{\mathrm{top}}^{eq}(\mathrm{IC}_{S^1}^{eq})|_{\mathbb{P}_{d, \ell}} = \left((x+\ell t)^p - t^{p-1}(x+\ell t) \right)c_{0,\infty}(x+\ell t, t) = \left( x^p - t^{p-1} x \right) c_{0,\infty}(x+\ell t, t),
    \end{equation}
    so that
    \begin{equation}
        c_{\mathrm{top}}^{eq}(\mathrm{Obs}_{S^1}^{eq}) \cup c_{\mathrm{top}}^{eq}(\mathrm{IC}_{S^1}^{eq}) |_{\mathbb{P}_{d, \ell}} = x \cdot (x^{p-1} - t^{p-1}) \  c_{0, \infty}(x+\ell t, t) \ \prod_{k=1}^{d-1} (x  + (\ell -k)t - \teq)^2 .
    \end{equation}

    The integrand on the right hand side of \eqref{eqn:integration-formula-loc-term} is therefore
    \begin{align}
        \int_{\mathbb{P}_{d, \ell}} \frac{c_{\mathrm{top}}^{eq}(\mathrm{Obs}_{S^1}^{eq}) \cup c_{\mathrm{top}}^{eq}(\mathrm{IC}_{S^1}^{eq}) }{ c_{\mathrm{top}} (N_{\mathbb{P}_{d, \ell}}) } &= 
            \frac{x \cdot (x^{p-1} - t^{p-1}) \ c_{0, \infty}(x+\ell t, t) \ \prod_{k=1}^{d-1} (x  + (\ell -k)t - \teq)^2 }{\prod_{j=1}^{\ell} (x + jt)^{2} \prod_{j=1}^{\beta - \ell} (x - jt)^{2}  (x^{p-1} - t^{p-1})^{2\alpha} }  \mbox{   if } \ell \le \beta 
            \\ &=\frac{x \cdot (x^{p-1} - t^{p-1}) \ c_{0, \infty}(x+\ell t, t) \ \prod_{k=1}^{d-1} (x  + (\ell -k)t - \teq)^2 }{\prod_{j=0}^\beta (x + (\ell-j)t)^2  (x^{p-1}-t^{p-1})^{2\alpha} }  \mbox{   if } \ell > \beta           
    \end{align}
    
    Finally, integration over $\mathbb{P}_{d, \ell}$ is given by extracting the coefficient of \begin{equation}x^{\dim \mathbb{P}_{d, \ell}} = \begin{cases} x^{2\alpha +1} & \mbox{ if } \ell \le \beta \\ x^{2 \alpha -1} & \mbox{ if } \ell > \beta \end{cases}.  \end{equation}
\end{proof}
This concludes the computation of the covariantly constant endomorphism for the equivariant quantum connection of $T^*\mathbb{P}^1$ in characteristic $p$.

\subsection{Uniqueness of the covariantly constant endomorphism}\label{ssec:uniqueness-flat-endo}
The computation above gives a full determination of an endomorphism $Q\Sigma_b^{S^1}$ that is flat for the equivariant quantum connection $\nabla_b^{S^1}$. It is given by the ($S^1$-equivariant) quantum Steenrod operation.

The endomorphism $Q\Sigma_b^{S^1}$ is in fact the unique non-trivial such endomorphism, in the sense described below. First observe that the flat endomorphisms have a $\mathbb{F}_p[\![q^p]\!]$-module structure, due to the linearity of the equation.

\begin{lem}\label{lem:module-structure-of-flat-endo}
    Let $\Sigma$ be a flat endomorphism for the mod $p$ quantum connection, that is it satisfies $\nabla \circ \Sigma - \Sigma \circ \nabla = 0$. For any $f = f(q^p) \in \mathbb{F}_p[\![q^p]\!]$, $f(q^p)\Sigma$ is also a flat endomorphism.
\end{lem}
\begin{proof}
    The result follows from $\partial_q f = 0$ and the Leibniz rule.
\end{proof}

\begin{prop}\label{prop:uniqueness-flat-endo}
    The $\mathbb{F}_p[\![q^p]\!](\!(\teq, t)\!)$-module of flat endomorphisms for the equivariant quantum connection (mod $p >2$)
    \begin{equation}
        \nabla_b^{S^1} = t q \partial_q - \begin{pmatrix} 0 & \frac{\teq^2 q}{1-q} \\ 1 & \frac{-2\teq q}{1-q} \end{pmatrix} 
    \end{equation}
    is free of rank $2$, generated by the identity endomorphism $\mathrm{id}$ and the quantum Steenrod operation $Q\Sigma_b^{S^1}$.
\end{prop}
\begin{proof}
    We first check the linear independence of the identity endomorphism $\mathrm{id}$ and the quantum Steenrod operation $Q\Sigma := Q\Sigma_b^{S^1}$. Suppose $c_1 \mathrm{id} + c_2 Q\Sigma = 0$ for $c_1, c_2 \in \mathbb{F}_p[\![q^p]\!] (\!(\teq, t)\!)$. By taking trace, we see that $c_1 = - (c_2/2) \mathrm{tr}(Q\Sigma)$, hence $c_2 ( Q\Sigma - \mathrm{tr}(Q\Sigma)/2) = 0$. Since the classical part $Q\Sigma|_{q=0}$ (the cup product with the Steenrod power $\mathrm{St}(b)$) has a nonzero off-diagonal term, $Q\Sigma$ is non-identity. This implies $c_2 = 0$ and hence $c_1 = 0$, as desired.

    Now suppose
    \begin{equation}
        \Sigma = \begin{pmatrix} a & b \\ c & d \end{pmatrix}, \quad a, b, c, d \in \mathbb{F}_p[\![q]\!](\!(\teq, t)\!)
    \end{equation}
    is a flat endomorphism for the connection $\nabla^{S^1}$. By subtracting from $\Sigma$ a suitable multiple of identity, we may assume that $d = 0$. The covariant constancy relation $[\nabla_b^{S^1}, \Sigma] = 0$ is equivalent to the following system of differential equations:
    \begin{align}
        &t q \partial_q a + \frac{\teq^2 q}{1-q} \cdot c - b = 0, \label{eqn:flat1}\\
        &t q \partial_q b - \frac{\teq^2 q}{1-q} \cdot a + \frac{2\teq q}{1-q} \cdot b = 0 , \label{eqn:flat2}\\
        &t q \partial_q c + a - \frac{2\teq q}{1-q} \cdot c = 0. \label{eqn:flat3}
    \end{align}
    Write $a, b, c$ as power series in $q$, so that $a = \sum a_k \ q^k$, $b = \sum b_k \ q^k$, $c = \sum c_k \ q^k$. From \eqref{eqn:flat1} and \eqref{eqn:flat3}, one can see that $b$ and $a$ resp. are divisible by $q$ and hence $a_0 = b_0 = 0$. There is a freedom of choice for $c_0$, which by the equations \eqref{eqn:flat1}, \eqref{eqn:flat2}, \eqref{eqn:flat3}, determines $a_k, b_k, c_k$ for all $0 \le k \le p-1$. More generally, it is easy to see that in fact 
    the part of $c$ with $p$-divisible degrees,
    \begin{equation}
        c' = \sum c_{pk} \ q^{pk},
    \end{equation} determines all coefficients $a_k, b_k, c_k$ for all $k$, hence determines the endomorphism $\Sigma$.
    
    With the degrees of freedom in the operation $\Sigma$ understood, consider the lower-left entry $c_\Sigma$ of $Q\Sigma_b^{S^1}$. Since $Q\Sigma_b^{S^1}|_{q=0} (x) = \mathrm{St}(b)\  \cup_{S^1} x = -t^{p-1} b \ \cup_{S^1} x $, we see that $c_\Sigma|_{q=0} = -t^{p-1}$. Therefore $c_\Sigma$ is an invertible power series in $\mathbb{F}_p[\![ q^p ]\!](\!(\teq, t)\!)$. By multiplying $Q\Sigma_b^{S^1}$ by $c' c_\Sigma^{-1}$ for any prescribed $c'$, we can obtain any flat endomorphism for $\nabla^{S^1}$.
\end{proof}

    

\section{Relationship to arithmetic flat sections}\label{sec:flat-sections}
In this section, we study the relationship between the $S^1$-equivariant quantum Steenrod operations $Q\Sigma_b^{S^1}$ computed for $X = T^*\mathbb{P}^1$ and the mod $p$ solutions (up to gauge transformations) of the equivariant quantum differential equation for $X = T^*\mathbb{P}^1$ obtained in \cite{Var-ar}. Following \cite{Var21}, we call such distinguished mod $p$ solutions the \emph{arithmetic flat sections} of the equivariant quantum connection. 

\subsection{Arithmetic flat sections}
In \cite{Var-ar}, polynomial solutions to the equivariant quantum differential equation in mod $p$ coefficients (up to gauge transformations) are obtained.

In \cref{sec:appendix-gauge} and \cref{sec:appendix-soln}, we review the derivation of the relevant gauge transform and Varchenko's construction of the arithmetic flat section. The result is summarized as follows.

\begin{lem}[\cref{ssec:appendix-gauge}]\label{lem:gauge-transform}
    By changing the geometric basis $1, b \in H^*_{S^1}(X) \otimes \mathrm{Frac}(H^*_{S^1}(\mathrm{pt})) \cong \mathbb{F}_p(\!(\teq)\!)[b]/(b^2)$ to the preferred basis $b-\teq, b \in H^*_{S^1}(X) \otimes \mathbb{F}_p (\!(\teq)\!)$ (see \eqref{eqn:appendix-stable-envelope}), the quantum differential equation $\nabla_b^{S^1} I = 0$ is written as
    \begin{equation}\label{eqn:QDE-stable-basis}
    q \partial_q I = - \frac{\teq}{t} \left( \begin{pmatrix} 0 & 0 \\ 1& 0  \end{pmatrix} + \frac{q}{1-q} \begin{pmatrix}
    1 & 1 \\ 1 & 1    
    \end{pmatrix} \right) I,
\end{equation}
for vector-valued functions $I$. 
\end{lem}

Note that the change of basis simplifies the connection matrix, in that the dependence on the equivariant parameters $\teq \in H^2(BS^1), t \in H^2(B\mathbb{Z}/p)$ becomes a single factor in $\teq / t$. Below, we specialize to a particular value of $\teq / t = \mu$, which is necessary for the construct the arithmetic flat sections (see \cref{sec:appendix-soln}). This specialization also simplifies the $S^1$-equivariant quantum Steenrod operation, see \cref{rem:QSt-simplifies}.

\begin{prop}[\cref{lem:appendix-arithmetic-flat-section}]\label{prop:mod-p-solution}
Fix $\mu \in \mathbb{F}_p$, and let $0 \le m \le p-1$ be the unique integer such that $\mu \equiv m \mod p$. For each such $\mu$, the vector
\begin{equation}\label{eqn:arithmetic-soln}
     I_\mu(q) = (-1)^{p-m} (1-q)^{2m} \left( \sum_{d=0}^{p-m} \binom{p-m-1}{p-m-d} \binom{p-m}{d} q^d , \  \sum_{d=0}^{p-m} \binom{p-m}{p-m-d} \binom{p-m-1}{d} q^d \right)
    \end{equation}
is a solution of the quantum differential equation \eqref{eqn:QDE-stable-basis} after the specialization $\teq/ t = \mu$, i.e. $\nabla_b^{S^1} I_\mu = 0$.
\end{prop}

In sum, the vector $I_\mu$ for a choice of $\mu \in \mathbb{F}_p$ is a solution to the quantum differential equation \eqref{eqn:QDE-stable-basis} mod $p$. Note that since the quantum differential equation is linear, any multiple $f(q) I_\mu(q)$ for $f(q) \in \mathbb{F}_p[\![q^p]\!]$ is again a solution.

\begin{defn}
    The $\mathbb{F}_p[\![q^p]\!]$-module generated by $I_\mu(q)$ of \eqref{eqn:arithmetic-soln} is the module of \emph{arithmetic flat sections}.
\end{defn}

Arithmetic flat sections of \cite{Var-ar} form a submodule of the $\mathbb{F}_p[\![ q^p]\!]$-module of all mod $p$ solutions. In our case of $X = T^*\mathbb{P}^1$, the module of arithmetic flat sections is (tautologically) a free $\mathbb{F}_p[\![q^p]\!]$-module of rank $1$.

\subsection{Action on arithmetic flat sections}
By \cref{prop:S1-cov-constancy}, $Q\Sigma_b^{S^1}$ is a covariantly constant endomorphism for $\nabla_b^{S^1}$ mod $p$. In particular, if $I$ is a mod $p$ solution of $\nabla_b^{S^1} I = 0$, then $Q\Sigma_b^{S^1} I$ must also be a solution. We study this action on the distinguished solution $I_\mu(q)$, and find that the $S^1$-equivariant quantum Steenrod operation annihilates it.

\begin{prop}\label{prop:QSt-annihilates-soln}
    Let 
    \begin{equation}
     I_\mu(q) = (-1)^{p-m} (1-q)^{2m} \left( \sum_{d=0}^{p-m} \binom{p-m-1}{p-m-d} \binom{p-m}{d} q^d , \  \sum_{d=0}^{p-m} \binom{p-m}{p-m-d} \binom{p-m-1}{d} q^d \right)
    \end{equation}
    from \eqref{eqn:arithmetic-soln} be the distinguished arithmetic flat section obtained by \cite{Var-ar} for a choice of the parameter $\mu \equiv m \in \mathbb{F}_p$. After specializing $\teq / t = \mu$, we have
    \begin{equation}
        Q\Sigma_b^{S^1}|_{\teq/t = \mu} \left(  I_\mu (q) \right) = 0.
    \end{equation}
\end{prop}

\begin{rem}\label{rem:QSt-simplifies}
    To make the connection with the arithmetic flat sections, which are only obtained after choosing a parameter $\mu$, we have fixed the specialization $\teq / t = \mu$. Such a specialization has an effect of simplifying the localization computation from \cref{cor:S1-QSt-computation}. For example, for $d < p$, the following matrix entry (structure constant in the stable envelope basis, see \eqref{eqn:appendix-stable-envelope}) has the form
    \begin{align}
    &\mathrm{Coeff}\left( q^d :  -(Q\Sigma_b^{S^1}(b-\teq), (b-\teq))_{S^1}) \right) \\
    = & -\teq \left. \sum_{\ell \le d} \frac{\mathrm{Coeff}\left( x^0 = 1 : (x^{p-1} - t^{p-1}) \ (x+\ell t - \teq)(x+\ell t - dt - \teq) \ \prod_{k=1}^{d-1} (x  + (\ell -k)t - \teq)^2 \right)}{ \prod_{j=1}^\ell (x+jt)^2 \prod_{j=1}^{\beta-\ell}(x-jt)^2}   \right|_{\teq/t = \mu} \notag\\
    =&  \ \mu \cdot t^p \ \sum_{\ell \le d} \frac{(\mu - \ell)(\mu - \ell + d)}{(d - \ell)!^2 \ell!^2 }\prod_{k=1}^{d-1} ( \mu - \ell + k )^2 \notag\\
    =& \ \mu \cdot t^p \ \sum_{\ell \le d} \binom{\mu - \ell + d -1}{d} \binom{\mu - \ell + d}{d} {\binom{d}{\ell}}^2 \notag
    \end{align}
    which are themselves values of (generalized) hypergeometric functions. Similar expressions exist for the other structure constants. 
\end{rem}

To give the proof of \cref{prop:QSt-annihilates-soln}, we need the following lemma describing how $Q\Sigma_b^{S^1}$ simplifies under the specialization $\teq /t = \mu$.
\begin{lem}\label{lem:QSt-structure-const-depends-on-remainder}
    After specializing $\teq/t = \mu$, the degree $d$ part of $Q\Sigma_b^{S^1}|_{\teq /t = \mu}$,
    \begin{equation}
        \mathrm{Coeff} \left( q^d : Q\Sigma_b^{S^1}|_{\teq / t = \mu} \right) \in \mathrm{Mat}_{2 \times 2}(\mathbb{F}_p[\![t]\!]) 
    \end{equation}
    only depends on the residue of $d$ mod $p$.
\end{lem}
\begin{proof}
    By the integration formula (\cref{prop:integration-formula-S1-QSt}), each structure constant has the form
    \begin{equation}
        \mathrm{Coeff} \left( q^d : (Q\Sigma_b^{S^1}(b_0), b_\infty)_{S^1} |_{\teq/ t = \mu} \right) = \left. \int_{\mathcal{P}_{S^1, d}^{eq}} c_{\mathrm{top}}^{eq}(\mathrm{Obs}_{S^1}^{eq}) \cup c_{\mathrm{top}}^{eq}(\mathrm{IC}_{S^1}^{eq}) \right|_{\teq/t = \mu}.
    \end{equation}
    It suffices to check that this integral is equal for $d$ and $d'$ if $d \equiv d'$ mod $p$. Equivalently by induction, it is enough to check that this integral is equal for $d$ and $d + p$. Recall that the (localized) equivariant cohomology algebra of $\mathcal{P}_{S^1, d}^{eq}$ is given by (\cref{lem:S^1-eq-coh-section})
    \begin{equation}
        \mathrm{Frac} (H^*_{\mathbb{Z}/p \times S^1}(\mathrm{pt})) \otimes H^*_{\mathbb{Z}/p \times S^1}(\mathcal{P}_{S^1, d}^{eq}) = \mathbb{F}_p(\!(\teq, \theta, t)\!) [\![H]\!] / H^2(H-t)^2 \cdots (H-dt)^2,
    \end{equation}
    and the obstruction bundle has $(\mathbb{Z}/p \times S^1)$-equivariant Euler class given by \cref{lem:S1-eq-euler-obs})
    \begin{equation}
        c_{\mathrm{top}}^{eq} (\mathrm{Obs}_{S^1, d}^{eq}) = (H- t - \teq)^2 (H-2t-\teq)^2 \cdots (H-(d-1)t - \teq)^2.
    \end{equation}
    The relation in the equivariant cohomology algebra of $\mathcal{P}_{S^1, d}^{eq}$ and $\mathcal{P}_{S^1, d+p}^{eq}$ are $\prod_{k=0}^d (H-kt)^2$ and $\prod_{k=0}^{d+p}(H-kt)^2$ respectively, which differ by a factor of
    \begin{equation}
        \prod_{k=d+1}^{d+p} (H-kt)^2 = H^2(H-t)^2 \cdots (H-(p-1)t)^2 = (H^p - t^{p-1} H)^2.
    \end{equation}
    Similarly, the Euler classes of the obstruction bundles for degree $d$ and $d+p$ are $\prod_{k=1}^{d-1} (H-kt-\teq)^2$ and $\prod_{k=1}^{d+p-1} (H-kt-\teq)^2$, which differ by the same factor of
    \begin{equation}
        \left. \prod_{k=d}^{d+p-1} (H - kt - \teq)^2 \right|_{\teq/t = \mu} = \prod_{k=d}^{d+p-1} (H-kt-\mu t)^2 = (H^p - t^{p-1} H)^2
    \end{equation}
    after the specialization $\teq/t = \mu$.

    The Euler class of the incidence constraint bundle (\cref{lem:S1-eq-euler-obs}) is manifestly only dependent on the remainder of $d$ mod $p$.

    Note that the integral for the structure constant in degree $d$
    \begin{equation}
        \left. \int_{\mathcal{P}_{S^1, d}^{eq}} c_{\mathrm{top}}^{eq}(\mathrm{Obs}_{S^1}^{eq}) \cup c_{\mathrm{top}}^{eq}(\mathrm{IC}_{S^1}^{eq}) \right|_{\teq/t = \mu}
    \end{equation}
    is the coefficient of $H^{2d+1}$ in the remainder $R(t, \theta, \teq)$ of $c_{\mathrm{top}}^{eq}(\mathrm{Obs}_{S^1}^{eq}) \cup c_{\mathrm{top}}^{eq}(\mathrm{IC}_{S^1}^{eq})$ divided by the relation $H^2(H-t)^2 \cdots (H-dt)^2$:
    \begin{equation}
    c_{\mathrm{top}}^{eq}(\mathrm{Obs}_{S^1}^{eq}) \cup c_{\mathrm{top}}^{eq}(\mathrm{IC}_{S^1}^{eq}) = \left[H^2(H-t)^2 \cdots (H-dt)^2 \right] Q(t, \theta, \teq) + R(t, \theta, \teq).
    \end{equation}
    (The division is taken in the Euclidean domain $\mathbb{F}_p(\!(t, \theta, \teq)\!) [\![H]\!]$.) The integral for the structure constant in degree $d+p$ is therefore the coefficient of $H^{2d+2p+1}$ in
    \begin{equation}
        (H^p - t^{p-1} H)^2 R(t, \theta, \teq),
    \end{equation}
    which is equal to the coefficient of $H^{2d+1}$ in $R(t, \theta, \teq)$, as desired.
\end{proof}

\begin{proof}[Proof of \cref{prop:QSt-annihilates-soln}]
    For convenience, we introduce the notation
    \begin{equation}
        \Sigma := Q\Sigma_b^{S^1}|_{\teq/ t = \mu}.
    \end{equation}

    Our goal is to show that $\Sigma I_\mu = 0$. By covariant constancy, we know that $\nabla_b^{S^1} \Sigma I_\mu = 0$.  As in \cref{prop:uniqueness-flat-endo}, one can show that any flat section of $\nabla^{S^1}$ is a $\mathbb{F}_p[\![q^p]\!](\!(\teq, t)\!)$ -multiple of $I_\mu$, hence there exists some power series $f(q^p)$ such that $\Sigma I_\mu = f(q^p) I_\mu$. It suffices to show that $f(q^p) = 0$.

    By the $\mathbb{F}_p[q]$-linearity of $\Sigma$, for $I'_\mu$ such that $I_\mu = (-1)^{p-m}(1-q)^{2m} I'_\mu$, i.e.
    \begin{equation}
        I'_\mu (q) = \left( \sum_{d=0}^{p-m} \binom{p-m-1}{d-1} \binom{p-m}{d} q^d , \  \sum_{d=0}^{p-m} \binom{p-m}{d} \binom{p-m-1}{d} q^d \right),
    \end{equation}
    we have $\Sigma I_\mu = (-1)^{p-m} (1-q)^{2m} I_\mu '$. Hence it follows that $\Sigma I_\mu ' = f(q^p) I_\mu '$.
    
    Expand $\Sigma = \Sigma_\mu$ in power series, so that
    \begin{equation}
        \Sigma = \sum \Sigma_{d} q^d =  \Sigma_{0} + \Sigma_{1}  q^1 + \Sigma_{2}  q^2 + \cdots.
    \end{equation}
    By \cref{lem:QSt-structure-const-depends-on-remainder}, we have $\Sigma_{d} = \Sigma_{d+p}$ for all $d > 0$. Therefore we have
    \begin{equation}
        \Sigma = \frac{1}{1-q^p} \left( \Sigma_{0} + \Sigma_{1} q + \cdots + \Sigma_{p-1} q^{p-1} \right) + \frac{q^p}{1-q^p}(\Sigma_0 - \Sigma_p).
    \end{equation}

    By letting $g(q^p) = (1-q^p) f(q^p)$, we therefore have
    \begin{equation}
    \left( \Sigma_{0} + \Sigma_{1} q + \cdots + \Sigma_{p-1} q^{p-1} \right) I_{\mu}' + q^p(\Sigma_0 - \Sigma_p) I_{\mu}' = g(q^p) I_{\mu}'.
    \end{equation}

    Since $I_\mu$ has entries that are polynomials in $q$ bounded in degree $\le (p-1)$, $g(q^p)$ must be of the form $g(q^p) = a+ bq^p$ by examining degrees. It suffices to show that $a = b = 0$. 
    
    Using the localization formula, one obtains the following simplified expression for the matrix $\Sigma = \sum \Sigma_d q^d$ in the stable envelope basis (see \cref{rem:QSt-simplifies}) for $0 \le d \le p-1$:
    \begin{equation}\label{eqn:ann-qst}
        \Sigma_d = \mu \ t^p \   \sum_{\ell \le d} \begin{pmatrix} -  a_{\ell-1, d-1} \ b_{\ell,d }  & -  a_{\ell-1,d-1} \ b_{\ell, d-1} \\  a_{\ell, d} \ b_{\ell, d} & a_{\ell, d} \  b_{\ell, d-1} \end{pmatrix}
    \end{equation}
    for
    \begin{align}
        a_{\ell, d} = \binom{\mu + d-\ell -1}{d} \binom{d}{\ell}, \quad b_{\ell, d} = \binom{\mu + d-\ell }{d} \binom{d}{\ell}.
    \end{align}
    Moreover, one can also compute that
    \begin{equation}\label{eqn:ann-qst-0p}
        \Sigma_0 = \mu t^p \begin{pmatrix} 0 & 0 \\ 1 & 0 \end{pmatrix}, \quad \Sigma_p = \mu t^p \begin{pmatrix} 0 & 1 \\ 1 & 0 \end{pmatrix}.
    \end{equation}

    Finally recall
    \begin{equation}\label{eqn:ann-section}
        I_\mu (q) = \left( \sum_{d=0}^{p-m} \binom{p-m-1}{d-1} \binom{p-m}{d} q^d ,  \ \sum_{d=0}^{p-m} \binom{p-m}{d} \binom{p-m-1}{d} q^d \right).
    \end{equation}
    From \eqref{eqn:ann-qst}, \eqref{eqn:ann-qst-0p}, and \eqref{eqn:ann-section}, it is now straightforward to verify that $a=b=0$ by examining the $q^0$-term and $q^{2p-m}$-term (the highest order term) on both sides of
        \begin{equation}
    \left( \Sigma_{0} + \Sigma_{1} q + \cdots + \Sigma_{p-1} q^{p-1} \right) I_{\mu}' + q^p(\Sigma_0 - \Sigma_p) I_{\mu}' = (a+bq^p) I_{\mu}'.
    \end{equation}
    For $q^0$-term, both sides are given by
    \begin{equation}
        \mu t^p \begin{pmatrix}
            0 & 0 \\ 1 & 0 
        \end{pmatrix} \begin{pmatrix} \binom{p-m-1}{-1} \binom{p-m-1}{0} \\ \binom{p-m}{0} \binom{p-m}{0} \end{pmatrix} = \binom{0}{0} = a \binom{0}{1},
    \end{equation}
    and for $q^{2p-m}$-term both sides are given by
    \begin{equation}
        q^p\mu t^p \begin{pmatrix}
            0 & -1 \\ 0 & 0 \end{pmatrix} \begin{pmatrix}
                \binom{p-m-1}{p-m-1}\binom{p-m}{p-m} q^{p-m} \\ \binom{p-m}{p-m} \binom{p-m-1}{p-m} q^{p-m}
            \end{pmatrix} = \binom{0}{0} = bq^p\binom{q^{p-m}}{0}.
    \end{equation}


    
\end{proof}

\appendix
\section{Gauge transformation of the quantum connection}\label{sec:appendix-gauge}
In this appendix, we carry out the necessary gauge transformation to identify the equivariant quantum connection of $T^*\mathbb{P}^1$ with the dynamical differential equation treated in \cite{Var-ar}. It is used to study the relationship between the ($S^1$-equivariant) quantum Steenrod operations and the arithmetic flat sections in \cref{sec:flat-sections}. 

\subsection{Dynamical differential equation and quantum differential equation}
Fix $p > 2$. The goal of this appendix is to identify the mod $p$ quantum connection
\begin{equation}\label{eqn:appendix-T*P1-quantum-connection}
    \nabla_b^{S^1} =  t q\partial_q - \begin{pmatrix} 0 & \frac{\teq^2 q}{1-q} \\ 1 & \frac{-2\teq q}{1-q} \end{pmatrix}
\end{equation}
for $X = T^*\mathbb{P}^1$ with the connection of the dynamical differential equation studied in \cite{Var-ar}. The quantum differential equation can be rewritten as
\begin{equation}\label{eqn:appendix-QDE}
    q \partial_q I = t^{-1} \begin{pmatrix}
        0 & \frac{\teq^2 q}{1-q} \\ 1 & \frac{-2\teq q}{1-q}
    \end{pmatrix} I
\end{equation}
for some vector-valued solution $I(q) \in H^*_{S^1} (X; \mathbb{F}_p)[\![q]\!]$.

We reproduce the dynamical differential equation \cite[Eq. (1.2)]{Var-ar} below. It is a rank $2$ connection on a two-dimensional base depending on a choice of a rational parameter $\mu \in \mathbb{Q}$. To obtain reduction mod $p$, we assume that $\mu = r/q$ for $\mathrm{gcd}(r, q) = 1$ satisfies $p \nmid q$. The mod $p$ dynamical differential equation is then written as
\begin{align}\label{eqn:appendix-T*P1-Varchenko-connection1}
    z_1 \partial_{z_1} I  &= \mu \left( \begin{pmatrix}
        0 & - 1 \\ 0 & 0 \end{pmatrix} + \frac{  z_1}{z_1 - z_2} \begin{pmatrix}
            -1 & 1 \\ 1 & -1
        \end{pmatrix} \right) I , \\
        \label{eqn:appendix-T*P1-Varchenko-connection2}
    z_2 \partial_{z_2} I &= \mu \left( \begin{pmatrix}
        0 & 0 \\ -1  & 0 \end{pmatrix} + \frac{  z_2}{z_2 - z_1} \begin{pmatrix}
            -1 & 1 \\ 1 & -1
        \end{pmatrix} \right) I .
\end{align}
The parameters $(z_1, z_2)$ are related to the Novikov parameter $q$ by the relation $z_2 / z_1 = q$. (See \cite[Theorem 4.3, Section 7.4]{TV14}, where $z_1, z_2$ are denoted by $\widetilde{q}_1^{-1}, \widetilde{q}_2^{-1}$. The ratio $q = \widetilde{q}_1/\widetilde{q}_2$ is the Novikov parameter for the $\mathrm{SL}_2$-flag variety). By pushing the equation forward using the relation $z_2 / z_1 = q$, so that we let $2 q \partial_q = z_2 \partial_{z_2} - z_1 \partial_{z_1}$, the equation is reduced to one dimension:
\begin{align}\label{eqn:appendix-DDE-beforegauge}
    q \partial_q I = \frac{\mu}{2} \left( \begin{pmatrix}
        0 & 1 \\ -1  & 0 \end{pmatrix} + \frac{ q+1}{q-1} \begin{pmatrix}
            -1 & 1 \\ 1 & -1
        \end{pmatrix} \right) I .
\end{align}

\subsection{The gauge transformation}\label{ssec:appendix-gauge}
We now provide the necessary gauge transformation that identifies the equations \eqref{eqn:appendix-QDE} and \eqref{eqn:appendix-DDE-beforegauge}. In general, such gauge transformations are given by the theory of cohomological stable envelopes as in \cite{MO19}. Below we treat the simplest case of $X = T^*\mathbb{P}^1$.

In the case of $X = T^*\mathbb{P}^1$, the ``stable envelope map'' that identifies the fiber of the dynamical differential equation $\mathbb{F}_p(\!(\teq)\!)^{\oplus 2}$ with the fiber of the quantum differential equation, $H^*_{S^1}(X) \otimes \mathrm{Frac}(H^*_{S^1}(\mathrm{pt})) \cong \mathbb{F}_p(\!(\teq)\!)[b]/(b^2)$, is given by \cite[Section 5]{TV14}:
\begin{equation}\label{eqn:appendix-stable-envelope}
    W : \mathbb{F}_p(\!(\teq)\!)^{\oplus 2} \to \mathbb{F}_p(\!(\teq)\!) \cdot 1 \oplus \mathbb{F}_p(\!(\teq)\!) \cdot b, \quad W(1,0) = b - \teq, \ W(0,1) = b.
\end{equation}
By changing basis to the stable envelope basis given by $W(1,0) = b- \teq$ and $W(0,1) = b$, the equation \eqref{eqn:appendix-QDE} transforms to
\begin{equation}\label{eqn:appendix-QDE-stable}
    q \partial_q I = - \frac{\teq}{t} \left( \begin{pmatrix} 0 & 0 \\ 1& 0  \end{pmatrix} + \frac{q}{1-q} \begin{pmatrix}
    1 & 1 \\ 1 & 1    
    \end{pmatrix} \right) I .
\end{equation}

The theory of stable envelopes moreover prescribes the gauge transformation for \eqref{eqn:appendix-DDE-beforegauge}.

\begin{lem}[{\cite[Corollary 7.3, 7.4]{TV14}}]\label{lem:appendix-gauge-transform}
    Let $I$ be a solution (flat section) of the dynamical differential equation \eqref{eqn:appendix-T*P1-Varchenko-connection1}, \eqref{eqn:appendix-T*P1-Varchenko-connection2}. Then
    \begin{equation}
        \widetilde{I} = z_1^{-\mu} (z_1 - z_2)^{2\mu} I 
    \end{equation}
    is a solution of the quantum differential equation, after identifying $\mu = \teq / t$ and $2 q \partial_q = z_2 \partial_{z_2} - z_1 \partial_{z_1}$. 
\end{lem}
\begin{proof}
    The ansatz $\widetilde{I}$ is provided by the theory of stable envelopes, see \cite[Corollary 7.3, 7.4]{TV14}. Once the ansatz is given, one can easily verify that if $I$ solves the equations \eqref{eqn:appendix-T*P1-Varchenko-connection1} and \eqref{eqn:appendix-T*P1-Varchenko-connection2}, then
    \begin{align}
        q \partial_q \widetilde{I} &= \frac{1}{2}\left(  z_2 \partial_{z_2} - z_1 \partial_{z_1} \right) [ z_1^{-\mu} (z_1 - z_2)^{2\mu} I ] \\ 
        &= \frac{1}{2} \left[ \left( - \frac{2\mu z_2}{z_1 - z_2} + \mu - \frac{2\mu z_1}{z_1 - z_2} \right)\widetilde{I} + z_1^{-\mu} (z_1 - z_2)^{2\mu} (z_2 \partial_{z_2} - z_1 \partial_{z_1}) I \right]  \\
        &= -\frac{\mu}{2} \left[ \left( -1 -2 \cdot  \frac{q+1}{q-1} \right) \begin{pmatrix} 1 & 0 \\ 0 & 1 \end{pmatrix} - \begin{pmatrix} 0 & 1 \\ -1 & 0 \end{pmatrix} - \frac{q+1}{q-1} \begin{pmatrix} -1 & 1 \\ 1 & -1 \end{pmatrix}  \right] \widetilde{I} \\
        &= - \mu \left( \begin{pmatrix} 0 & 0 \\ 1& 0  \end{pmatrix} + \frac{q}{1-q} \begin{pmatrix}
    1 & 1 \\ 1 & 1    
    \end{pmatrix} \right) \widetilde{I}
    \end{align}
    where we replace each occurrence of $z_2 / z_1 $ with $q$. Note that this agrees with \eqref{eqn:appendix-QDE-stable} after the identification $\mu = \teq / t$.
\end{proof}

\section{Derivation of arithmetic flat sections}\label{sec:appendix-soln}
In this appendix, we review the construction of polynomial solutions to the equivariant quantum differential equation in mod $p$ coefficients (up to gauge transformations) from \cite{Var-ar}. The method to obtain such solutions is an application of an idea credited to Manin \cite{Man61}, and uses the description of the solutions to the same equation in characteristic $0$ in terms of hypergeometric integrals \cite{TV14}.

In our case of $X = T^*\mathbb{P}^1$, the relevant \emph{master function} (the integrand of the hypergeometric integral) has the form (see \cite[Section 1]{Var-ar})
\begin{equation}\label{eqn:master-function}
    \Phi (s; z_1, z_2; \mu) = s^{\mu} (s-z_1)^{-\mu} (s-z_2)^{-\mu} \in \mathbb{Q}[z_1, z_2][s], 
\end{equation}
where $\mu \in \mathbb{Q}$ is some rational parameter. For $\mu$ to admit a well-defined representative in $\mathbb{F}_p$, we assume the following:
\begin{assm}\label{assm:mu-reduce-mod-p}
    Let $\mu = r/q \in \mathbb{Q}$ for $r, q$ coprime. Assume that $p \nmid q$, so that there exists a unique integer $0 \le m \le p-1$ such that $\mu \equiv m$ (mod $p$).
\end{assm}
In the main body of the text \cite{Var-ar}, the case $\mu = 1/2$ is discussed. 

Under \cref{assm:mu-reduce-mod-p}, fix some $0 \le m \le p-1$ such that $\mu \equiv m$ (mod $p$). Let
\begin{equation}\label{eqn:master-function-mod-p}
    \Phi_p (s; z_1, z_2 ; \mu) = s^{m} (s-z_1)^{p-m} (s-z_2)^{p-m} \in \mathbb{F}_p[z_1, z_2][s]
\end{equation}
be a \emph{polynomial} in $s$ with coefficients in $\mathbb{F}_p[z_1, z_2]$ obtained by replacing the occurrence of $\mu$ with $m$, and $-\mu$ with $p-m$. For $i = 1, 2$, denote
\begin{equation}\label{eqn:solution-entries}
    \Psi_{p,i} (z_1, z_2) = \mathrm{Coeff} \left( s^{p-1} : \frac{\Phi_p (s, z_1, z_2)}{s - z_i} \right) \in \mathbb{F}_p [z_1, z_2]
\end{equation}
for the coefficient of $s^{p-1}$ in the polynomials $\Phi_p / (s-z_1)$ and $\Phi_p / (s-z_2)$. The main computation from \cite{Var-ar} can be rephrased as follows.

\begin{lem}[{\cite[Section 3]{Var-ar}}]\label{lem:appendix-arithmetic-flat-section}
    Let
    \begin{equation}
        I (z_1, z_2) = \begin{pmatrix}
            \Psi_{p, 1} \\ \Psi_{p, 2}
        \end{pmatrix} \in \mathbb{F}_p[z_1, z_2]^{\oplus 2}
    \end{equation}
    be a vector-valued function whose entries are given by $\Psi_{p, i}$'s of \eqref{eqn:solution-entries}. Then under the identification $z_2/z_1 = q$, the function
    \begin{equation}
        \widetilde{I}(q) = (1-q)^{2m} \begin{pmatrix} z_1^{-(p-m)} \Psi_{p,1} \\ z_1^{-(p-m)} \Psi_{p,2}  \end{pmatrix} \in \mathbb{F}_p[q]^{\oplus 2}
    \end{equation} 
    is only dependent on $z_2 /z_1 = q$ and solves the equivariant quantum differential equation (after a suitable change of basis, and the specialization $\teq = \mu \cdot t = m t$).
\end{lem}
\begin{proof}
    First, we rewrite the equivariant quantum differential equation
    \begin{equation}\label{eqn:QDE-usual-basis}
        \nabla_b^{S^1} I = 0 \quad \iff \quad 
    q \partial_q I = t^{-1} \begin{pmatrix}
        0 & \frac{\teq^2 q}{1-q} \\ 1 & \frac{-2\teq q}{1-q}
    \end{pmatrix} I
    \end{equation}
    by changing the geometric basis $1, b \in H^*_{S^1}(X) \otimes \mathrm{Frac}(H^*_{S^1}(\mathrm{pt})) \cong \mathbb{F}_p(\!(\teq)\!)[b]/(b^2)$ to the preferred basis $b-\teq, b \in H^*_{S^1}(X) \otimes \mathbb{F}_p (\!(\teq)\!)$ (see \eqref{eqn:appendix-stable-envelope}). In the new basis the differential equation has the form
    \begin{equation}\label{eqn:QDE-stable-basis-appendix}
    q \partial_q I = - t^{-1} \teq \left( \begin{pmatrix} 0 & 0 \\ 1& 0  \end{pmatrix} + \frac{q}{1-q} \begin{pmatrix}
    1 & 1 \\ 1 & 1    
    \end{pmatrix} \right) I.
\end{equation}
Note that the change of basis simplifies the connection matrix, in that the dependence on the equivariant parameters $\teq \in H^2(BS^1), t \in H^2(B\mathbb{Z}/p)$ becomes a single factor in $\teq / t$. Below we specialize to $\mu = \teq / t$, equivalently $\teq = \mu t = m t$.

Next we claim that $\widetilde{I}(q)$ is indeed a vector with entries in $\mathbb{F}_p[q]$, that it is only dependent on $z_2 / z_1 = q$. By binomial expansion we see (cf. \cite[Lemma 3.3]{Var-ar})
\begin{equation}
    \Psi_{p, 1} = \mathrm{Coeff} \left( s^{p-1} : s^{m} (s-z_1)^{p-m-1}(s-z_2)^{p-m} \right) = (-1)^{p-m} \sum_{ j+k = p-m} \binom{p-m-1}{j} \binom{p-m}{k} z_1^j z_2^k,
\end{equation}
hence
\begin{equation}
    z_1^{-(p-m)} \Psi_{p, 1} = (-1)^{p-m} \sum_{\ell=0}^{p-m} \binom{p-m-1}{p-m-k} \binom{p-m}{k} q^k \in \mathbb{F}_p[q]
\end{equation}
is indeed a well-defined polynomial in $q$. Similarly, $z_1^{-(p-m)} \Psi_{p, 2} \in \mathbb{F}_p[q]$.

To see that
\begin{equation}\label{eqn:gauge-transformed-soln}
    \widetilde{I}(q) = (1-q)^{2m} \begin{pmatrix} z_1^{-(p-m)} \Psi_{p,1} \\ z_1^{-(p-m)} \Psi_{p,2}  \end{pmatrix} 
\end{equation}
solves the equivariant quantum differential equation \eqref{eqn:QDE-stable-basis-appendix}, we refer to the fact that
\begin{itemize}
    \item (\cite[Theorem 3.1]{Var-ar}) $I(z_1, z_2) = (\Psi_{p, 1}, \Psi_{p,2})$ solves the mod $p$ dynamical differential equation of \cite{Var-ar};
    \item (\cref{lem:appendix-gauge-transform}) Setting $\widetilde{I}(z_1, z_2) = z_1^{-\mu}(z_1-z_2)^{2\mu}I$ yields a gauge transformation that identifies the dynamical differential equation with the quantum differential equation, so that $\widetilde{I}(z_1, z_2)$ is a solution to the quantum differential equation;
    \item By replacing $\mu$ with $m \equiv \mu$, we obtain the expression \eqref{eqn:gauge-transformed-soln}.
\end{itemize}
We refer to \cref{sec:appendix-gauge} for a more detailed explanation of the alluded gauge transformation.
\end{proof}

\section{Localization formula in low degrees of $\teq$}\label{sec:appendix-low-order-h}
Here we record the combinatorial manipulations to compute the terms in the localization formula (\cref{cor:S1-QSt-computation}) for the structure constants of the $S^1$-equivariant quantum Steenrod operations. Below, $d = \alpha p + \beta$ so that $\alpha$ and $\beta$ are quotient and remainder of $d$ modulo $p$, respectively. First let
    \begin{equation}
        C_{d, \ell} := C_{d, \ell}(x, t, \teq) = (x^{p-1} - t^{p-1})^{1 - 2 \alpha} \ c_{0, \infty}(x + \ell t, t) \ \prod_{k=1}^{d-1} (x  + (\ell -k)t - \teq)^2.
    \end{equation}

We denote the coefficient of $x^k$ in a power series $P(x) \in \mathbb{F}_p(\!(\teq, t, \theta)\!)[\![x]\!]$ by $(x^k: P(x)) \in \mathbb{F}_p(\!(\teq, t, \theta)\!)$. The localization formula is then written as
\begin{align}
    &(Q\Sigma_b^{S^1}(b_0), b_\infty)_{S^1} = (Q \Sigma_b(b_0), b_\infty)_{S^1} \notag \\
     &+  \teq \  q^d  \sum_{d > 0}  \left[ \sum_{\ell=0}^\beta  \left( x^{2\alpha}:   C_{d, \ell} \ { \prod_{j=1}^\ell (x+jt)^{-2} \prod_{j=1}^{\beta-\ell} (x-jt)^{-2}  } \right) \right. \left. + \sum_{\ell = \beta + 1}^{p-1} \left( x^{2\alpha - 2} : C_{d, \ell} \ {  \prod_{j=0}^\beta (x+(\ell-j)t)^{-2} } \right)  \right]  \notag,
\end{align}
which reduces to
    \begin{align}
        &(Q\Sigma_b^{S^1}(b_0), b_\infty)_{S^1} = (\mathrm{St}(b) \cup b_0, b_\infty)_{S^1}  \notag \\
         &+ \teq  \left( \sum_{d > 0, p \nmid d} \left[ -  \frac{t^{p-3}}{\beta^2}  c_{0, \infty}(\beta  t, t) \right] q^d + \sum_{d>0, p \mid d} \left[ -t^{p-1} (x^2:c_{0, \infty}(x, t)) \right] q^{d} \right)   \notag \\
         &+ \teq^2  \left( \sum_{d > 0, p \nmid d} \left[ \sum_{m=1, m \not \equiv \beta}^{d-1} \frac{ 2 t^{p-4}}{(\beta - m) \beta^2 } c_{0, \infty}(\beta t, t) - 2  \alpha t^{p-4} \left(x^1 : \left( \frac{t}{\beta^2} - \frac{2}{\beta^{3}}x \right)c_{0, \infty}(x+\beta t, t) \right) \right] q^d   \right. \notag \\
         &\quad \quad \quad  \left. \sum_{d >0, p | d} \left[ \sum_{m=1, m \not \equiv 0}^{d-1} \left( x^2 : -2 \ \frac{x^{p-1} - t^{p-1}}{x-mt} \ c_{0, \infty} (x, t) \right) -2 \alpha  \left( x^3 : (x^{p-1} - t^{p-1})\ c_{0, \infty} (x, t) \right) \right]  q^d \right) \notag \\
         &+ O(\teq^3). \notag
    \end{align}
for low degrees in $\teq$. Such results are used in the introduction in the paper to highlight the relevance of the higher terms in $\teq$, in contrast to formulas in \cite{BMO11}. We provide the computation of these lower order terms below. 

First, the lowest order ($\teq^0$) term is just the contribution from the non-equivariant quantum Steenrod operations. Since the Gromov--Witten invariants of $T^*\mathbb{P}^1$ are trivial, this reduces to the classical Steenrod operations:
\begin{equation}
    (Q\Sigma_b^{S^1}(b_0), b_\infty)_{S^1} = (Q \Sigma_b(b_0), b_\infty)_{S^1} = (\mathrm{St}(b) \cup b_0, b_\infty)_{S^1}.
\end{equation}

For the first order $\teq^1$ term, note there is a factor of $\teq$ for all non-classical ($q^{d>0}$) terms in the localization formula. Therefore, the first order term is computed by taking the $\teq$-constant term of $C_{d, \ell}$,
\begin{equation}
    C_{d, \ell}^0 := C_{d, \ell}(x, t, 0) =  (x^{p-1} - t^{p-1})^{1 - 2 \alpha} \ c_{0, \infty}(x + \ell t, t) \ \prod_{k=1}^{d-1} (x  + (\ell -k)t)^2,
\end{equation}
and computing
\begin{equation}\label{eqn:loc-formula-h1}
    \sum_{\ell=0}^\beta  \left( x^{2\alpha}:   C_{d, \ell}^0 \ { \prod_{j=1}^\ell (x+jt)^{-2} \prod_{j=1}^{\beta-\ell} (x-jt)^{-2}  } \right) + \sum_{\ell = \beta + 1}^{p-1} \left( x^{2\alpha - 2} : C_{d, \ell}^0 \ {  \prod_{j=0}^\beta (x+(\ell-j)t)^{-2} } \right).
\end{equation}
Observe that $C_{d, \ell}^0$ simplifies, using
\begin{equation}
    \prod_{k=0}^{p-1} ( x+ (\ell -k) t) = \prod_{k=0}^{p-1} (x - k  t) = x^{p} - t^{p-1}x;
\end{equation}
we have
\begin{align}
    C^0_{d, \ell } &= (x^{p-1} - t^{p-1})^{1 - 2 \alpha} \ c_{0, \infty}(x + \ell t, t) \ \prod_{k=1}^{d-1} (x  + (\ell -k)t)^2 \notag \\
    &= {\begin{cases} (x^{p-1} - t^{p-1})^{1 - 2\alpha}  \ c_{0, \infty}(x + \ell t, t) \ (x^p - t^{p-1}x)^{2\alpha} \prod_{j =1}^{\beta - 1} (x + (\ell -j )t)^2 &  \mbox{ if } \beta \neq 0 
    \\ (x^{p-1} - t^{p-1})^{1 - 2\alpha}  \ c_{0, \infty}(x + \ell t, t) \ (x^p - t^{p-1} x)^{2\alpha - 2} \prod_{j=1}^{p-1} (x + (\ell -j) t)^2 & \mbox { if } \beta = 0 \end{cases}} \notag \\
    &= { \begin{cases} x^{2\alpha} (x^{p-1} - t^{p-1}) \ c_{0, \infty}(x + \ell t, t) \ \prod_{j =1}^{\beta - 1} (x + (\ell -j )t)^2  & \mbox { if } \beta \neq 0 \\ x^{2\alpha - 2} (x^{p-1} - t^{p-1})^{-1}  \ c_{0, \infty}(x + \ell t, t) \ \prod_{j=1}^{p-1} (x + (\ell -j) t)^2 & \mbox { if } \beta = 0 \end{cases} }.
\end{align}

For $\beta \neq 0$, note that $C^{0}_{d, \ell}$ is divisible by $x^{2\alpha}$ for $\ell > \beta$ and by $x^{2\alpha +2}$ for $\ell < \beta$ respectively, and therefore only the $\ell = \beta$ term in the localization formula \eqref{eqn:loc-formula-h1} contributes. This term is given by
\begin{align}
    \left( x^{2\alpha}: \frac{C^{0}_{d, \beta}}{\prod_{j=1}^\beta (x+jt)^2} \right) &= \left(x^0 :   \frac{(x^{p-1} - t^{p-1}) \ c_{0, \infty}(x + \ell t, t) \prod_{j=1}^{\beta - 1} (x+ (\beta -j )t)^2}{ \prod_{j=1}^{\beta} (x+ j  t)^2 }   \right)  \\
    &= \frac{ - t^{p-1} \ c_{0, \infty} (\beta t, t) \cdot (\beta -1)!^2 \cdot t^{2\beta -2}}{ \beta!^2 \cdot t^{2\beta}} \notag \\
    &= \frac{ - t^{p-3} \ c_{0, \infty} (\beta t, t)}{\beta^2} \notag.
\end{align}

For $\beta = 0$, note that $C^0_{d, \ell}$ is divisible by $x^{2\alpha + 1}$ for $\ell > \beta$, and again only the $\ell = \beta$ term contributes. This term is given by
\begin{align}
    \left( x^{2 \alpha} :  C^0_{d, \beta} \right) &= \left( x^{2\alpha} :x^{2\alpha -2} \cdot \frac{ c_{0, \infty} (x + \beta t, t) (x^{p-1} -t^{p-1})^2   }{(x^{p-1} - t^{p-1}) } \right) \\
    &= \left( x^2 : -t^{p-1} \ c_{0, \infty} (x, t) \right) \notag.
\end{align}

This completes the computation for the $\teq^1$ term. We continue to compute the second order ($\teq^2$) term. We take the $\teq^1$-linear terms of $C_{d, \ell}$, given by
\begin{equation}
C_{d, \ell}^1 := \partial_h C_{d, \ell}(x, t, h)|_{h=0} = (x^{p-1} - t^{p-1})^{1 - 2 \alpha} \ c_{0, \infty}(x + \ell t, t) \sum_{d=1}^{m-1} \prod_{k=1}^{d-1} \frac{- 2 (x+ (\ell -k) t)^2}{(x + (\ell - m ) t)}.
\end{equation}

Similar simplifications as above show
\begin{align}
    C_{d, \ell}^1 = \begin{cases} -2 x^{2\alpha - 1} \ c_{0, \infty} ( x + \ell t, t) \ \prod_{j=1}^{\beta - 1} ( x+ (\ell -j ) t)^2 \sum_{m=1}^{d-1} \prod_{k=0, k \not \equiv m }^{p-1} (x + (\ell -k ) t) &  \mbox { if } \beta \neq 0 \\ -2 x^{2\alpha - 3} \ \frac{c_{0, \infty} ( x + \ell t, t)} { (x^{p-1} - t^{p-1})^{2}} \prod_{j=1}^{p - 1} ( x+ (\ell -j ) t)^2  \sum_{m=1}^{d-1}  \prod_{k=0, k \not \equiv m }^{p-1} (x + (\ell -k ) t)  & \mbox { if } \beta = 0 \end{cases}.
\end{align}

As before, only the $\ell = \beta$ terms contribute. For $\beta \neq 0$, this term is given by
\begin{align}
    & \left( x^{2\alpha} : \frac{C^1_{d, \beta}}{\prod_{j=1}^\beta ( x+ jt)^2} \right) \\
    &= \left( x^1 : \frac{-2 \  c_{0, \infty} ( x+ \beta t, t) \ \sum_{m=1}^{d-1} \prod_{k=0, k \not \equiv m}^{p-1} (x + (\beta - k) t)  }{(x+\beta t)^2} \right) \notag \\
    &=  \sum_{m=1, m \not \equiv \beta}^{d-1} \left( x^1 : \frac{- 2 (x^p - t^{p-1}x) \ c_{0, \infty}(x + \beta t, t) }{(x + (\beta - m)t) (x+ \beta t)^2 } \right)  \notag \\
    & \quad \quad + \sum_{m=1, m \equiv \beta}^{d-1}  \left(x^1 : \frac{-2 (x^{p-1} - t^{p-1}) \ c_{0, \infty}(x+\beta t, t) } {(x+\beta t)^2} \right) \notag \\
    &=  \sum_{m=1, m \not \equiv \beta}^{d-1} \frac{ 2 t^{p-4}}{(\beta - m) \beta^2 } c_{0, \infty}(\beta t, t) - 2  \alpha t^{p-4} \left(x^1 : \left( \frac{t}{\beta^2} - \frac{2}{\beta^{3}}x \right)c_{0, \infty}(x+\beta t, t) \right)  \notag .
\end{align}
For $\beta = 0$, this term is given by
\begin{align}
    &\left( x^{2\alpha} : C^1_{d, \beta}  \right) \\
    &= \left( x^3 :  -2 \ c_{0, \infty} (x + \beta t, t) \sum_{m=1}^{d-1} \prod_{k=0, k \not\equiv m}^{p-1} (x + (\beta - k) t)  \right) \notag \\
    &= \sum_{m=1, m \not \equiv 0}^{d-1} \left( x^2 : -2 \ \frac{x^{p-1} - t^{p-1}}{x-mt} \ c_{0, \infty} (x, t) \right) -2 \alpha  \left( x^3 : (x^{p-1} - t^{p-1})\ c_{0, \infty} (x, t) \right) \notag.
\end{align}

\bibliographystyle{amsalpha}
\bibliography{cov}

\end{document}